\documentstyle[amssymb,amsfonts]{amsart}

\newcommand{\der}{\nabla}
\newcommand{\cder}{\D^{(A)}}
\newcommand{\rder}{\mbox{$\nabla \mkern-13mu /$\,}}

\newcommand{\rcder}{\mbox{$\D \mkern-13mu /$\,}^{(A)}}
\newcommand{\les}{\lesssim}
\newcommand{\bea}{\begin{eqnarray}}
\newcommand{\eea}{\end{eqnarray}}

\newenvironment{proof}{\noindent {\bf Proof} }{\endprf\par}

\def \endprf{\hfill  {\vrule height6pt width6pt depth0pt}\medskip}
\def\beaa{\begin{eqnarray*}}
\def\eeaa{\end{eqnarray*}}
\def\pa{\partial}
\def\a{{\alpha}}
\def\b{{\beta}}
\def\ga{\gamma}
\def\Ga{\Gamma}
\def\de{\delta}

\def\eps{\epsilon}
\def\la{\lambda}
\def\La{\Lambda}
\def\si{\sigma}
\def\Si{\Sigma}
\def\om{\omega}
\def\Om{\Omega}

\def\th{\theta}
\def\ze{\zeta}

\def\Lb{{\underline{L}}}
\def\D{{\bf D}}
\def\J{{\bf J}}

\def\T{{\bf T}}
\def\g{{\bf g}}
\def\pr{\partial}
\def\lap{\Delta}
\def\SSS{{\Bbb S}}

\begin{document}
\theoremstyle{plain}
  \newtheorem{theorem}{Theorem}
  \newtheorem{proposition}[subsection]{Proposition}
  \newtheorem{lemma}[subsection]{Lemma}

\theoremstyle{remark}
  \newtheorem{remark}[subsection]{Remark}
  \newtheorem{remarks}[subsection]{Remarks}

\theoremstyle{definition}
  \newtheorem{definition}[subsection]{Definition}

\include{psfig}
\title[Yang-Mills on Curved Spacetimes]{The Global Existence of Yang-Mills Fields \\ on Curved Space-Times}
\author{Sari Ghanem}
\address{Institut de Math\'ematiques de Jussieu, Universit\'e Paris Diderot - Paris VII, 75205 Paris Cedex 13, France}
\email{ ghanem@@math.jussieu.fr}
\maketitle

\begin{abstract} This is an introductory chapter in a series in which we take a systematic study of the Yang-Mills equations on curved space-times. In this first, we provide standard material that consists in writing the proof of the global existence of Yang-Mills fields on arbitrary curved space-times using the Klainerman-Rodnianski parametrix combined with suitable Gr\"onwall type inequalities. While the Chru\'sciel-Shatah argument requires a simultaneous control of the $L^{\infty}_{loc}$ and the $H^{2}_{loc}$ norms of the Yang-Mills curvature, we can get away by controlling only the $H^{1}_{loc}$ norm instead, and write a new gauge independent proof on arbitrary, fixed, sufficiently smooth, globally hyperbolic, curved 4-dimensional Lorentzian manifolds. This manuscript is written in an expository way in order to provide notes to Master's level students willing to learn mathematical General Relativity. 
\end{abstract}

\setcounter{page}{1}
\pagenumbering{arabic}

\section{Introduction}

Gauge field theories, such as the Maxwell equations and the Yang-Mills equations, arise in important physical theories to describe electromagnetism and the weak and strong interactions, and are to some extent mathematically related to the Einstein vacuum equations in General Relativity. Indeed, using Cartan formalism the Einstein vacuum equations can be written as the Yang-Mills equations except to the fact that the background geometry is part of the unknown solution of the evolution problem in General Relativity, while in Yang-Mills theory one can fix the background to be a given space-time.

In a classical paper, [EM1]-[EM2], Eardley and Moncrief proved global existence of solutions of the Yang-Mills equations in the 4-dimensional Minkowski background. This is an introductory chapter in a series in which we aim to extend their global regularity result to curved backgrounds. In this first, we write the proof of the global existence of Yang-Mills fields on arbitrary fixed curved space-time.

The Eardley-Moncrief result made a use of a hyperbolic formulation of the problem. Indeed, while the Yang-Mills equations say that the Yang-Mills curvature is divergence free on the background geometry, one can obtain a hyperbolic formulation by taking the covariant divergence of the Bianchi identity. This leads to a tensorial covariant wave equation on the Yang-Mills curvature with a non-linear term. It is exactly the study of this non-linear term that permits one to answer the question of local well-posdness, and global well-posdness of the equations. In this formulation, the initial data consists of the Yang-Mills potential, that is a one form valued in the Lie algebra, and the electric field (loosely speaking the time derivative of the potential) on a given spacelike Cauchy hypersurface $\Sigma$. The initial data set has to verify itself the Yang-Mills equations, that is the covariant divergence of the electric field vanishes. One looks for a Yang-Mills curvature that satisfies the Yang-Mills equations such that once restricted on this hypersurface $\Sigma$ the Yang-Mills curvature corresponds to that given by the prescribed potential and electric field.

Eardley and Moncrief proved global existence of solutions of the Yang-Mills equations in the 4-dimensional Minkowski background by proving a local existence result and providing pointwise estimates on the curvature, [EM1]-[EM2]. Their approach depended on the fundamental solution of the wave equation on flat space-time, and the use of the Cronstr\"om gauge condition, that has the remarkable advantage of expressing the potential as a function of the curvature directly in terms of an integral, to estimate the non-linear term. Later on, this result was extended by Chru\'sciel and Shatah, [CS], to curved space-times using the same approach, by making use of the Friedlander parametrix for the wave equation in causal domains in curved space-times, [Fried], and the Cronstr\"om gauge condition as well. In a recent paper, [KR1], Klainerman and Rodnianski constructed a parametrix for the wave equation which permitted them to give a new gauge independent proof of the Eardley-Moncrief result [EM2] in a Minkowski background.

The Klainerman-Rodnianski's approach relies on their derivation of a covariant representation formula for the wave equation on arbitrary, smooth, globally hyperbolic, curved space-times, in which the integral terms are supported on the past null cone. As the authors pointed out, their parametrix can be immediately adapted to gauge covariant derivatives; this is because the scalar product on the Lie algebra $<\; ,\;>$ is Ad-invariant. They used it to give a new gauge independent proof of the Eardley-Moncrief result [EM2], of which the only ingredient is the conservation of the energy. As the authors mentioned, one can generalize their proof of the global existence of Yang-Mills fields on the flat Minkowski space-time to arbitrary smooth, globally hyperbolic, curved space-times under the assumption that there exists a timelike verctor field $\frac{\pa}{\pa t}$ of which the deformation tensor is finite, as it has been assumed in previous work by Chru\'sciel and Shatah, [CS].

In this manuscript, we provide standard material, but not so clearly pointed out in literature, that consists in writing the proof of the global existence of Yang-Mills fields on arbitrary curved space-times by using the Klainerman-Rodnianski parametrix combined with suitable Gr\"onwall type inequalities. While the Chru\'sciel-Shatah argument requires a simultaneous control of the $L^{\infty}_{loc}$ and the $H^{2}_{loc}$ norms of the Yang-Mills curvature, we can get away by controlling only the $H^{1}_{loc}$ norm instead. However, we were unable to get rid of any control on the gradient of the Yang-Mills curvature, as it is the case in the proof on Minkowski space-time in [KR1]. Hence, this provides a new gauge independent proof and improves the Chru\'sciel-Shatah's result, [CS], for sufficiently smooth, globally hyperbolic, curved 4-dimensional Lorentzian manifolds.\\

\subsection{The statement}\

More precisely, we will prove the following theorem,\\

\begin{theorem}  \label{theoremglocalexstenceYang-Mills}
Let $(M, \g)$ be a curved 4-dimensional Lorentzian manifold. We know by then that at each point $p \in M$, there exists a frame $\{ \hat{t}, n, e_{a}, e_{b} \}$ where,
\beaa
\g &=& - d\hat{t}^{2} + dn^{2} + de_{a}^{2} + de_{b}^{2} 
\eeaa

We assume that $\g$ is sufficiently smooth, $M$ is globally hyperbolic, and that there exists a timelike vector field $\frac{\pa}{\pa t}$ and $C(t) \in L_{loc}^{1}$, such that for all $\hat{\mu}, \hat{\nu} \in \{\hat{t}, n, e_{a}, e_{b} \}$, the components of the deformation tensor $\pi^{\hat{\mu}\hat{\nu}}( \frac{\pa}{\pa t})  = \frac{1}{2} [ \der^{\hat{\mu}}  (\frac{\pa}{\pa t})^{\hat{\nu}}+  \der^{\hat{\nu}}  (\frac{\pa}{\pa t})^{\hat{\mu}} ] $ verify,
\beaa
| \pi^{\hat{\mu}\hat{\nu}}( \frac{\pa}{\pa t}) |_{L^{\infty}_{loc(\Sigma_{t})}} \leq C(t) 
\eeaa

where $\Sigma_{t}$ are the $t= constant$ hypersurfaces, and coordinate $t$ could be defined only locally. Let, $\Sigma_{t=t_{0}} $ be a Cauchy hypersurface prescribed by $t=t_{0}$. Let $F_{\hat{\mu}\hat{\nu}}$ be the components of the Yang-Mills fields in the frame $\{\hat{t}, n, e_{a}, e_{b} \}$, defined as the anti-symmetric 2-tensor solution of the Cauchy problem of the Yang-Mills equations $\cder_{\a}F^{\a\b} = 0$, where the initial data prescribed on the Cauchy hypersurface $\Sigma_{t=0}$ verifies the Yang-Mills constraint equations, $${\cder}^{\b} F_{\hat{t}\b} (t= 0) = 0 $$

Then, we have that local solutions to the Yang-Mills equations can be extended globally in $t$ if,
\beaa
E_{F}^{\frac{\pr}{\pr t}} (t=t_{0}) < \infty
\eeaa
and,
\beaa
E_{ \cder F}^{\frac{\pr}{\pr t}} (t=t_{0}) < \infty
\eeaa
where,
\beaa
E_{F}^{\frac{\pr}{\pr t}} (t=0) &=&  \int_{q \in \Sigma_{t=0} }  \frac{1}{2} [  | F_{\hat{t}n}|^{2} + |F_{\hat{t}a}|^{2} +  | F_{\hat{t}b}|^{2} +   |F_{na}|^{2} +   |F_{nb}|^{2} + | F_{ab}|^{2} ] (q) \\
&& \quad \quad \quad . \sqrt{ - \g(\frac{\pr}{\pr t },\frac{\pr}{\pr t }) } dV_{\Sigma}(q)
\eeaa
and,
\beaa
E_{\cder F}^{\frac{\pr}{\pr t}} (t=t_{0}) &=&  \int_{q \in \Sigma_{t=0} }  \frac{1}{2} [  | \cder F_{\hat{t}n}|^{2} + | \cder F_{\hat{t}a}|^{2} +  | \cder F_{\hat{t}b}|^{2} +   | \cder F_{na}|^{2}  \\
&&\quad \quad \quad +   | \cder F_{nb}|^{2} + | \cder F_{ab}|^{2} ] (q)  . \sqrt{ - \g(\frac{\pr}{\pr t },\frac{\pr}{\pr t }) } dV_{\Sigma} (q)
\eeaa
where,
\beaa
|\cder F_{\hat{\mu}\hat{\nu}}|^{2} &=&  |\cder_{\hat{t}}F_{\hat{\mu}\hat{\nu}}|^{2} + |\cder_{n}F_{\hat{\mu}\hat{\nu}}|^{2}  + |\cder_{e_{a}} F_{\hat{\mu}\hat{\nu}}|^{2}  + |\cder_{e_{b}} F_{\hat{\mu}\hat{\nu}}|^{2} \\
\eeaa

\end{theorem}

\subsection{Strategy of the proof}\

As we will show, see \eqref{hyperbolic}, the Yang-Mills fields satisfy a non-linear hyperbolic differential equation on the background geometry. Since the scalar product on the Lie algebra $<\; , \; >$ is Ad-invariant, the Klainerman-Rodnianski parametrix can be immediately generalized (see Appendix) to gauge covariant derivatives to give a representation formula for solutions of $(\Box^{(A)}_{\g} F )_{\mu\nu} = S_{\mu\nu}$, where $S_{\mu\nu}$ is a source tensor, and hence it can be used for the Yang-Mills fields, see \eqref{KSparametrixYMsetting}.

We would like to bound all the terms in the representation formula in a way that we could use Gr\"onwall lemma to deduce that the $L^{\infty}$ norm of $F$ will stay finite (see \eqref{termstoapplyGronwall}). For this we need a parameter in which the extension of local solutions can make sense; this would be a timelike vector field $\frac{\pa}{\pa t}$.

The main advantage of the parametrix is that all it's integral terms are supported on the past null cone. Naively, one can hope that those can be bounded by the flux of the energy generated from $\frac{\pa}{\pa t}$. Thus, if one can bound the energy flux along the null cones, the proof might go through. To bound the energy flux, one needs, as in [CS], to assume that the deformation tensor of a timelike vector field has it's integral in $t$ finite on bounded domains:
\beaa
| \pi^{\hat{\mu}\hat{\nu}}( \frac{\pa}{\pa t}) |_{L^{\infty}_{loc(\Sigma_{t})}} \leq C(t)  \in L_{loc}^{1}
\eeaa
 Using the divergence theorem on the energy-momentum tensor contracted with $\frac{\pa}{\pa t}$ will lead to an inequality on the energy. The assumption on the deformation tensor above can show using Gr\"onwall lemma that the local energy will stay finite, see \eqref{localenergywillstayfinite}. Using this and the assumption on the deformation tensor again, one can show that the space-time integral generated from the divergence theorem will stay finite in $t$, see \eqref{boundingthespacetimeintegralfromthedivergencetheorem}, from which one can deduce the finiteness of the flux \eqref{finitenessflux}.

The integral terms supported on the past null cone in the Klainerman-Rodnianski parametrix involve a term that is a generalization of the fundamental solution of the wave equation on flat space-time to curved space-times. This is $\la_{\mu\nu}$ that is a two tensor solution of a transport equation along the null cone given by \eqref{eq:transport} and \eqref{eq:initial condition}. Using the transport equation, one can prove that the $L^{\infty}$ norm of $s\la_{\mu\nu}$, where $s$ is the geodesic parameter for a null vector field $L$ normal to the null cone used to define the transport equation for $\la_{\mu\nu}$, will be bounded by the initial data for $\la_{\mu\nu}$, see \eqref{linfinitynormofslamda}. Yet, since we would want to apply Gr\"onwall lemma, the terms which contain $\la_{\mu\nu}$ and $F_{\mu\nu}$ can be bounded as in \eqref{controllingtermslamdaF}, by controlling $s\la_{\mu\nu}$, see \eqref{linfinitynormofslamda}, and $s^{-1} F_{\mu\nu}$, see \eqref{controllingsF}. The terms which contain $\la_{\mu\nu}$ and $[F, F ]$ can be bounded as in  \eqref{ControllinglamdabracketFF} by using the finiteness of the energy flux, see \eqref{Controllingsminus1Fusingfinitnessofflux}.

However, a major difference with the situation on Minkowski space, is in the way to deal with the term which contains $\hat{\lap}^{(A)}\la_{\a\b}$ and $F$, where $\hat{\lap}^{(A)} \la_{\a\b} $ is the induced Laplacian on the 2-sphere prescribed by $s = constant$ defined by \eqref{laplacianonab}. In Minkowski space, one can control directly $\hat{\lap}^{(A)}\la_{\a\b}$ as shown by Rodnianski and Klainerman in [KR1], because one can close a system of transport equations along the null cone. On curved space-times, we are unable to close such a system, consequently, we will use the divergence theorem on $\SSS^{2}$, see \eqref{integrationbypartsons2}, so as to bring the problem to controlling $\cder_{a} \la$ and $\cder_{a} F$, where these are the derivatives tangential to the 2-sphere prescribed by $s = constant$.

To control $\cder_{a} \la$ we will follow [KR3], see \eqref{controlofthetangderivativeoflamdaasinKR3}. Since the area element on the 2-spheres is at the level of $s^{2}$, see \eqref{areaexpression}, one would want to control the $L^{2}$ norm of $s \cder_{a} \la$ on the null cone, with respect to the measure $ds d\sigma^{2}$, where $d\sigma^{2}$ is the usual volume form on $\SSS^{2}$. One could try to use the fundamental theorem of calculus directly to control the $L^{2}$ norm on $\SSS^{2}$ then integrate in $s$, yet by doing so, we would find ourselves confronted to controlling near the vertex $p$ ($s=0$) a quantity of the type $(\frac{1}{s} - tr\chi)$, where $\chi$ is the null second fundamental form of the null hypersurfaces. This quantity cannot be controlled even in the 4-dimensional Minkowski space where $tr\chi = \frac{2}{s}$. To change the factor in front of $\frac{1}{s}$ from $1$ to $2$, one would need to apply the fundamental theorem of calculus to control the $L^{2}$ norm on the null cone of $s^{2} \cder_{a} \la$ instead of $s \cder_{a} \la$, see \eqref{derivativeofthehnormsquaredofs2lamda}. Since it is the $L^{2}$ norm, this means that one would have to consider $s^{4} |\cder_{a} \la|^{2}$ instead of $s^{2} |\cder_{a} \la|^{2}$ for applying the fundamental theorem of calculus. However, since what we want to control is the integral on the null cone of $s^{2} |\cder_{a} \la|^{2}$, which is bigger than that of $s^{4} |\cder_{a} \la|^{2}$, near $s=0$, we would need to lower the power on $s$, for this one can actually control the integral on $\SSS^{2}$ of $s^{-1} s^{4} |\cder_{a} \la|^{2}$ by applying the fundamental theorem of calculus to $s^{4} |\cder_{a} \la|^{2}$ as described above, see \eqref{sminus1fundamentaltheoremcalculussfourdlamdasquared} and \eqref{controlontheLtwonormonStwoofsthreedlamdasquared}. This would allow then to control the integral on $\SSS^{2}$ of $ s^{3} |\cder_{a} \la|^{2}$ in a way that one could then get an estimate on  the $L^{1}$ norm on $\SSS^{2}$ for $s^{2} |\cder_{a} \la|^{2}$, see \eqref{estimateontheLonenormofstwolamdasquaredtouseLtwomaximumprinciple}, which permits one to apply the $L^{2}$ maximum principle to control the integral on the null cone of $s^{2} |\cder_{a} \la|^{2}$ near the vertex $p$ ($s=0$), see \eqref{estimateonLtwonormonthenullconeofserivativelamda}. Away from the vertex $s=0$ the integral is clearly finite and hence, this would give the desired control.

In order to control the $L^{2}$ norm of $\cder_{a} F$ on the null cone, we will use the energy momentum tensor of the wave equation $T_{1}$ after contracting the free indices of the Yang-Mills fields with respect to a Riemannian metric $h$, as in [CS], see \eqref{energy-momuntumtensorwaveequationafterfullcontractionwithrespecttoh}. Since it is a full contraction, we can compute it by choosing a normal frame , i.e. a frame where the Christoffel symbols vanish at that point, and hence we can get the derivatives inside the scalar product as covariant derivatives (and also as gauge covariant derivatives using the fact that the scalar product is Ad-invariant) instead of partial derivatives. Since it is the energy momentum tensor for the wave equation, the boundary term supported on the null cone obtained after contracting $T_{1}$ with the normalized timelike vector field, $\frac{\pa}{\pa \hat{t}}$, and applying the divergence theorem in a region inside the null cone, is at the level of the $L^{2}$ norm of $\cder_{a} F$ and  $\cder_{L} F$, see \eqref{T1alphabetathatL}, and thus it controls the $L^{2}$ norm of $\cder_{a} F$. We know by then, from the divergence theorem, that this can be controlled by a quantity that is at the level of a space integral of $T_{1}^{\hat{t}\hat{t}}$ on the initial spacelike hypersurface and in addition a spacetime integral of  $|\cder F| ( |\cder F| + |F| + |F|^{2} )$, see \eqref{controlafterdivergencetheoremappliedwithT1}, where $|F|$ and $|F|^{2}$ arise from the sources of the tensorial gauge hyperbolic wave equation verified by $F$, and $|\cder F|$ in the parenthesis is due to the fact that the deformation tensor of $\frac{\pa}{\pa \hat{t}}$, as well as the covariant derivative of $h$, do not vanish. As we wish to get rid of the gradient of $F$, so as to have a control that involves an integral or a double integral of the square of the $L^{\infty}$ norm of $F$, see \eqref{cotrolontheL2normofcderaFonthenullcone}, we recall that the divergence theorem that we applied previously also permits one to control the space integral of $T_{1}^{\hat{t}\hat{t}}$ on the spacelike hypersurface, that is at the level of the $L^{2}$ norm of the $\cder F$, by the same quantity that controls the boundary term on the null cone, \eqref{controlafterdivergencetheoremappliedwithT1}. This allows one to use Gr\"onwall lemma, after using $a.b \les a^{2} + b^{2}$, and the conservation of the energy that is at the level of the $L^{2}$ norm of $F$, to control the $L^{2}$ norm of $\cder F$ on the spacelike hypersurfaces by the desired quantity, see \eqref{inequalitytocontrolgradientofF}. Injecting this in the previous control on the $L^{2}$ norm of $\cder_{a} F$, \eqref{controlafterdivergencetheoremappliedwithT1}, and using again $a.b \les a^{2} + b^{2}$ and the conservation of the energy, leads to the desired control \eqref{cotrolontheL2normofcderaFonthenullcone}.

Now, the parametrix \eqref{KSparametrixYMsetting} permits us to control the value of the Yang-Mills fields contracted with an arbitrary tensor, at a point $q$ in space-time, by the estimates mentioned above. We would want to establish a Gr\"onwall type inequality in $t$ on the $L^{\infty}$ norm of $F$ on $\Sigma_{t}^{p}$, the spacelike hypersurfaces prescribed by $t = constant$ in the past of a point $p$, so as to deduce the finiteness of the fields at the point $p$. To obtain this, we take the supremum on $q \in \Sigma_{t}^{p}$ in the inequality described above, i.e. after using the parametrix and the above estimates, see \eqref{termstoapplyGronwall}. This can be used to show that $||F||_{L^{\infty}(\Sigma_{t}^{p})}$ verifies a generalized Gr\"onwall type inequality \eqref{Pachpatte} to which Pachpatte in [Pach], proved a result that ensures that the solutions will stay finite. A local existence result would give that solutions of the Yang-Mills equations will either blow up in finite time, or they will be defined globally in time. Hence, that the non-blow up result that we have established gives that local solutions of the Yang-Mills equations can be extended globally in time $t$, under the assumptions of theorem \eqref{theoremglocalexstenceYang-Mills}.

\begin{remark}
The whole manuscript is written in an expository way, where we detail all the calculations, and we show standard material to make these notes self-contained. We also detail well known material in the Appendix.
\end{remark}

\textbf{Acknowledgments.} The author would like to thank his PhD thesis advisors, Fr\'ed\'eric H\'elein and Vincent Moncrief, for their advice and support, and Sergiu Klainerman for suggesting the problem as a first stage in a research proposal for the author's doctoral dissertation. This work was supported by a full tuition fellowship from Universit\'e Paris VII - Institut de Math\'ematiques de Jussieu, and from the Mathematics Department funds of Yale University. The author would like to thank the Mathematics Department of Yale University for their  kindness and hospitality while completing this work. We also thank Arick Shao for looking at the Appendix and for making remarks about it. The manuscript was edited by the author while receiving financial support from the Albert Einstein Institute, Max-Planck Institute for Gravitational Physics, and we would like to thank them for their kind invitation and hospitality, and for their interest in our work.\\

\section{The Field Equations}

In this section we present the Yang-Mills curvature, and we derive the Yang-Mills equations from the Yang-Mills Lagrangian. We will also show the Bianchi identities.

\subsection{The Yang-Mills curvature}

Let $(M, \g)$ be a four dimensional globally hyperbolic Lorentzian manifold. Let G be a compact Lie group, and ${\cal G}$ its Lie algebra such that it has a faithful real matrix representation \{$\th_{a}$\}. Let $<$ $,$ $>$ be a positive definite Ad-invariant scalar product on ${\cal G}$. The Yang Mills potential can be regarded locally as a ${\cal G}$-valued one form $A$ on $M$, say $$A = A^{(a)}_{\alpha}\th_{a}dx^{\alpha} = A_{\alpha}dx^{\alpha}$$ in a given system of coordinates. The gauge covariant derivative of a ${\cal G}$-valued tensor $\Psi$ is defined as
\bea
\textbf{D}^{(A)}_{\alpha}\Psi = \der_{\alpha}\Psi + [A_{\alpha},\Psi]
\eea
where $\der_{\alpha}$ is the space-time covariant derivative of Levi-Cevita on $(M,\g)$, and $\der_{\alpha}\Psi$ is the tensorial covariant derivative of $\Psi$, that is
\bea
\notag
(\der_{\alpha}\Psi)(X, Y, Z, \ldots) &=& \pa_{\alpha} (\Psi(X, Y, Z, \ldots)) - \Psi(\der_{\alpha}X, Y, Z, \ldots) \\
\notag
&&- \Psi (X, \der_{\alpha}Y, Z, \ldots) - \Psi (X, Y, \der_{\alpha}Z, \ldots) \\
&& - ............... - ...............
\eea
The tensorial second order derivative is defined as
\bea
\notag
(\der_{\beta} \der_{\alpha}\Psi)(X, Y, Z, \ldots) &=&\pa_{\beta} [(\der_{\alpha}\Psi)(X, Y, Z, \ldots)] - (\der_{\der_{\beta} e_{\alpha}} \Psi)(X, Y, Z, \ldots) \\
\notag
&&- (\der_{\alpha} \Psi )(\der_{\beta}X, Y, Z, \ldots)  - (\der_{\alpha} \Psi)  (X, \der_{\beta} Y, Z, \ldots) \\
&& - ............... - ...............
\eea
By letting 
\beaa
\notag
(\der_{\beta} (\der_{\alpha}\Psi)) (X, Y, Z, \ldots) &=& \pa_{\beta} [(\der_{\alpha}\Psi)(X, Y, Z, \ldots)] - (\der_{\alpha} \Psi )(\der_{\beta}X, Y, Z, \ldots)  \\
&& - (\der_{\alpha} \Psi)  (X, \der_{\beta} Y, Z, \ldots)  - ............... 
\eeaa
We can then write
\bea
\notag
(\der_{\beta} \der_{\alpha}\Psi)(X, Y, Z, \ldots) &=& (\der_{\beta} (\der_{\alpha}\Psi))(X, Y, Z, \ldots) - (\der_{\der_{\beta} e_{\alpha}} \Psi)(X, Y, Z, \ldots) \\
\eea
The Yang-Mills curvature is a ${\cal G}$-valued two form $$F = F^{(a)}_{\alpha\beta}\th_{a}dx^{\alpha} \wedge dx^{\beta} = F_{\alpha\beta}  dx^{\alpha} \wedge dx^{\beta}$$ obtained by commutating in a system of coordinates two gauge covariant derivatives of a ${\cal G}$-valued tensor $\Psi$, where the tensorial second order gauge derivative of $\Psi$ is defined by
\bea
\textbf{D}^{(A)}_{\a}\textbf{D}^{(A)}_{\b}\Psi = \D^{2}_{\a\b}\Psi = \cder_{\a}(\cder_{\b}\Psi) - \cder_{\der_{\a}e_{\b}}\Psi
\eea
\begin{eqnarray*}
\textbf{D}^{(A)}_{\a}(\textbf{D}^{(A)}_{\b}\Psi) &=& \der_{\a}(\textbf{D}^{(A)}_{\b}\Psi) + [A_{\a},\textbf{D}^{(A)}_{\b}\Psi] \\ 
&=& \der_{\a} ( \der_{\b}\Psi + [A_{\b},\Psi]) + [A_{\a},\der_{\b}\Psi + [A_{\b},\Psi]] \\
& =& \der_{\a}(\der_{\b}\Psi) + [\der_{\a}A_{\b},\Psi] + [A_{\b},\der_{\a}\Psi] + [A_{\a},\der_{\b}\Psi] + [A_{\a},[A_{\b},\Psi]]\\
\end{eqnarray*}
As is a system of coordinates $[e_{\a}, e_{\b}] = 0 = \der_{\a}e_{\b} - \der_{\b}e_{\a}$ (the metric is assumed to be torsion free), then
\bea
\notag
&& \textbf{D}^{(A)}_{\a} \textbf{D}^{(A)}_{\b}\Psi - \textbf{D}^{(A)}_{\b} \textbf{D}^{(A)}_{\a}\Psi \\
\notag
&=& \textbf{D}^{(A)}_{\a} (\textbf{D}^{(A)}_{\b}\Psi) - \textbf{D}^{(A)}_{\b}( \textbf{D}^{(A)}_{\a}\Psi) - (\cder_{\der_{\a}e_{\b}}\Psi) + (\cder_{\der_{\a}e_{\b}}\Psi) \\
\notag
 &=& \textbf{D}^{(A)}_{\a} (\textbf{D}^{(A)}_{\b}\Psi) - \textbf{D}^{(A)}_{\b}( \textbf{D}^{(A)}_{\a}\Psi) + (\cder_{(\der_{\a}e_{\b} - \der_{\b}e_{\a})}\Psi) \\
\notag
&=& \textbf{D}^{(A)}_{\a} (\textbf{D}^{(A)}_{\b}\Psi) - \textbf{D}^{(A)}_{\b}( \textbf{D}^{(A)}_{\a}\Psi) + 0 \\
\notag
&=& \der_{\a}(\der_{\b}\Psi) + [\der_{\a}A_{\b},\Psi] + [A_{\b},\der_{\a}\Psi] + [A_{\a},\der_{\b}\Psi] + [A_{\a},[A_{\b},\Psi]] \\
\notag
&&     - \der_{\b}(\der_{\a}\Psi) - [\der_{\b}A_{\a},\Psi] - [A_{\a},\der_{\b}\Psi] - [A_{\b},\der_{\a}\Psi] - [A_{\b},[A_{\a},\Psi]] \\
\notag
&=&\sum_{i} {{R_{a_{i}}}^{\ga}}_{\a\b} \Psi_{....\ga....} + [\der_{\a}A_{\b},\Psi] + [A_{\a},[A_{\b},\Psi]] - [\der_{\b}A_{\a},\Psi] - [A_{\a},\der_{\b}\Psi] \\
&& - [A_{\b},[A_{\a},\Psi]] \\
\notag
&=&\sum_{i} {{R_{a_{i}}}^{\ga}}_{\a\b} \Psi_{....\ga....} + [\der_{\a}A_{\b} - \der_{\b}A_{\a} + [A_{\a},A_{\b}],\Psi] \\
&=& \sum_{i} {{R_{a_{i}}}^{\ga}}_{\a\b} \Psi_{....\ga....} + [F_{\a\b},\Psi] 
\eea
 where $\Psi = \Psi_{a_{1}a_{2}.....a_{i}.....}$, and $\ga$ is at the $i^{th}$ place.\
This gives
\bea
F_{\a\b} = \der_{\a}A_{\b} - \der_{\b}A_{\a} + [A_{\a},A_{\b}]
\eea

\subsection{The Yang-Mills equations}

The Yang-Mills Lagrangian is given by $$L = -\frac{1}{4}<F_{\a\b},F^{\a\b}>$$

A compact variation $(F(s),U)$, where $U$ is any compact set of $M$, can be written in terms of a compact variation $(A(s),U)$ of a gauge potential in the following manner:

$$\dot{F}_{\a\b} = \frac{d}{ds}F_{\a\b}(s)|_{s=0} = \der_{\a}\dot{A}_{\b} - \der_{\b}\dot{A}_{\a} + [\dot{A}_{\a},A_{\b}] + [A_{\a},\dot{A}_{\b}] $$

where \[\dot{A} = \frac{d}{ds}A(s)|_{s=0}\]

The action principle gives
\begin{eqnarray*}
 \frac{d}{ds}L(s)|_{s=0} &=& -\frac{1}{2}<\dot{F}_{\a\b},F^{\a\b}>_{\g}dv_{\g} = 0 \\
&=& -\frac{1}{2}\int_{U}<\der_{\a}\dot{A}_{\b} - \der_{\b}\dot{A}_{\a} + [\dot{A}_{\a},A_{\b}] + [A_{\a},\dot{A}_{\b}], F^{\a\b}>_{\g}dv_{\g} \\
&=& -\frac{1}{2}\int_{U}<\der_{\a}\dot{A}_{\b}, F^{\a\b}>_{\g}dv_{\g} + \frac{1}{2}\int_{U}< \der_{\b}\dot{A}_{\a}, F^{\a\b}>_{\g}dv_{\g} \\
&& -\frac{1}{2}\int_{U}<[\dot{A}_{\a},A_{\b}], F^{\a\b}>_{\g}dv_{\g} -\frac{1}{2}\int_{U}< [A_{\a},\dot{A}_{\b}], F^{\a\b}>_{\g}dv_{\g}\\
&=& -\int_{U}<\der_{\a}\dot{A}_{\b}, F^{\a\b}>_{\g}dv_{\g} - \int_{U}<[\dot{A}_{\a},A_{\b}], F^{\a\b}>_{\g}dv_{\g}
\end{eqnarray*}
(where we have used the anti-symmetry of $F$)

$$= -\int_{U}<\dot{A}_{\b}, \der_{\a}F^{\a\b}>_{\g}dv_{\g} - \int_{U}<[\dot{A}_{\a},A_{\b}], F^{\a\b}>_{\g}dv_{\g}$$

(where we have integrated by parts, and the boundary terms are zero since F has compact support)

On the other hand $$- \int_{U}<[\dot{A}_{\a},A_{\b}], F^{\a\b}>_{\g}dv_{\g} = \int_{U}<[\dot{A}_{\b},A_{\a}], F^{\a\b}>_{\g}dv_{\g}$$ 
(By anti-symmetry of $F$)
$$= \int_{U}<\dot{A}_{\b},[A_{\a}, F^{\a\b}]>_{\g}dv_{\g}$$
 because $<\; ,\;>$ is Ad-invariant.
This yields to

 $$0 = <\dot{A}_{\b}, \der_{\a}F^{\a\b} + [A_{\a}, F^{\a\b}]>_{\g}dv_{\g} = \int_{U}<\dot{A}_{\b}, \textbf{D}^{(A)}_{\a}F^{\a\b}>_{\g}dv_{\g}$$

So the covariant divergence of the curvature is zero 
\bea
\textbf{D}^{(A)}_{\a}F^{\a\b} = 0 \label{eq:YM}
\eea

On the other hand, computing 
\begin{eqnarray*}
&&\textbf{D}^{(A)}_{\a}F_{\mu\nu} + \textbf{D}^{(A)}_{\mu}F_{\nu\a} + \textbf{D}^{(A)}_{\nu}F_{\a\mu} \\
&=& \der_{\a}F_{\mu\nu} + [A_{\a}, F_{\mu\nu}] + \der_{\mu}F_{\nu\a} + [A_{\mu}, F_{\nu\a}] +  \der_{\nu}F_{\a\mu} + [A_{\nu}, F_{\a\mu}] \\
&=& \der_{\a}(\der_{\mu}A_{\nu} - \der_{\nu}A_{\mu} + [A_{\mu},A_{\nu}]) + [A_{\a}, \der_{\mu}A_{\nu} - \der_{\nu}A_{\mu} + [A_{\mu},A_{\nu}]] \\
&&+ \der_{\mu}(\der_{\nu}A_{\a} - \der_{\a}A_{\nu} + [A_{\nu},A_{\a}]) + [A_{\mu}, \der_{\nu}A_{\a} - \der_{\a}A_{\nu} + [A_{\nu},A_{\a}]] \\
&&+  \der_{\nu}(\der_{\a}A_{\mu} - \der_{\mu}A_{\a} + [A_{\a},A_{\mu}]) + [A_{\nu}, \der_{\a}A_{\mu} - \der_{\mu}A_{\a} + [A_{\a},A_{\mu}]] \\
&=&  \der_{\a}\der_{\mu}A_{\nu} - \der_{\a}\der_{\nu}A_{\mu} + [\der_{\a}A_{\mu},A_{\nu}] + [A_{\mu},\der_{\a}A_{\nu}]  \\
&&+ [A_{\a}, \der_{\mu}A_{\nu} - \der_{\nu}A_{\mu} + [A_{\mu},A_{\nu}]] + \der_{\mu}\der_{\nu}A_{\a} - \der_{\mu}\der_{\a}A_{\nu} \\
&&+ [\der_{\mu}A_{\nu},A_{\a}] + [A_{\nu}, \der_{\mu}A_{\a}] + [A_{\mu}, \der_{\nu}A_{\a} - \der_{\a}A_{\nu} + [A_{\nu},A_{\a}]] \\
&&+  \der_{\nu}\der_{\a}A_{\mu} - \der_{\nu}\der_{\mu}A_{\a} + [\der_{\nu}A_{\a},A_{\mu}] \\
&&+ [A_{\a},\der_{\nu}A_{\mu}] + [A_{\nu}, \der_{\a}A_{\mu} - \der_{\mu}A_{\a} + [A_{\a},A_{\mu}]]
\end{eqnarray*}
(where $\der_{\a}\der_{\mu} A = \der_{\a}(\der_{\mu}A) - \der_{\der_{\a}{e_{\mu}}} A$ is the tensorial covariant derivative of $A$)\
\begin{eqnarray*}
&=& {{R_{\nu}}^{\ga}}_{\a\mu}A_{\ga} + {{R_{\a}}^{\ga}}_{\mu\nu}A_{\ga} + {{R_{\mu}}^{\ga}}_{\nu\a}A_{\ga} + [A_{\a}, [A_{\mu},A_{\nu}]] +  [A_{\mu}, [A_{\nu},A_{\a}]] \\
&& +  [A_{\nu}, [A_{\a},A_{\mu}]]\\
&=& - ( {R^{\ga}}_{\nu\a\mu} + {R^{\ga}}_{\a\mu\nu} + {R^{\ga}}_{\mu\nu\a} ) A_{\ga} + [A_{\a}, [A_{\mu},A_{\nu}]] +  [A_{\mu}, [A_{\nu},A_{\a}]] \\
&& +  [A_{\nu}, [A_{\a},A_{\mu}]] \\
&=& 0
\end{eqnarray*}
by Bianchi identity and symmetry of the curvature tensor.

So we have, 

\bea
\textbf{D}^{(A)}_{\a}F_{\mu\nu} + \textbf{D}^{(A)}_{\mu}F_{\nu\a} + \textbf{D}^{(A)}_{\nu}F_{\a\mu} = 0 \label{eq:Bianchi}
\eea

The equations \eqref{eq:YM} and \eqref{eq:Bianchi} form the Yang-Mills equations. The Maxwell equations correspond to the abelian case where $[\; ,\;] =0$, and therefore $\cder = \der$.\

The Cauchy problem for the Yang-Mills equations formulates as the following: given a Cauchy hypersurface $\Sigma$ in M, and a ${\cal G}$-valued one form $A_{\mu}$ on $\Sigma$, and a ${\cal G}$-valued one form $E_{i}$ on $\Sigma$ satisfying $\mbox{\D}^{(A)}_{i}E^{i}$, we are looking for a ${\cal G}$-valued two form $F_{\mu\nu}$ satisfying the Yang-Mills equations such that once $F_{\mu\nu}$ restricted on M we have  $F_{0i} = E_{i}$, and such that $F_{\mu\nu}$ corresponds to the curvature derived from the Yang-Mills potential $A_{\mu}$ (i.e. $F_{\a\b} = \der_{\a}A_{\b} - \der_{\b}A_{\a} + [A_{\a},A_{\b}]$).\\

\section{Motivation}

Our motivation for the systematic study of gauge field theories such as the Yang-Mills equations is to have insights into the Einstein vacuum equations in General Relativity. Indeed, using Cartan formalism the Einstein vacuum equations can be viewed as to some extent mathematically related to the Yang-Mills equations. General Relativity postulates that the space-time is a 4-dimensional Lorentzian manifold $(M, \g)$, that satisfies the Einstein vacuum equations $R_{\mu\nu} = 0$, where $R_{\mu\nu}$ is the Ricci curvature, i.e. $R_{\mu\nu} = {R^{\ga}}_{\mu\ga\nu}$, and where $R_{\mu\nu\a\b}$ is the Riemann tensor associated to the metric $\g$, defined by, 
$$ R(X, Y,  U, V) = \g\big(X, \big[ \der_U \der_V-\der_V \der_U -\der_{[U,V]} Y\big]\big) $$
where $X, Y, U, V$ are vectorfields in the tangent bundle of $M$. We will see that using Cartan formalism one can write the Riemann tensor as a Yang-Mills curvature.

\subsection{Cartan formalism}\

At a point $p$ of the space-time, one can choose a normal frame, which means a frame such that $
\g(e_\a, e_\b) (p) =\mbox{diag}(-1,1,\ldots,1) $, and $\frac{\pa}{\pa \si} \g(e_{\a}, e_{\b})(p) = 0$, i.e. the first partial derivatives of the metric at $p$ vanish. Cartan formalism consists in defining the connection 1-form,
\bea
(A)_{\a\b}(X)=\g(\der_{X}e_\b,e_\a)  \label{cartanformalism}
\eea
where $\der$ is the Levi-Cevita connection. Thus, since $A = A_{\mu} dx^{\mu}$, we can write,
\beaa
(A_\mu)_{\a\b}=  (A)_{\a\b}(\frac{\pa}{\pa \mu})      =       \g(\der_{\mu}e_\b ,e_\a)
\eeaa

Computing,
\beaa
R(e_{\a}, e_{\b} , \frac{\pa}{\pa \mu }, \frac{\pa}{\pa \nu}) &=& \g\big(e_{\a},  \big( \der_{\frac{\pa}{\pa \mu}} \der_{\frac{\pa}{\pa \nu }}-\der_{\frac{\pa}{\pa \nu }} \der_{\frac{\pa}{\pa \mu }} -\der_{[\frac{\pa}{\pa \mu },\frac{\pa}{\pa \nu }]} \big ) e_{\b} \big) \\
&=& \g\big(e_{\a},   \der_{\frac{\pa}{\pa_{\mu}}} \der_{\frac{\pa}{\pa \nu }} e_{\b} -\der_{\frac{\pa}{\pa \nu}} \der_{\frac{\pa}{\pa \mu}} e_{\b} -\der_{[\frac{\pa}{\pa \mu},\frac{\pa}{\pa \nu}]} e_{\b} \big) \\
&=& \frac{\pa}{\pa \mu} \g ( e_{\a},    \der_{\frac{\pa}{\pa \nu}} e_{\b} ) -  \g ( \der_{\frac{\pa}{\pa \mu}} e_{\a},    \der_{\frac{\pa}{\pa \nu}} e_{\b} )   \\
&&- \big [ \frac{\pa}{\pa \nu} \g ( e_{\a},    \der_{\frac{\pa}{\pa \mu}} e_{\b} ) -  \g ( \der_{\frac{\pa}{\pa \nu}} e_{\a},    \der_{\frac{\pa}{\pa \mu}} e_{\b} ) \big ] \\
&& - \g( e_{\a}, \der_{[\frac{\pa}{\pa \mu},\frac{\pa}{\pa \nu}]} e_{\b} ) \\
&=& \frac{\pa}{\pa \mu} \g ( e_{\a},    \der_{\frac{\pa}{\pa \nu}} e_{\b} ) -   \frac{\pa}{\pa \nu} \g ( e_{\a},    \der_{\frac{\pa}{\pa \mu}} e_{\b} )  \\
&& +  \g ( \der_{\frac{\pa}{\pa \nu}} e_{\a},    \der_{\frac{\pa}{\pa \mu}} e_{\b} )  -  \g ( \der_{\frac{\pa}{\pa \mu}} e_{\a},    \der_{\frac{\pa}{\pa \nu}} e_{\b} )    \\
\eeaa

($ [\frac{\pa}{\pa \mu},\frac{\pa}{\pa \nu}] = 0 $ since they are coordinate vectorfields)\\

\beaa 
&=& \frac{\pa}{\pa \mu} (A_\nu)_{\a\b} -   \frac{\pa}{\pa \nu} (A_\mu)_{\a\b}  +  \g ( \der_{\frac{\pa}{\pa \nu}} e_{\a},    \der_{\frac{\pa}{\pa \mu}} e_{\b} )  -  \g ( \der_{\frac{\pa}{\pa \mu}} e_{\a},    \der_{\frac{\pa}{\pa \nu}} e_{\b} )    \\
&=& \frac{\pa}{\pa \mu} (A_\nu)_{\a\b} -   \frac{\pa}{\pa \nu} (A_\mu)_{\a\b}  +  \g ( e^{\la} ( \der_{\frac{\pa}{\pa \nu}} e_{\a}) e_{\la},    \der_{\frac{\pa}{\pa \mu}} e_{\b} )  -  \g ( e^{\la} (\der_{\frac{\pa}{\pa \mu}} e_{\a}) e_{\la},    \der_{\frac{\pa}{\pa \nu}} e_{\b} )    \\
&=& \frac{\pa}{\pa \mu} (A_\nu)_{\a\b} -   \frac{\pa}{\pa \nu} (A_\mu)_{\a\b}  +  e^{\la} ( \der_{\frac{\pa}{\pa \nu}} e_{\a}) \g (  e_{\la},    \der_{\frac{\pa}{\pa \mu}} e_{\b} )  -  e^{\la} (\der_{\frac{\pa}{\pa \mu}} e_{\a}) \g (  e_{\la},    \der_{\frac{\pa}{\pa \nu}} e_{\b} )    \\
&=& \frac{\pa}{\pa \mu} (A_\nu)_{\a\b} -   \frac{\pa}{\pa \nu} (A_\mu)_{\a\b}  +  e^{\la} ( \der_{\frac{\pa}{\pa \nu}} e_{\a}) (A_\mu)_{\la\b}  -  e^{\la} (\der_{\frac{\pa}{\pa \mu}} e_{\a}) (A_\nu)_{\la\b}    
\eeaa

Computing at the point $p$,
\beaa
 (A_\mu)^{\la}\,_{\a} = \g^{\ga\la}  (A_\mu)_{ \ga \a} =  \g^{\la\la} (A_\mu)_{\la\a} = \g(e_{\la}, e_{\la})^{-1} (A_\mu)_{\la\a} = \g(e_{\la}, e_{\la})^{-1} \g(\der_{\mu}e_\a ,e_\la)
\eeaa

Since the metric is compatible, we have $\der \g = 0$, and thus,

\beaa
\der_{\mu} \g(e_{\a}, e_{\la}) &=& \frac{\pa}{\pa \mu}  \g(e_{\a}, e_{\la}) - \g( \der_{\mu} e_{\a}, e_{\la})  - \g(e_{\a}, \der_{\mu} e_{\la})\\
&=&  - \g( \der_{\mu} e_{\a}, e_{\la})  - \g(e_{\a}, \der_{\mu} e_{\la}) = 0
\eeaa

Therefore,
\beaa
 \g( \der_{\mu} e_{\a}, e_{\la})  = - \g(e_{\a}, \der_{\mu} e_{\la}) 
\eeaa
and thus, the matrix $A$ is anti-symmetric, ie. $(A_{\mu})_{\a\b} = (A_{\mu})_{\b\a} $.

We have,
\beaa
\der_{\frac{\pa}{\pa \mu}} e_{\a} = e^{\la} (\der_{\frac{\pa}{\pa \mu}} e_{\a}) e_{\la}
\eeaa
Thus,
\beaa
\g(e_{\la}, \der_{\frac{\pa}{\pa \mu}} e_{\a}) = \g(e_{\la}, e^{\la} (\der_{\frac{\pa}{\pa \mu}} e_{\a}) e_{\la} ) = \g(e_{\la}, e_{\la}) e^{\la} (\der_{\frac{\pa}{\pa \mu}} e_{\a})
\eeaa
Consequently,
\beaa
e^{\la} (\der_{\frac{\pa}{\pa \mu}} e_{\a}) = \g(e_{\la}, e_{\la})^{-1} \g(e_{\la}, \der_{\frac{\pa}{\pa \mu}} e_{\a} ) = (A_\mu)^{\la}\,_{\a}
\eeaa
and,
\beaa
e^{\la} (\der_{\frac{\pa}{\pa \nu}} e_{\a}) =  \g(e_{\la}, e_{\la})^{-1} \g(e_{\la}, \der_{\frac{\pa}{\pa \nu}} e_{\a} ) = (A_\nu)^{\la}\,_{\a}
\eeaa

Therefore,

\beaa
R(e_{\a}, e_{\b} , \frac{\pa}{\pa \mu }, \frac{\pa}{\pa \nu}) &=& \pa_{\mu} (A_{\nu})_{\a\b} - \pr_{\nu} (A_{\mu})_{\a\b}+ (A_\nu)^{\la}\,_{\a}  (A_{\mu})_{\la\b} - (A_\mu)^{\la}\,_{\a} (A_\nu)_{\la\b} \\
&=& \pa_{\mu} (A_{\nu})_{\a\b} - \pr_{\nu} (A_{\mu})_{\a\b}  - (A_{\nu})_{\a}\,^{\la}  (A_{\mu})_{\la\b} + (A_{\mu})_{\a}\,^{\la}  (A_\nu)_{\la\b} 
\eeaa
(by anti-symmetry of the matrix $A$).

As commutator of matrices, we have,
\beaa
[A_\mu, A_\nu] = A_\mu  A_\nu -
A_\nu A_\mu \\
\eeaa
and thus,
\beaa
([A_\mu, A_\nu])_{\a\b}=(A_\mu)_{\a}{\,^\la} \,( A_\nu)_{\la\b}- (A_\nu)_{\a}{\,^\la} \,( A_\mu)_{\la\b}
\eeaa

Consequently, we get,
\bea
R_{\a\b\mu\nu}=\pr_\mu (A_\nu)_{\a\b}-\pr_\nu (A_\mu)_{\a\b}  + ([A_\mu, A_\nu])_{\a\b},
\eea

We have,
\beaa
\der_{\mu} A_{\nu} = \pr_{\mu} (A_\nu) - (A)( \der_{\mu} \frac{\pa}{\pa \nu})
\eeaa
and,
\beaa
\der_{\nu} A_{\mu} = \pr_{\nu} (A_\mu) - (A)( \der_{\nu} \frac{\pa}{\pa \mu})
\eeaa
Thus,
\beaa
\der_{\mu} A_{\nu}-\der_{\nu} A_{\mu} &=& \pr_{\mu} (A_\nu)  -\pr_{\nu} (A_\mu) + (A)( \der_{\nu} \frac{\pa}{\pa \mu}) - (A)( \der_{\mu} \frac{\pa}{\pa \nu})\\
&=& \pr_{\mu} (A_\nu)  -\pr_{\nu} (A_\mu) + (A)( \der_{\nu} \frac{\pa}{\pa \mu} -  \der_{\mu} \frac{\pa}{\pa \nu})\\  
&=& \pr_{\mu} (A_\nu)  -\pr_{\nu} (A_\mu) + (A)( [\frac{\pa}{\pa \nu}, \frac{\pa}{\pa \mu}] ) 
\eeaa
(because the metric is symmetric)
\beaa
&=& \pr_{\mu} (A_\nu)  -\pr_{\nu} (A_\mu) \\
\eeaa
(since $\frac{\pa}{\pa \mu}$, and $\frac{\pa}{\pa \nu}$ are coordinate vectorfields, therefore they commute).\\
As a result, 
\bea
R_{\a\b\mu\nu}=\big(\der_\mu A_\nu-\der_\nu A_\mu + [A_\mu, A_\nu]\big)_{\a\b} 
\eea
Since the curvature tensor of the connection $A$ is,
\bea
(F_{\mu\nu})_{\a\b} = \big(\der_\mu A_\nu-\der_\nu A_\mu + [A_\mu, A_\nu]\big)_{\a\b} \label{cartanformalism2}
\eea
We get,
\bea
R_{\a\b\mu\nu} = (F_{\mu\nu})_{\a\b}
\eea

\subsection{The Einstein equations in a Yang-Mills form}\

The following is a well known proposition, of which we sketch the proof.

\begin{proposition}
Let $(M, \g)$ be a 4-dimensional Lorentzian manifold that is Ricci flat, i.e. $R_{\mu\nu} = 0$. Then, we have $( {\D^{(A)}}^{\mu} F_{\mu\nu} )_{\a\b} = 0$, where $A$ and $F$ are defined as in \eqref{cartanformalism} and \eqref{cartanformalism2}.
\end{proposition}

\begin{proof}\

Computing,
\beaa
\der_{\si} R_{\a\b\mu\nu} &=& \frac{\pa}{\pa \si} R(e_{\a}, e_{\b}, \frac{\pa}{\pa \mu}, \frac{\pa}{\pa \nu}) - R(\der_{\si} e_{\a}, e_{\b}, \frac{\pa}{\pa \mu}, \frac{\pa}{\pa \nu}) - R(e_{\a}, \der_{\si} e_{\b}, \frac{\pa}{\pa \mu}, \frac{\pa}{\pa \nu}) \\
&& - R(e_{\a}, e_{\b},  \der_{\si} \frac{\pa}{\pa \mu}, \frac{\pa}{\pa \nu}) - R(e_{\a}, e_{\b}, \frac{\pa}{\pa \mu}, \der_{\si} \frac{\pa}{\pa \nu}) \\
&=& \frac{\pa}{\pa \si} R(e_{\a}, e_{\b}, \frac{\pa}{\pa \mu}, \frac{\pa}{\pa \nu})  - R(e_{\a}, e_{\b},  \der_{\si} \frac{\pa}{\pa \mu}, \frac{\pa}{\pa \nu}) - R(e_{\a}, e_{\b}, \frac{\pa}{\pa \mu}, \der_{\si} \frac{\pa}{\pa \nu})    \\
&& - R(\der_{\si} e_{\a}, e_{\b}, \frac{\pa}{\pa \mu}, \frac{\pa}{\pa \nu}) - R(e_{\a}, \der_{\si} e_{\b}, \frac{\pa}{\pa \mu}, \frac{\pa}{\pa \nu}) \\
&=& ( \der_{\si}  F_{\mu\nu} )_{\a\b}  - R(\der_{\si} e_{\a}, e_{\b}, \frac{\pa}{\pa \mu}, \frac{\pa}{\pa \nu}) - R(e_{\a}, \der_{\si} e_{\b}, \frac{\pa}{\pa \mu}, \frac{\pa}{\pa \nu}) \\
&=& ( \der_{\si}  F_{\mu\nu} )_{\a\b} - R(e^{\la} ( \der_{\si} e_{\a} ) e_{\la}, e_{\b}, \frac{\pa}{\pa \mu}, \frac{\pa}{\pa \nu}) - R(e_{\a}, e^{\la} ( \der_{\si} e_{\b} ) e_{\la}, \frac{\pa}{\pa \mu}, \frac{\pa}{\pa \nu}) \\
&=& ( \der_{\si}  F_{\mu\nu} )_{\a\b} - e^{\la} ( \der_{\si} e_{\a} ) R( e_{\la}, e_{\b}, \frac{\pa}{\pa \mu}, \frac{\pa}{\pa \nu}) - e^{\la} ( \der_{\si} e_{\b} ) R(e_{\a},  e_{\la}, \frac{\pa}{\pa \mu}, \frac{\pa}{\pa \nu}) \\
 &=& ( \der_{\si}  F_{\mu\nu} )_{\a\b} - (A_\si)^{\la}\,_{\a}   (F_{\mu\nu})_{\la\b} - (A_\si)^{\la}\,_{\b}  (F_{\mu\nu})_{\a\la} \\
 &=& ( \der_{\si}  F_{\mu\nu} )_{\a\b} + (A_\si)_{\a}\,^{\la}   (F_{\mu\nu})_{\la\b} - (F_{\mu\nu})_{\a\la} (A_\si)^{\la}\,_{\b}  
\eeaa
(using the anti-symmetry of $A$). Thus,

\bea
\notag
\der_{\si} R_{\a\b\mu\nu} &=& ( \der_{\si}  F_{\mu\nu} )_{\a\b}  + ( [A_\si, F_{\mu\nu}])_{\a\b} \\
&=&  ( \D^{(A)}_{\si} F_{\mu\nu} )_{\a\b}
\eea

Computing,
\beaa
 ( {\D^{(A)}}^{\mu} F_{\mu\nu} )_{\a\b} &=& \der^{\mu}  R_{\a\b\mu\nu} = \g^{\mu\si} \der_{\si} R_{\a\b\mu\nu} \\
&=& \der_{\si}  \g^{\mu\si} R_{\a\b\mu\nu} 
\eeaa
(because the metric is compatible)
\beaa
&=& \der_{\si}  R_{\a\b}\,^{\si}\,_{\nu}  = \der_{\si}  R^{\si}\,_{\nu\a\b}  
\eeaa
(using the symmtery of the Riemann tensor)
\beaa
&=& - \der_{\a} R^{\si}\,_{\nu\b\si}  - \der_{\b} R^{\si}\,_{\nu\si\a} 
\eeaa
(where we have used another symmetry of the Riemann tensor)
\beaa
&=&  \der_{\a} R^{\si}\,_{\nu\si\b}  - \der_{\b} R^{\si}\,_{\nu\si\a}  = \der_{\a} R_{\nu\b} - \der_{\b} R_{\nu\a} = 0
\eeaa
(since the Einstein vacuum equations say that $R_{\mu\ga} = 0 = {R^{\si}}_{\mu\si\ga}$ ). We get,
\bea
 ( {\D^{(A)}}^{\mu} F_{\mu\nu} )_{\a\b} = 0 \label{eq:divfree}
\eea

\end{proof}

The second Bianchi identitiy for the Riemann tensor,
\bea
\notag
0 &=& \der_{\a} R_{\ga\si\mu\nu} + \der_{\mu} R_{\ga\si\nu\a} + \der_{\nu} R_{\ga\si\a\mu} \\
&=& \textbf{D}^{(A)}_{\a}F_{\mu\nu} + \textbf{D}^{(A)}_{\mu}F_{\nu\a} + \textbf{D}^{(A)}_{\nu}F_{\a\mu} \label{eq:permut}
\eea
which is the Bianchi identity for the Yang-Mills fields. The equations above \eqref{eq:divfree} and \eqref{eq:permut} are the Yang-Mills equations except to the fact that the background geometry $(M, \g)$ is part of the unknown that we are looking for while trying to solve the Einstein vacuum equations.

This analogy between the Einstein equations in General Relativity and the Yang-Mills equations has been pursued by V. Moncrief in [M], by developing an integral representation formula for the curvature tensor in General Relativity, and then independently by I. Rodnianski and S. Klainerman in [KR1] and [KR3] partly as a desire to adapt the Eardley-Moncrief argument [EM1]-[EM2] to General Relativity.\\

\section{The Proof of the Global Existence of Yang-Mills Fields on Arbitrary, Sufficiently Smooth, Globally Hyperbolic, Curved Lorentzian Manifolds}

\subsection{A hyperbolic formulation for the Yang-Mills equations}\

It is known that the Yang-Mills fields can be shown to satisfy a tensorial hyperbolic wave equation with sources, on the background geometry. To see this, we start by taking the covariant divergence of \eqref{eq:Bianchi}, we obtain:

$${\cder}^{\a}{\cder}_{\a}F_{\mu\nu} + {\cder}^{\a}{\cder}_{\mu}F_{\nu\a} + {\cder}^{\a}{\cder}_{\nu}F_{\a\mu} = 0$$

We have:
\begin{eqnarray*}
{\cder}^{\a}{\cder}_{\mu}F_{\nu\a} &=& {\cder}_{\mu}{\cder}^{\a}F_{\nu\a} +  \der^{\a}\der_{\mu}F_{\nu\a} -  \der_{\mu}\der^{\a}F_{\nu\a} + [{F^{\a}}_{\mu}, F_{\nu\a}] \\
&=& {\cder}_{\mu}{\cder}^{\a}F_{\nu\a} +{{{R_{\nu}}^{\ga}}^{\a}}_{\mu}F_{\ga\a} + {{{R_{\a}}^{\ga}}^{\a}}_{\mu}F_{\nu\ga} + [F^{\a}_{\mu},F_{\nu\a}] \\
&=& {\cder}_{\mu} ({\cder}^{\a}F_{\nu\a} ) - {\cder}^{\der_{nu} e_\a}F_{\nu\a}+ {{{R_{\nu}}^{\ga}}^{\a}}_{\mu}F_{\ga\a} + {{{R_{\a}}^{\ga}}^{\a}}_{\mu}F_{\nu\ga} \\
&& + [{F^{\a}}_{\mu}, F_{\nu\a}] \\
&=& 0   - {\cder}^{\der_{\nu} e_\a}F_{\nu\a}+ {{{R_{\nu}}^{\ga}}^{\a}}_{\mu}F_{\ga\a} + {{{R_{\a}}^{\ga}}^{\a}}_{\mu}F_{\nu\ga} + [{F^{\a}}_{\mu}, F_{\nu\a}] 
\end{eqnarray*}
(by equation \eqref{eq:YM}). By choosing a normal frame at each point in space-time, i.e. a frame where $\g(e_\a, e_\b) =\mbox{diag}(-1,1,\ldots,1) $ and $ \der_{\a} e_{\b} = 0$ at that point, to compute the contraction ${\cder}^{\der_{\nu} e_\a}F_{\nu\a}$, we get that it vanishes.
So
\begin{eqnarray*}
 {\cder}^{\a}{\cder}_{\mu}F_{\nu\a} &=& R_{\nu\ga\a\mu}F^{\ga\a} + {{{R_{\a}}_{\ga}}^{\a}}_{\mu}{F_{\nu}}^{\ga} + [{F^{\a}}_{\mu}, F_{\nu\a}] \\
&=& R_{\nu\ga\a\mu}F^{\ga\a} + R_{\ga\mu}{F_{\nu}}^{\ga} + [{F^{\a}}_{\mu}, F_{\nu\a}] \\
\end{eqnarray*}
$$ {\cder}^{\a}{\cder}_{\mu}F_{\nu\a}= R_{\ga\mu\nu\a}F^{\a\ga} + R_{\mu\ga}{F_{\nu}}^{\ga} + [{F^{\a}}_{\mu}, F_{\nu\a}] $$

On the other hand, we have:
\begin{eqnarray*}
{\cder}^{\a}{\cder}_{\nu}F_{\a\mu} &=& {\cder}_{\nu}{\cder}^{\a}F_{\a\mu} +  \der^{\a}\der_{\nu}F_{\a\mu} -  \der_{\nu}\der^{\a}F_{\a\mu} + [{F^{\a}}_{\nu}, F_{\a\mu}] \\
&=& 0 + {{{R_{\a}}^{\ga}}^{\a}}_{\mu}F_{\ga\mu} + {{{R_{\mu}}^{\ga}}^{\a}}_{\nu}F_{\a\ga} + [{F^{\a}}_{\nu},F_{\a\mu}] \\
&& \text{(by equation \eqref{eq:YM})} \\
&=& {{{R_{\a}}_{\ga}}^{\a}}_{\mu}{F^{\ga}}_{\mu} + R_{\mu\ga\a\nu}F^{\a\ga} + [F_{\a\nu}, {F^{\a}}_{\mu}] \\
&=&  R_{\ga\nu}{F^{\ga}}_{\mu} + R_{\ga\mu\nu\a}F^{\a\ga} + [{F^{\a}}_{\mu}, F_{\nu\a}] \
\end{eqnarray*}

(where we have used the anti-symmetry of F)\

$$= R_{\nu\ga}{F^{\ga}}_{\mu} + R_{\ga\mu\nu\a}F^{\a\ga} + [{F^{\a}}_{\mu}, F_{nu\a}]$$

As we have 

$$\textbf{D}^{(A)}_{\a}F_{\mu\nu} + \textbf{D}^{(A)}_{\mu}F_{\nu\a} + \textbf{D}^{(A)}_{\nu}F_{\a\mu} = 0$$

we get $${\cder}^{\a} {\cder}_{\a}F_{\mu\nu} + 2R_{\ga\mu\nu\a}F^{\a\ga} + R_{\mu\ga}{F_{\nu}}^{\ga} + R_{\nu\ga}{F^{\ga}}_{\mu} + 2[{F^{\a}}_{\mu}, F_{\nu\a}] = 0 $$

We obtain:
\bea
\notag
\Box^{(A)}_{\g} F_{\mu \nu} = {\cder}^{\a}{\cder}_{\a}F_{\mu\nu} = -2R_{\ga\mu\nu\a}F^{\a\ga} - R_{\mu\ga}{F_{\nu}}^{\ga} - R_{\nu\ga}{F^{\ga}}_{\mu} - 2[{F^{\a}}_{\mu}, F_{\nu\a}] \\  \label{hyperbolic}
\eea

Due to the equation \eqref{hyperbolic}, the held belief is that the Yang-Mills equations are hyperbolic in nature.\\

\subsection{Energy estimates}\

Consider the energy momentum tensor:

\bea
T_{\mu\nu} = < F_{\mu\b},{F_{\nu}}^{\b}>  - \frac{1}{4} \g_{\mu\nu} <F_{\a\b},F^{\a\b}>
\eea

We will wright $<F_{\a\b},F^{\a\b}>$ as $F_{\a\b}.F^{\a\b}$ to lighten the notation.\

Taking the covariant divergence of $T_{\mu\nu}$ we obtain:
\begin{eqnarray*}
\der^{\nu} T_{\mu\nu} &=&  \der^{\nu} ( F_{\mu\b}.{F_{\nu}}^{\b} - \frac{1}{4} \g_{\mu\nu} F_{\a\b}.F^{\a\b} ) \\
&=& ({\cder}^{\nu} F_{\mu\b}) .{F_{\nu}}^{\b} +  F_{\mu\b}.{\cder}^{\nu} {F_{\nu}}^{\b} - \frac{1}{4} \g_{\mu\nu}{\cder}^{\nu}  F_{\a\b} .F^{\a\b}  \\
&& - \frac{1}{4} \g_{\mu\nu} F_{\a\b}.  {\cder}^{\nu} F^{\a\b}  
\end{eqnarray*}
(where we used the fact that the metric is Killing, i.e. $ \der \g = 0$, and that  $<$ , $>$ is Ad-invariant )\
\begin{eqnarray*}
&=&   ({\cder}^{\nu} F_{\mu\b}) .{F_{\nu}}^{\b} - \frac{1}{2} \g_{\mu\nu}{\cder}^{\nu} F_{\a\b} .F^{\a\b} \\
\end{eqnarray*}
(we used the field equations)
\begin{eqnarray*}
&=&   (\cder_{\a} F_{\mu\b}) .F^{\a\b} - \frac{1}{2}(\cder_{\mu}  F_{\a\b}) .F^{\a\b} \\
&=&   (\cder_{\a} F_{\mu\b}) .F^{\a\b} + \frac{1}{2}( \cder_{\a} F_{\b\mu} ) .F^{\a\b} + \frac{1}{2} (\cder_{\b} F_{\mu\a} ) .F^{\a\b}
\end{eqnarray*}
(using the Bianchi identities)\
\begin{eqnarray*}
&=&   (\cder_{\a} F_{\mu\b}) .F^{\a\b} - \frac{1}{2}( \cder_{\a} F_{\mu\b} ) .F^{\a\b} + \frac{1}{2} (\cder_{\a} F_{\mu\b} ) .F^{\b\a} \\
&=&   (\cder_{\a} F_{\mu\b}) .F^{\a\b} - \frac{1}{2}( \cder_{\a} F_{\mu\b} ) .F^{\a\b} - \frac{1}{2} (\cder_{\a} F_{\mu\b} ) .F^{\a\b} 
\end{eqnarray*}
(where we used the anti-symmetry of $F$ in the last two equalities)\
\bea
&&= 0
\eea

Considering a vector field $V^{\nu}$ we let $$J_{\mu}(V) = V^{\nu}T_{\mu\nu}$$

We have
\begin{eqnarray*}
\der^{\mu} J_{\mu}(V) &=& \der^{\mu} ( V^{\nu}T_{\mu\nu} ) \\
&=& \der^{\mu} ( V^{\nu}) T_{\mu\nu}
\end{eqnarray*}
(since $T$ is divergenceless )
\begin{eqnarray*}
= \frac{1}{2} ( \der^{\mu} ( V^{\nu}) T_{\mu\nu} +  \der^{\mu} ( V^{\nu}) T_{\mu\nu} ) = \frac{1}{2} ( \der^{\mu} ( V^{\nu}) T_{\mu\nu} +  \der^{\nu} ( V^{\mu}) T_{\mu\nu} ) 
\end{eqnarray*}
(where we used the symmetry of $T_{\mu\nu}$)
\bea
= \pi^{\mu\nu}(V) T_{\mu\nu}
\eea
where $\pi^{\mu\nu}(V)$ is the deformation tensor that is,
\bea
\pi^{\mu\nu}(V) = \frac{1}{2} ( \der^{\mu}  V^{\nu}+  \der^{\nu}  V^{\mu} ) 
\eea

Applying the divergence theorem on $ J_{\mu}(V)$ in a region $B$ bounded to the past by a spacelike hypersurface $\Sigma_{1}$ and to the future by a spacelike hypersurface $\Sigma_{2}$, and by a null hypersurface $N$, we obtain: \
\bea
\notag
\int_{B}  \pi^{\mu\nu}(V) T_{\mu\nu} dV_{B}  = \int_{\Sigma_{1}} J_{\mu}(V) w^{\mu} dV_{\Sigma_{1}} -  \int_{\Sigma_{2}} J_{\mu}(V) w^{\mu} dV_{\Sigma_{2}} - \int_{N}J_{\mu}(V) w_{N}^{\mu} dV_{N} \\
\eea

where $w^{\mu}$ are the unit normal to the hypersurfaces $\Sigma$, $w_{N}^{\mu}$ is any null generator of $N$, $dV_{\Sigma}$ are the induced volume forms and $dV_{N}$ is defined such that the divergence theorem applies. \

Taking $V = \frac{\pr}{\pr t}$, where $\frac{\pr}{\pr t}$ is a timelike vector field.

Taking $B = \Sigma_{+} \cap J^{-}(p)$, we get:\

\[
\int_{\Sigma_{+} \cap J^{-}(p)}  \pi^{\mu\nu}(\frac{\pr}{\pr t}) T_{\mu\nu} dV_{B}  = \int_{\Sigma \cap J^{-}(p)} J_{\mu}(\frac{\pr}{\pr t}) w^{\mu} dV_{\Sigma_{+}} - \int_{N^{-}(p)\cap \Sigma_{+}}J_{\mu}(\frac{\pr}{\pr t}) w_{N}^{\mu} dV_{N}
\]

where $w =  - \frac{\pr}{\pr \hat{t}}$, is the normalized timelike vector field, i.e.
\bea
\g(\frac{\pr}{\pr \hat{t} },\frac{\pr}{\pr \hat{t} }) = -1
\eea

\begin{definition}
Define the energy $E_{t}^{\frac{\pr}{\pr t}}$ by,
\bea
 E_{t}^{\frac{\pr}{\pr t}} = \int_{\Sigma_{t}} J_{\mu}(\frac{\pr}{\pr t}) (\frac{\pr}{\pr \hat{t} })^{\mu} dV_{\Sigma_{t}} = \int_{\Sigma_{t}} J_{\mu}(\frac{\pr}{\pr \hat{t} }) (\frac{\pr}{\pr \hat{t} })^{\mu} \sqrt{ - \g(\frac{\pr}{\pr t },\frac{\pr}{\pr t }) } dV_{\Sigma_{t}}  
\eea

And define the flux $F^{\frac{\pr}{\pr t }} (N^{-}(p)\cap \Sigma_{+} )$ by,
\bea
F^{\frac{\pr}{\pr t }} (N^{-}(p)\cap \Sigma_{+} ) = -  \int_{N^{-}(p)\cap \Sigma_{+}} J_{\mu}(\frac{\pr}{\pr t }) w_{N^{-}(p)}^{\mu} dV_{N^{-}(p)}
\eea

\end{definition}

We get:
\bea
\int_{\Sigma_{+} \cap J^{-}(p)}  \pi^{\mu\nu}(\frac{\pr}{\pr t }) T_{\mu\nu} dV_{B}  = -   E_{t =0}^{\frac{\pr}{\pr t }} (\Sigma\cap J^{-}(p))  + F^{\frac{\pr}{\pr t }} (N^{-}(p)\cap \Sigma_{+} )
\eea

\begin{definition}

We define $\{L, \Lb, e_{1}, e_{2}\}$ a null frame as in \eqref{defnullframe1}, \eqref{defnullframe2}, \eqref{defnullframe3}, and \eqref{defnullframe4}, in the following manner:

We define $L$ as in \eqref{definitionoftheparameters}, and we define $\Lb$ as:

\bea
\Lb = - \g(L, \hat{t})^{-1} (2 \hat{t} + \g(L, \hat{t})^{-1} L )  \label{definitionofLbar}
\eea
Define $e_{i}$, $i \in \{1, 2 \}$, such that,
\bea
\g(e_{i}, e_{j}) = \delta_{ij}
\eea
\bea
\g(L, e_{i}) = \g(\Lb, e_{i}) = 0
\eea

\end{definition}

We verify that

\bea
\notag
\g(\Lb, \Lb) &=&\g_{L\hat{t}}^{-2} \g(2 \hat{t} + \g_{L\hat{t}}^{-1} L , 2 \hat{t} + \g_{L\hat{t}}^{-1} L ) = 4 \g_{L \hat{t}}^{-2} \g(\hat{t}, \hat{t})  + 4 \g_{L\hat{t}}^{-2} \g_{L\hat{t}}^{-1} \g(\hat{t}, L)  \\
\notag
&=& -4 \g_{L \hat{t}}^{-2} + 4\g_{L \hat{t}}^{-2} \\
&=& 0
\eea
and,
\bea
\notag
\g(L, \Lb) &=& - \g_{L \hat{t}}^{-1 }\g(L, 2 \hat{t} + \g_{L\hat{t}}^{-1} L ) = -2 \g_{L \hat{t}}^{-1 } \g(L, \hat{t})  \\
&=& -2 \label{scalarproductLLbar}
\eea

We have
\bea
\hat{t} = - \frac{( \g_{L\hat{t}}^{-1}  L + \g_{L\hat{t}} \Lb)}{2}
\eea

\subsubsection{Computing explicitly $F^{\frac{\pr}{\pr t}} (N^{-}(p)\cap \Sigma_{+} )$}\

\begin{eqnarray*}
&& J_{\mu}(\frac{\pr}{\pr \hat{t}})  L^{\mu}  \\
&=& T_{\mu \hat{t}}  L^{\mu} = T_{L \hat{t}} = F_{L \b}.{F_{\hat{t}}}^{\b} - \frac{1}{4} \g_{L \hat{t}} F_{\a\b}.F^{\a\b} \\
&=& F_{L \Lb}.{F_{\hat{t}}}^{\Lb} +  F_{L e_{a}}.{F_{\hat{t}}}^{e_{a}} +  F_{L e_{b}}.{F_{\hat{t}}}^{e_{b}}  - \frac{1}{2} \g_{L \hat{t}} F_{\Lb L}.F^{\Lb L}  - \frac{1}{2} \g_{L \hat{t}} F_{L a}.F^{L a} \\
&&- \frac{1}{2} \g_{L \hat{t}} F_{L b}.F^{L b} - \frac{1}{2} \g_{L \hat{t}} F_{\Lb a}.F^{\Lb a} - \frac{1}{2} \g_{L \hat{t}} F_{\Lb b}.F^{\Lb b} - \frac{1}{2} \g_{L \hat{t}} F_{ab}.F^{ab} \\
&=&  - \frac{1}{2}  F_{L \Lb}.F_{\hat{t} L} +  F_{L e_{a}}.F_{\hat{t} e_{a}} +  F_{L e_{b}}.F_{\hat{t} e_{b}}   + \frac{1}{8} \g_{L \hat{t}} F_{L \Lb}.F_{L \Lb}    + \frac{1}{4} \g_{L \hat{t}} F_{L a}.F_{\Lb a} \\
&& + \frac{1}{4} \g_{L \hat{t}} F_{L b}.F_{\Lb b} + \frac{1}{4} \g_{L \hat{t}} F_{\Lb a}.F_{L a} + \frac{1}{4} \g_{L \hat{t}} F_{\Lb b}.F_{L b} - \frac{1}{2} \g_{L \hat{t}} F_{ab}.F_{ab} \\
&=&  - \frac{1}{2}  F_{L \Lb}.F_{\hat{t} L} +  F_{L e_{a}}.F_{\hat{t} e_{a}} +  F_{L e_{b}}.F_{\hat{t} e_{b}}   + \frac{1}{8} \g_{L \hat{t}} F_{L \Lb}.F_{L \Lb}    + \frac{1}{2} \g_{L \hat{t}} F_{L a}.F_{\Lb a}  \\
&& + \frac{1}{2} \g_{L \hat{t}} F_{L b}.F_{\Lb b}  - \frac{1}{2} \g_{L \hat{t}} F_{ab}.F_{ab}
\end{eqnarray*}

We have,

\begin{eqnarray*}
- \frac{1}{2}  F_{L \Lb}.F_{\hat{t} L} &=&  (- \frac{1}{2}) (-\frac{1}{2}) \g(L, \hat{t})  F_{L \Lb}.F_{\Lb L} = - \frac{1}{4} \g_{L\hat{t}}  F_{L \Lb}.F_{L \Lb} \\
 F_{L e_{a}}.F_{\hat{t} e_{a}} +  F_{L e_{b}}.F_{\hat{t} e_{b}} &=& - \frac{1}{2}\g_{L \hat{t}}^{-1} F_{L e_{a}}.F_{L e_{a}}   - \frac{1}{2}\g_{L\hat{t}}  F_{L e_{a}}.F_{\Lb e_{a}}  - \frac{1}{2} \g_{L \hat{t}}^{-1}  F_{L e_{b}}.F_{L e_{b}} \\
&& - \frac{1}{2} \g_{L\hat{t}}  F_{L e_{b}}.F_{\Lb e_{b}} 
\eeaa

Therefore,
\beaa
\notag
&&J_{\mu}(\frac{\pr}{\pr \hat{t}})  L^{\mu} \\
\notag
&=&  - \frac{1}{4} \g_{L\hat{t}}  F_{L \Lb}.F_{L \Lb} - \frac{1}{2} \g_{L \hat{t}}^{-1} F_{L e_{a}}.F_{L e_{a}}   - \frac{1}{2} \g_{L\hat{t}}  F_{L e_{a}}.F_{\Lb e_{a}}  - \frac{1}{2} \g_{L \hat{t}}^{-1} F_{L e_{b}}.F_{L e_{b}} \\
&& - \frac{1}{2} \g_{L\hat{t}}  F_{L e_{b}}.F_{\Lb e_{b}} + \frac{1}{8} \g_{L\hat{t}}  F_{L \Lb}.F_{L \Lb}    + \frac{1}{2}  \g_{L\hat{t}}  F_{L a}.F_{\Lb a}  + \frac{1}{2} \g_{L\hat{t}}  F_{L b}.F_{\Lb b}  \\
&& - \frac{1}{2} \g_{L\hat{t}}  F_{ab}.F_{ab} \\
&=&  - \frac{1}{8} \g_{L\hat{t}} F_{L \Lb}.F_{L \Lb} - \frac{1}{2} \g_{L \hat{t}}^{-1} F_{L e_{a}}.F_{L e_{a}}    - \frac{1}{2} \g_{L \hat{t}}^{-1} F_{L e_{b}}.F_{L e_{b}}     - \frac{1}{2}  \g_{L\hat{t}}  F_{ab}.F_{ab} 
\end{eqnarray*}

Thus,
\bea
\notag
&& F^{\frac{\pr}{\pr t }} (N^{-}(p)\cap \Sigma_{+} ) \\
\notag
&=&  \int_{N^{-}(p)\cap \Sigma_{+}}   ( \frac{1}{8} \g_{L\hat{t}} |F_{L \Lb}|^{2} + \frac{1}{2} \g_{L \hat{t}}^{-1} |F_{L e_{a}}|^{2}  +  \frac{1}{2} \g_{L \hat{t}}^{-1} |F_{L e_{b}}|^{2}  +  \frac{1}{2} \g_{L\hat{t}}  |F_{ab}|^{2} ) (q) \\
\eea
where $| \; . \; |$ is the norm deduced from $<$ , $>$.\

\subsubsection{Computing $E_{t}^{\frac{\pr}{\pr t}} $}\

\bea
E_{t}^{\frac{\pr}{\pr t}} = \int_{\Sigma_{t}} J_{\mu}(\frac{\pr}{\pr \hat{t}}) (\frac{\pr}{\pr \hat{t}})^{\mu} \sqrt{ - \g(\frac{\pr}{\pr t },\frac{\pr}{\pr t }) }    dV_{\Sigma_{t}} =  \int_{\Sigma_{t}} T_{\hat{t} \hat{t}}  \sqrt{ - \g(\frac{\pr}{\pr t },\frac{\pr}{\pr t }) }   dV_{\Sigma_{t}}
\eea
$$T_{\hat{t} \hat{t}} = F_{\hat{t} \b}.{F_{\hat{t}}}^{\b} - \frac{1}{4} \g_{\hat{t} \hat{t}} F_{\a\b}.F^{\a\b}$$

Choosing the frame $\{ \hat{t}, n, e_{a}, e_{b} \} $\

\begin{eqnarray*}
T_{\hat{t} \hat{t}} &=&  F_{\hat{t} n}.F_{\hat{t} n} + F_{\hat{t} a}.F_{\hat{t} a} +  F_{\hat{t} b}.F_{\hat{t} b} + \frac{1}{2}  F_{\hat{t}n}.F^{\hat{t}n} + \frac{1}{2}  F_{\hat{t}a}.F^{\hat{t}a} \\
&&+ \frac{1}{2}  F_{\hat{t}b}.F^{\hat{t}b} + \frac{1}{2}  F_{na}.F^{na} + \frac{1}{2}  F_{nb}.F^{nb} + \frac{1}{2}  F_{ab}.F^{ab}\\
&=&  F_{\hat{t} n}.F_{\hat{t} n} + F_{\hat{t} a}.F_{\hat{t} a} +  F_{\hat{t} b}.F_{\hat{t} b} - \frac{1}{2}  F_{\hat{t}n}.F_{\hat{t}n} - \frac{1}{2}  F_{\hat{t}a}.F_{\hat{t}a} - \frac{1}{2}  F_{\hat{t}b}.F_{\hat{t}b} \\
&&+ \frac{1}{2}  F_{na}.F_{na} + \frac{1}{2}  F_{nb}.F_{nb} + \frac{1}{2}  F_{ab}.F_{ab} \\
&=& \frac{1}{2}  F_{tn}.F_{\hat{t}n} + \frac{1}{2}  F_{\hat{t}a}.F_{\hat{t}a} + \frac{1}{2}  F_{\hat{t}b}.F_{\hat{t}b} + \frac{1}{2}  F_{na}.F_{na} + \frac{1}{2}  F_{nb}.F_{nb} + \frac{1}{2}  F_{ab}.F_{ab} \\
&=&  \frac{1}{2} [  | F_{\hat{t}n}|^{2} + |F_{\hat{t}a}|^{2} +  | F_{\hat{t}b}|^{2} +   |F_{na}|^{2} +   |F_{nb}|^{2} + | F_{ab}|^{2} ] (q) \geq 0 
\end{eqnarray*}

Thus,
\bea
E_{t=t_{0}}^{\frac{\pr}{\pr t}} (\Sigma\cap J^{-}(p)) \leq  E_{t=t_{0}}^{\frac{\pr}{\pr t }} \label{positivityofthe energy}
\eea

\subsubsection{Finiteness of the flux from finite initial energy}

Let $\Sigma^{-}_{t}$ be the past of  $\Sigma_{t}$.

We also have,
\bea
\notag
&& \int_{\Sigma_{+} \cap \Sigma^{-}_{t}  \cap J^{-}(p) }  \pi^{\mu\nu}(\frac{\pr}{\pr t }) T_{\mu\nu} dV_{B}  \\
\notag
 &\les& \sum_{\hat{\mu}, \hat{\nu} \in \{ \hat{t}, n, e_{a}, e_{b} \}  } \int_{\Sigma_{+} \cap  \Sigma^{-}_{t}  \cap J^{-}(p) }   |\pi^{\hat{\mu}\hat{\nu}}(\frac{\pr}{\pr t }) | |< F_{\hat{\mu}\b},{F_{\hat{\nu}}}^{\b}>  - \frac{1}{4} \g_{\hat{\mu}\hat{\nu}} <F_{\a\b},F^{\a\b}> | dV_{B}  \\
\notag
&\les& \sum_{\a, \b, \hat{\mu}, \hat{\nu} \in \{ \hat{t}, n, e_{a}, e_{b} \}  } \int_{t_{0}}^{t}   |\pi^{\hat{\mu}\hat{\nu}}(\frac{\pr}{\pr t })   |_{L^{\infty}_{\Sigma_{\overline{t}}  \cap J^{-}(p) }}    \int_{\Sigma_{\overline{t}}  \cap J^{-}(p) } ( |F_{\hat{\mu}\b}|^{2} + |{F_{\hat{\nu}}}^{\b}|^{2} \\
&&  \quad \quad \quad \quad \quad \quad \quad  + |F_{\a\b}|^{2} + |F^{\a\b} |^{2}) .\sqrt{ - \g(\frac{\pr}{\pr t },\frac{\pr}{\pr t }) }dV_{\Sigma_{\overline{t}}}  \label{inequalityondeformationtensorspacetimeintegral}
\eea
(by using $a.b \les a^{2} + b^{2}$).\\

 As in [CS] we assume that the deformation tensor of $\frac{\pr}{\pr t}$ is finite. More precisely, we assume that for all $\hat{\mu}, \hat{\nu} \in \{\hat{t}, n, e_{a}, e_{b} \}$, the components of the deformation tensor, $\pi^{\hat{\mu}\hat{\nu}}( \frac{\pa}{\pa t })  = \frac{1}{2} [ \der^{\hat{\mu}}  (\frac{\pa}{\pa t })^{\hat{\nu}}+  \der^{\hat{\nu}}  (\frac{\pa}{\pa t })^{\hat{\mu}} ] $, verify,
\bea
| \pi^{\hat{\mu}\hat{\nu}}( \frac{\pa}{\pa t }) |_{L^{\infty}_{loc(\Sigma_{t})}} \leq C(t)  \label{deformationtensorfinite}
\eea
where $C(t) \in L_{loc}^{1}$.\\

Applying the divergence theorem again in the future of $\Sigma$ and the past of $\Sigma_{t} \cap J^{-}(p) $, we get:
\bea
\notag
E_{t}^{\frac{\pr}{\pr t}}  (\Sigma_{t} \cap J^{-}(p)) &=& \int_{\Sigma_{+} \cap   \Sigma^{-}_{t} \cap J^{-}(p) }  \pi^{\mu\nu}(\frac{\pr}{\pr t}) T_{\mu\nu} dV_{B}  +   E_{t=t_{0}}^{\frac{\pr}{\pr t}} (\Sigma_{t_{0}} \cap J^{-}(p))\\
&\les&  E_{t=t_{0}}^{\frac{\pr}{\pr t}}   +  \int_{t =t_{0}}^{t }  C(\overline{t})  E_{\overline{t} }^{\frac{\pr}{\pr t }} (\Sigma_{\overline{t}} \cap J^{-}(p))  d \overline{t}     \label{localenergywillstayfinite}
\eea
(where we used \eqref{inequalityondeformationtensorspacetimeintegral}, \eqref{deformationtensorfinite}, and the positivity of the energy \eqref{positivityofthe energy}).

Using Gr\"onwall lemma, we get that $E_{t}^{\frac{\pr}{\pr t}} (\Sigma_{t} \cap J^{-}(p))   $ is finite and continuous in $t$, and therefore
\bea
\notag
 \int_{\Sigma_{+} \cap  J^{-}(p) }  \pi^{\mu\nu}(\frac{\pr}{\pr t }) T_{\mu\nu} dV_{B}  &\les& \int_{\hat{t} =t_{0}}^{t} c(\overline{t}) E_{\overline{t}}^{\frac{\pr}{\pr t }}  (\Sigma_{\overline{t}} \cap J^{-}(p))   d \overline{t} \\
\notag
&\les&  \int_{t_{0}}^{t} c(\overline{t}) d\overline{t} \\
&\les& c(t_{p} ) \label{boundingthespacetimeintegralfromthedivergencetheorem}
\eea

and therefore  $$\int_{\Sigma_{+} \cap J^{-}(p)}  \pi^{\mu\nu}(\frac{\pr}{\pr t}) T_{\mu\nu} dV_{B}  +  E_{t=t_{0}}^{\frac{\pr}{\pr t}} (\Sigma\cap J^{-}(p))  \lesssim c(t_{p}) E_{t=t_{0}}^{\frac{\pr}{\pr t}}  $$

Hence,
\bea
F^{\frac{\pr}{\pr t}} (N^{-}(p)\cap \Sigma_{+} )  \lesssim c(t_{p}) E_{t=t_{0}}^{\frac{\pr}{\pr t}} \label{finitenessflux}
\eea

This finiteness of the flux will play a key role in the proof.\\

\subsection{Definitions and notations}\

\begin{definition}
We define $\la_{\a\b}$ as in \eqref{eq:transport} and \eqref{eq:initial condition}, by fixing at $p$ a ${\cal G}$-valued anti-symmetric 2-tensor $\J_{p}$, and defining $\la_{\a\b}$ as the unique 2-tensor field along $N^{-}(p)$, the boundary of the causal past of $p$, that verifies the linear transport equation:
\beaa
\textbf{D}^{(A)}_{L}\la_{\a\b} + \frac{1}{2}tr\chi\la_{\a\b} = 0  \\
(s\la_{\a\b})(p) = \J_{\a\b}(p) 
\eeaa
where $s$ is the affine parameter on $N^{-}(p)$ defined as in \eqref{definitionoftheparameters}, and $\chi$ is the null second fundamental form of $N^{-}(p)$ defined as in \eqref{definitionofchi}, and $tr\chi$ defined as in \eqref{definitionoftraceofchi}.

\end{definition}

\begin{definition}
We define a timelike foliation $\Sigma_{t}$ by considering $t = constant$ hypersurfaces.
\end{definition}

\begin{definition}
We define positive definite Riemannian metric as in \eqref{positiveriemannianmetrich}, in the following manner:

\beaa
h(e_{\a}, e_{\b}) = \g(e_{\a}, e_{\b}) + 2 \g(e_{\a},  \frac{\pr}{\pr \hat{t}} ) . \g(e_{\b}, \frac{\pr}{\pr \hat{t}} )
\eeaa

where
\beaa
\frac{\pr}{\pr \hat{t}} = (- \g( \frac{\pr}{\pr t} , \frac{\pr}{\pr t}  ) )^{-\frac{1}{2}} \frac{\pr}{\pr t} 
\eeaa
\end{definition}

\begin{definition}
For any ${\cal G}$-valued 2-tensor $K$, we define
\beaa
|K|^{2} =  h_{\a\mu} h_{\b\nu} |K^{\mu\nu}|. |K^{\a\b}|
\eeaa
and
\beaa
|K|_{L^\infty} =  || (h_{\a\mu} h_{\b\nu} |K^{\mu\nu}|. |K^{\a\b}| )^{\frac{1}{2}} ||_{L^{\infty}} = || (|K|^{2} )^{\frac{1}{2}} ||_{L^{\infty}}
\eeaa

We recall \eqref{Cauchy-Schwarzinequalitywithmetrich}, that for any two  ${\cal G}$-valued tensors $K$ and $G$, we have
\beaa
| <K_{\a\b}, G^{\a\b}>| \les ( |K|^{2} )^{\frac{1}{2}}. ( |G|^{2} )^{\frac{1}{2}} 
\eeaa

\end{definition}

\subsubsection{Notations}

We denote by $N^{-}_{\tau}(p)$ the portion of $N^{-}(p)$ to the past of $\Sigma_{t_{p} }$ and to the future of $\Sigma_{t_{p} - \tau}$. 

$t^{*}$ and $\overline{t}$ are values of $t$, where $\frac{\pa}{\pa t }$ is a timelike vector field verifying  \eqref{deformationtensorfinite}.

$\overline{s}$ and $\hat{s}$ are values of $s$.

We denote by $s_{\tau}$, the largest value of $s$ on $N^{-}_{\tau}(p)$.

We let $\Sigma_{t}^{p} = \Sigma_{t}\cap J^{-}(p) $.

\subsection{Estimates for $s \la_{\a\b}$}\

\begin{proposition}
We have,
\bea
\sup _{N^{-}_{\tau}(p)} |s \la| \leq C(p, \tau ) |J|  \label{linfinitynormofslamda}
\eea
\end{proposition}

\begin{proof}\

We proved in \eqref{boundingB} that,

\bea
\sup _{0 \le \overline{s} \le s}  |\overline{s} \la|^{2} \leq C(p, s) |J|^{2}
\eea

Hence,
\bea
\sup _{N^{-}_{\tau}(p)} |s \la |^{2} \leq C(p, s_{\tau} ) |J|^{2} \label{boundBbysstarandJ}
\eea
(where in this last inequality $s_{\tau}$ is the largest value of $s$ on $N^{-}_{\tau} (p)$ ).\

In view of \eqref{relationsandt},
\bea
s = t_{p} - t + O(t_{p}-t) 
\eea
we get,
\bea
\sup _{N^{-}_{\tau}(p)} |s \la | \leq C(p, \tau ) |J| 
\eea

\end{proof}

\begin{proposition}
\bea
||  \la ||_{L^{2}(N^{-}_{\tau}(p))}   \leq(\tau)^{\frac{1}{2}}       C(p, \tau) |J|  
\eea
\end{proposition}

\begin{proof}\

We have,
\beaa
\int_{\SSS^{2}}  | s \la |^{2} (u=0, s, \om) da_{s}   &\lesssim&  \int_{\SSS^{2}}   C(p, s_{\tau})^{2} |J|^{2}   da_{s} \\
&& \text{(by \eqref{boundBbysstarandJ} )}\\
&\leq&  C(p, s_{\tau})^{2} |J|^{2}  \int_{\SSS^{2}}   1 da_{s} \\
& \lesssim& C(p, s_{\tau})^{2} |J|^{2}  s^{2} 
\eeaa

Thus, $$ \int_{\SSS^{2}}  | \la |^{2} (u=0, s, \om) da_{s}  \leq C(p, s_{\tau})^{2} |J|^{2}   $$

$$ || \la_{\a\b} ||_{L^{2}(N^{-}_{\tau}(p))}^{2} = \int_{0}^{s^{*}_{\tau} } \int_{\SSS^{2}}  |  \la |^{2} (u=0, s, \om) da_{s} ds $$ (where $s^{*}_{\tau} $ is the largest value of $s$ for a fixed $\om$ such that $(u=0, s, \om) \in N^{-}_{\tau}(p) ) $\

$$  \lesssim  C(p, s_{\tau} )^{2} |J|^{2}  \int_{0}^{s_{\tau} } 1   ds   \leq  s_{\tau}      C(p, s_{\tau})^{2} |J|^{2}   $$

Thus, $$ ||  \la ||_{L^{2}(N^{-}_{\tau}(p))}   \leq(s_{\tau} )^{\frac{1}{2}}       C(p, s_{\tau}  ) |J|   $$

Therefore, since $t_{p} - t = s + o(s)$,

\bea
||  \la ||_{L^{2}(N^{-}_{\tau}(p))}   \leq(\tau)^{\frac{1}{2}}       C(p, \tau) |J|  
\eea

\end{proof}

\subsection{Estimates for $|| \cder_{a} \la ||_{L^{2}(N^{-}_{\tau}(p))} $} \label{controlofthetangderivativeoflamdaasinKR3}\

\begin{definition}
Let
\bea
\hat{\chi}_{ab} = \chi_{ab} - \frac{1}{2} tr\chi \de_{ab} 
\eea
and let $\ze_{a}$ be defined as in \eqref{defzea}:
\bea
\ze_{a} &=& \frac{1}{2}\g(\der_{a}L, \Lb)  
\eea
\end{definition}

\begin{lemma}

\bea
\notag
\cder_{L} \cder_{a} \la_{\a\b} &=&     -  tr\chi  \cder_{a}  \la_{\a\b}    - \hat{\chi}_{ab} \cder_{b} \la_{\a\b} - \frac{1}{2} (\der_{a} tr\chi ) \la_{\a\b} - \frac{1}{2}   \ze_{a}   tr\chi \la_{\a\b} \\
&& + [F_{L a} , \la_{\a\b} ] + {{R_{\a}}^{\ga}}_{L a} \la_{\ga \b} + {{R_{\b}}^{\ga}}_{L a} \la_{\a \ga} \label{rcderLrcderalambdaalphabeta}
\eea

\end{lemma}

\begin{proof}

We have,

\begin{eqnarray*}
\cder_{L} \cder_{a} \la_{\a\b} &=&  \D^{(A)}_{a} \D^{(A)}_{L} \la_{\a\b}  + [F_{L a} , \la_{\a\b} ]  + {{R_{\a}}^{\ga}}_{L a} \la_{\ga \b} + {{R_{\b}}^{\ga}}_{L a} \la_{\a \ga} \\
 &=&  \D^{(A)}_{a} ( \D^{(A)}_{L} \la_{\a\b} ) - \D^{(A)}_{\der_{a} L} \la_{\a\b} + [F_{L a} , \la_{\a\b} ]  \\
&&+ {{R_{\a}}^{\ga}}_{L a} \la_{\ga \b} + {{R_{\b}}^{\ga}}_{L a} \la_{\a \ga} \\
 &=&  \D^{(A)}_{a} ( - \frac{tr\chi}{2} \la_{\a\b} )  - \D^{(A)}_{  \der_{a} L} \la_{\a\b}  + [F_{L a} , \la_{\a\b} ]  \\
&& + {{R_{\a}}^{\ga}}_{L a} \la_{\ga \b} + {{R_{\b}}^{\ga}}_{L a} \la_{\a \ga} 
\end{eqnarray*}

on $N^{-}_{\tau}(p) $.\

We remind that for any vectorfield $X$, we have $$X= - \frac{1}{2} \g(X, \Lb) L -  \frac{1}{2} \g(X, L) \Lb  + \g(X, e_{a}) e_{a}$$
Taking $X = \der_{a}L$, we get $$\der_{a}L = - \frac{1}{2} \g(\der_{a}L, \Lb) L -  \frac{1}{2} \g(\der_{a}L, L) \Lb  + \g(\der_{a}L, e_{b}) e_{b}$$
we get,
\bea
\der_{a} L =   \g(\der_{a}L, e_{b}) e_{b} - \ze_{a} L = \chi_{ab}  e_{b} - \ze_{a} L
\eea

Thus,

\bea
\notag
\cder_{L} \cder_{a} \la_{\a\b} &=&  - \frac{1}{2} (\der_{a} tr\chi ) \la_{\a\b}    - \frac{1}{2} tr\chi ( \cder_{a}  \la_{\a\b} )  +   \ze_{a}   \D^{(A)}_{L} \la_{\a\b} \\
&& - \chi_{ab} \cder_{b} \la_{\a\b} + [F_{L a} , \la_{\a\b} ] + {{R_{\a}}^{\ga}}_{L a} \la_{\ga \b} + {{R_{\b}}^{\ga}}_{L a} \la_{\a \ga} 
\eea

Since $\hat{\chi}_{ab} = \chi_{ab} - \frac{1}{2} tr\chi \de_{ab}$, and since we have $\cder_{L} \la_{\a\b} = - \frac{1}{2} tr \chi \la_{\a\b}$ on  $N^{-}_{\tau}(p)$, we get,

\beaa
\notag
\cder_{L} \cder_{a} \la_{\a\b} &=&  - \frac{1}{2} (\der_{a} tr\chi ) \la_{\a\b}    - \frac{1}{2} tr\chi ( \cder_{a}  \la_{\a\b} )  - \frac{1}{2}  \ze_{a}   tr\chi \la_{\a\b} \\
\notag
&& - \hat{\chi}_{ab} \cder_{b} \la_{\a\b} - \frac{1}{2} tr\chi \de_{ab} ( \cder_{b}  \la_{\a\b} ) + [F_{L a} , \la_{\a\b} ]  \\
&& + {{R_{\a}}^{\ga}}_{L a} \la_{\ga \b} + {{R_{\b}}^{\ga}}_{L a} \la_{\a \ga}\\
\notag
&=&     -  tr\chi  \cder_{a}  \la_{\a\b}    - \hat{\chi}_{ab} \cder_{b} \la_{\a\b} - \frac{1}{2} (\der_{a} tr\chi ) \la_{\a\b} - \frac{1}{2}  \ze_{a}   tr\chi \la_{\a\b} \\
&& + [F_{L a} , \la_{\a\b} ] + {{R_{\a}}^{\ga}}_{L a} \la_{\ga \b} + {{R_{\b}}^{\ga}}_{L a} \la_{\a \ga} 
\eeaa

\end{proof}

We want to control $|| \cder_{a} \la ||_{L^{2}(N^{-}_{\tau}(p))} $, where we are summing over $a \in \{1, 2\}$, by abuse of notation. Following [KR3], let's compute,
$$ \overline{s} \int_{\SSS^{2}} \overline{s}^{2} | \D^{(A)}_{a} \la |^{2} d\sigma^{2} =  \int_{\SSS^{2}} \overline{s}^{-1} \overline{s}^{4} | \D^{(A)}_{a} \la |^{2} d\sigma^{2} =  \int_{\SSS^{2}} | \overline{s}^{-1} \int_{0}^{\overline{s}} \rder_{L}  |s^{2} \D^{(A)}_{a} \la |^{2} ds | d\sigma^{2} $$

\begin{lemma}
Let $\Psi$ be a ${\cal G}$-valued tensor, we have,
\bea \label{estimateonderivativeofafullcontractionwithhintermsofmixedterms}
| \der_{\si} |\Psi|^{2} |(p) &\le& C(p) [  |\cder_{\si} \Psi|. | \Psi | +   |\Psi|^{2} ]
\eea
where $C(p)$ depends on the space-time geometry on the point $p$.
\end{lemma}

\begin{proof}

\beaa
 \der_{\si} |\Psi|^{2} &=&  \der_{\si} ( h_{\a\mu} h_{\b\nu}  |\Psi^{\mu\nu}|. |\Psi^{\a\b} | ) =  \der_{\si} ( h_{\a\mu} h_{\b\nu} ) .  |\Psi^{\mu\nu}|. |\Psi^{\a\b} | \\
&& +  h_{\a\mu} h_{\b\nu}  . \der_{\si} ( |\Psi^{\mu\nu}|. |\Psi^{\a\b} | ) 
\eeaa
Therefore,
\beaa
| \der_{\si} |\Psi|^{2} | &\le& | (\der_{\si} h_{\a\mu})  h_{\b\nu}  |.  |\Psi^{\mu\nu}|. |\Psi^{\a\b} | + | h_{\a\mu} (\der_{\si} h_{\b\nu} )| .  |\Psi^{\mu\nu}|. |\Psi^{\a\b} | \\
&& + | h_{\a\mu} h_{\b\nu} | .  ( | \cder_{\si} \Psi^{\mu\nu}| + |\Psi(\der_{\si} e^{\mu}, e^{\nu})| + | \Psi(e^{\mu}, \der_{\si} e^{\nu}|   ) . |\Psi^{\a\b} |) \\
&& + | h_{\a\mu} h_{\b\nu}|. |  \Psi^{\mu\nu}|. ( |\cder_{\si} \Psi^{\a\b} |  + | \Psi(\der_{\si} e^{\a}, e^{\b}) |+ | \Psi (e^{\a}, \der_{\si} e^{\b}) | ) 
\eeaa
(due to \eqref{derivativeestimateonthenormofacomponent}).\\

Choosing a normal frame (where the Christoffel symbols vanish at that point) to consider the contactions, using \eqref{derivativeofthemetrich}, and the fact that the metric is smooth, we get,
\beaa
\notag
| \der_{\si} |\Psi|^{2} |(p) &\le& C(p) [ h_{\a\mu} h_{\b\nu}  |\cder_{\si} \Psi^{\mu\nu}|. | \Psi^{\a\b} | +  h_{\a\mu} h_{\b\nu}  |\Psi^{\mu\nu}|. |\Psi^{\a\b} | ] \\ 
\eeaa
Using Cauchy-Schwarz, we obtain the desired estimate.

\end{proof}

We can compute,
\begin{eqnarray*}
\rcder_{L} (s^{2} \cder_{a} \la_{\a\b})  & =&   2 s \D^{(A)}_{a} \la_{\a\b} + s^{2} \cder_{L} \cder_{a} \la_{\a\b}  \\
& = &  s^{2} ( 2 s^{-1} \D^{(A)}_{a} \la_{\a\b} +  \cder_{L} \cder_{a} \la_{\a\b}  )  
\end{eqnarray*}

Thus, using \eqref{estimateonderivativeofafullcontractionwithhintermsofmixedterms}, we get

\bea
\notag
| \der_{L} |s^{2} \cder_{a} \la |^{2} | &\les& |s^{2} ( 2 s^{-1} \D^{(A)}_{a} \la +  \cder_{L} \cder_{a} \la  ) |.| s^{2} \cder_{a} \la |   + | s^{2} \cder_{a} \la |^{2}   \\  \label{derivativeofthehnormsquaredofs2lamda}
\eea

Therefore,
\bea
\notag
&&  \overline{s} \int_{\SSS^{2}} \overline{s}^{2} | \D^{(A)}_{a} \la |^{2} d\sigma^{2} \\
\notag
&\les&   \int_{\SSS^{2}} \overline{s}^{-1} | \overline{s}^{2} \D^{(A)}_{a} \la |^{2} d\sigma^{2} \\
&\les&   \int_{\SSS^{2}} \overline{s}^{-1} \int_{0}^{\overline{s}} | \der_{L} | \overline{s}^{2} \D^{(A)}_{a} \la |^{2}| ds d\sigma^{2}  \label{sminus1fundamentaltheoremcalculussfourdlamdasquared} \\
\notag
&\lesssim&  \int_{\SSS^{2}} [  \overline{s}^{-1} \int_{0}^{\overline{s}}  |s^{2} ( 2 s^{-1} \D^{(A)}_{a} \la +  \cder_{L} \cder_{a} \la  ) |.| s^{2} \cder_{a} \la |   + | s^{2} \cder_{a} \la |^{2} ds ] d\sigma^{2} \\
\notag
&\lesssim& \int_{\SSS^{2}} [ \overline{s}^{-1}  \int_{0}^{\overline{s}} \sup_{s \in [0, \overline{s} ] } s^{4} | \D^{(A)}_{a} \la|^{2}  ds \\
\notag
&& +  \overline{s}^{-1} \int_{0}^{\overline{s}} \eps^{-\frac{1}{2}} s^{\frac{5}{2}} | 2 s^{-1} \D^{(A)}_{a} \la +  \cder_{L} \cder_{a} \la |. \eps^{\frac{1}{2}} s^{\frac{3}{2}} | \D^{(A)}_{a} \la| ds ] d\sigma^{2} \\
\notag
&\lesssim&    \overline{s}^{-1}  \sup_{s \in [0, \overline{s} ] } (s)  \int_{\SSS^{2}} \int_{0}^{\overline{s}}   \sup_{s \in [0, \overline{s} ] } ( s^{3}  |\D^{(A)}_{a} \la|^{2}  ) ds  d\sigma^{2} +  \overline{s}^{-1}  \eps \int_{\SSS^{2}} \int_{0}^{\overline{s}}   \sup_{s \in [0, \overline{s} ] } ( s^{3}  |\D^{(A)}_{a} \la|^{2}  ) ds  d\sigma^{2}  \\
\notag
&& +   \frac{1}{\eps}   \int_{\SSS^{2}}  [  \overline{s}^{-1} \int_{0}^{\overline{s}} s^{\frac{5}{2}} | 2 s^{-1} \D^{(A)}_{a} \la +  \cder_{L} \cder_{a} \la |  ds  ]^{2} d\sigma^{2} \\
\notag
&\lesssim&    \overline{s}  \int_{\SSS^{2}}   \sup_{s \in [0, \overline{s} ] } ( s^{3}  |\D^{(A)}_{a} \la|^{2}  )  d\sigma^{2}  + \eps   \int_{\SSS^{2}}   \sup_{s \in [0, \overline{s} ] } ( s^{3}  |\D^{(A)}_{a} \la|^{2}  )  d\sigma^{2} \\
&&+    \frac{1}{\eps}   \int_{\SSS^{2}}  [  \overline{s}^{-1} \int_{0}^{\overline{s}} s^{\frac{5}{2}} | 2 s^{-1} \D^{(A)}_{a} \la +  \cder_{L} \cder_{a} \la |  ds  ]^{2} d\sigma^{2}  \label{controlontheLtwonormonStwoofsthreedlamdasquared}
\eea

Taking the supremum in this last inequality on $\overline{s} \in [0, \hat{s} ] $, where $\overline{s}$ and $\hat{s}$ are values of $s$, we get,
\begin{eqnarray*}
&& \sup_{\overline{s} \in [0, \hat{s} ] } \int_{\SSS^{2}} \overline{s}^{3} | \D^{(A)}_{a} \la |^{2} (0, \overline{s}, \om) d\sigma^{2} \\
& \lesssim& (\hat{s} + \eps )   \sup_{\overline{s} \in [0, \hat{s} ] }  \int_{\SSS^{2}}     \sup_{s \in [0, \overline{s} ] } ( s^{3}  |\D^{(A)}_{a} \la|^{2}  )   d\sigma^{2}  \\
&& +    \sup_{\overline{s} \in [0, \hat{s} ] }   \int_{\SSS^{2}}  [  \overline{s}^{-1} \int_{0}^{\overline{s}} s^{\frac{5}{2}} | 2 s^{-1} \D^{(A)}_{a} \la +  \cder_{L} \cder_{a} \la |  ds ]^{2} d\sigma^{2} 
\end{eqnarray*}

Choosing $\hat{s}$ and $\eps$ small enough depending on the space-time geometry on $p$, we obtain:

\bea
\notag
&& \sup_{\overline{s} \in [0, \hat{s} ] } \int_{\SSS^{2}} \overline{s}^{3} | \D^{(A)}_{a} \la |^{2} (0, \overline{s}, \om) d\sigma^{2}\\
&\lesssim&  \sup_{\overline{s} \in [0, \hat{s} ] }   \int_{\SSS^{2}}  [  \overline{s}^{-1} \int_{0}^{\overline{s}} s^{\frac{5}{2}} | 2 s^{-1} \D^{(A)}_{a} \la +  \cder_{L} \cder_{a} \la |  ds ]^{2} d\sigma^{2} \label{boundonthesupremumoftheL2normonS2ofs3squareofcderlambda}
\eea

We want to control 
\begin{eqnarray*}
&& \sup_{\overline{s} \in [0, \hat{s} ] }   \int_{\SSS^{2}}  [  \overline{s}^{-1} \int_{0}^{\overline{s}} s^{\frac{5}{2}} | 2 s^{-1} \D^{(A)}_{a} \la +  \cder_{L} \cder_{a} \la |  ds ]^{2} d\sigma^{2} \\
&=& || \sup_{\overline{s} \in [0, \hat{s} ] }   \overline{s}^{-1} \int_{0}^{\overline{s}} s^{\frac{5}{2}} | 2 s^{-1} \D^{(A)}_{a} \la +  \cder_{L} \cder_{a} \la |  ds ||_{L^{2}_{\om}}^{2}
\end{eqnarray*}
 where the $L^{p}_{\om}$ denotes the $L^{p}$ norm on $s =$ constant, with the canonical induced volume form $d\sigma^{2}$ induced on $\SSS^{2}$. \

\bea
\notag
&&|| \sup_{\overline{s} \in [0, \hat{s} ] }   \overline{s}^{-1} \int_{0}^{\overline{s}} s^{\frac{5}{2}} | 2 s^{-1} \D^{(A)}_{a} \la +  \cder_{L} \cder_{a} \la |  ds ||_{L^{2}_{\om}}^{2}  \\
\notag
&=& || \sup_{\overline{s} \in [0, \hat{s} ] }   \overline{s}^{-1} \int_{0}^{\overline{s}} s^{\frac{5}{2}} | 2 s^{-1} \D^{(A)}_{a} \la -tr\chi \D^{(A)}_{a} \la + tr\chi \D^{(A)}_{a} \la +  \cder_{L} \cder_{a} \la |  ds ||_{L^{2}_{\om}}^{2}   \\
\notag
&=& || \sup_{\overline{s} \in [0, \hat{s} ] }   \overline{s}^{-1} \int_{0}^{\overline{s}} s^{\frac{5}{2}} | 2 s^{-1} \D^{(A)}_{a} \la_{\a\b} -tr\chi \D^{(A)}_{a} \la_{\a\b} - \hat{\chi}_{ab} \cder_{b} \la_{\a\b} - \frac{1}{2} (\der_{a} tr\chi ) \la_{\a\b} \\
\notag
&& - \frac{1}{2}   \ze_{a}   tr\chi \la_{\a\b} + [F_{L a} , \la_{\a\b} ] + {{R_{\a}}^{\ga}}_{L a} \la_{\ga \b} + {{R_{\b}}^{\ga}}_{L a} \la_{\a \ga}  |_{h}  ds ||_{L^{2}_{\om}}^{2}   \\
\notag
\eea
where we used \eqref{rcderLrcderalambdaalphabeta}, and where $| \; \; |_{h}$ means that we consider a full contraction with respect to the metric $h$, in the indices $\a, \b$. Hence,

\bea
\notag
&&|| \sup_{\overline{s} \in [0, \hat{s} ] }   \overline{s}^{-1} \int_{0}^{\overline{s}} s^{\frac{5}{2}} | 2 s^{-1} \D^{(A)}_{a} \la +  \cder_{L} \cder_{a} \la |  ds ||_{L^{2}_{\om}}^{2}  \\
\notag
& \lesssim & || \sup_{\overline{s} \in [0, \hat{s} ] }   \overline{s}^{-1} \int_{0}^{\overline{s}} s^{\frac{5}{2}} | 2 s^{-1} \D^{(A)}_{a} \la -tr\chi \D^{(A)}_{a} \la - \hat{\chi}_{ab} \cder_{b} \la  |  ds ||_{L^{2}_{\om}}^{2} \\
\notag
&& + || \sup_{\overline{s} \in [0, \hat{s} ] }   \overline{s}^{-1} \int_{0}^{\overline{s}} s^{\frac{5}{2}} |  - \frac{1}{2} (\der_{a} tr\chi ) \la  |  ds ||_{L^{2}_{\om}}^{2} \\
\notag
&&+  || \sup_{\overline{s} \in [0, \hat{s} ] }   \overline{s}^{-1} \int_{0}^{\overline{s}} s^{\frac{5}{2}} |  - \frac{1}{2}   \ze_{a}   tr\chi \la |  ds ||_{L^{2}_{\om}}^{2}  \\
\notag
&&+  || \sup_{\overline{s} \in [0, \hat{s} ] }   \overline{s}^{-1} \int_{0}^{\overline{s}} s^{\frac{5}{2}} | [F_{L a} , \la_{\a\b} ] |_{h}  ds ||_{L^{2}_{\om}}^{2}   \\
\notag
&&+ || \sup_{\overline{s} \in [0, \hat{s} ] }   \overline{s}^{-1} \int_{0}^{\overline{s}} s^{\frac{5}{2}} | {{R_{\a}}^{\ga}}_{L a} \la_{\ga \b} + {{R_{\b}}^{\ga}}_{L a} \la_{\a \ga}  |_{h}  ds ||_{L^{2}_{\om}}^{2}   \\
&=&  I_{1} +  I_{2} +  I_{3} +  I_{4} +  I_{5} \label{I1I2I3I4I5}
\eea
where $I_{i}\; , \; i \in \{1, ..., 5\}$, are defined in order.

\subsubsection{Estimating $I_{1}$}

\begin{eqnarray*}
 I_{1} &=&  || \sup_{\overline{s} \in [0, \hat{s} ] }   \overline{s}^{-1} \int_{0}^{\overline{s}} s^{\frac{5}{2}} | 2 s^{-1} \D^{(A)}_{a} \la -tr\chi \D^{(A)}_{a} \la - \hat{\chi}_{ab} \cder_{b} \la  |  ds ||_{L^{2}_{\om}}^{2} \\
 & \lesssim& || \sup_{\overline{s} \in [0, \hat{s} ] }    \overline{s}^{-1}  \sup_{s \in [0, \overline{s}]} | s^{\frac{5}{2}}  \cder_{a} \la |  \int_{0}^{\overline{s}}  | 2 s^{-1} \ -tr\chi - \hat{\chi}_{ab}   |  ds ||_{L^{2}_{\om}}^{2} 
\end{eqnarray*}

(we remind that we were summing over $a$ with abuse of notation)

\begin{eqnarray*}
& \lesssim& || \sup_{s \in [0, \hat{s} ] }  (s^{\frac{3}{2}}  \cder_{a} \la ) ||_{L^{2}_{\om}}^{2} || \sup_{\overline{s} \in [0, \hat{s} ] }  \int_{0}^{\overline{s}}  | 2 s^{-1} \ -tr\chi - \hat{\chi}_{ab}   |  ds ||_{L^{\infty}_{\om}}^{2} \\
& \lesssim& || \sup_{\overline{s} \in [0, \hat{s} ] }  \int_{0}^{\overline{s}} 1 . | 2 s^{-1}  -tr\chi - \hat{\chi}_{ab}   |  ds ||_{L^{\infty}_{\om}}^{2}   || \sup_{s \in [0, \hat{s} ] }  (s^{\frac{3}{2}}  \cder_{a} \la ) ||_{L^{2}_{\om}}^{2} \\
& \lesssim& || \sup_{\overline{s} \in [0, \hat{s} ] }  \overline{s} \int_{0}^{\overline{s}} | 2 s^{-1} -tr\chi - \hat{\chi}_{ab}   |^{2}  ds ||_{L^{\infty}_{\om}}   || \sup_{s \in [0, \hat{s} ] }  (s^{\frac{3}{2}}  \cder_{a} \la ) ||_{L^{2}_{\om}}^{2} \\
& \lesssim& || \sup_{\overline{s} \in [0, \hat{s} ] }  \int_{0}^{\overline{s}} | 2 s^{-1}  -tr\chi - \hat{\chi}_{ab}   |^{2}  ds ||_{L^{\infty}_{\om}}   || \sup_{\overline{s} \in [0, \hat{s} ] }  \overline{s}  ||_{L^{\infty}_{\om}}   || \sup_{s \in [0, \hat{s} ] }  (s^{\frac{3}{2}}  \cder_{a} \la ) ||_{L^{2}_{\om}}^{2}\\
& \lesssim& || \sup_{\overline{s} \in [0, \hat{s} ] }  \int_{0}^{\overline{s}} [ | 2 s^{-1} -tr\chi|^{2} + | \hat{\chi}_{ab}   |^{2}]  ds ||_{L^{\infty}_{\om}}   || \sup_{\overline{s} \in [0, \hat{s} ] }  \overline{s}  ||_{L^{\infty}_{\om}}   || \sup_{s \in [0, \hat{s} ] }  (s^{\frac{3}{2}}  \cder_{a} \la ) ||_{L^{2}_{\om}}^{2}\\
\end{eqnarray*}

We know that
\bea 
| 2 s^{-1} \ -tr\chi| = O(s^{2}) \label{controlontracechi}
\eea
and
\bea
\int_{0}^{\overline{s}} | \hat{\chi}_{ab}   |^{2}]  ds  \lesssim 1
\eea
(see proposition 3.1 in [Wang]).

We get,
\bea
I_{1}  \lesssim \hat{s} || \sup_{s \in [0, \hat{s} ] }  (s^{\frac{3}{2}}  \cder_{a} \la ) ||_{L^{2}_{\om}}^{2} \label{I1}
\eea

\subsubsection{Estimating $I_{2}$}

\begin{eqnarray*}
I_{2} &\lesssim& || \sup_{\overline{s} \in [0, \hat{s} ] }   \overline{s}^{-1} \int_{0}^{\overline{s}} s^{\frac{5}{2}} |  - \frac{1}{2} (\der_{a} tr\chi ) \la  |  ds ||_{L^{2}_{\om}}^{2} \\
& \lesssim& || \sup_{\overline{s} \in [0, \hat{s} ] }   \overline{s} \la    ||_{L^{\infty}_{\om}}^{2}   || \sup_{\overline{s} \in [0, \hat{s} ] }   \overline{s}^{-1} \int_{0}^{\overline{s}} s^{\frac{3}{2}} |  - \frac{1}{2} (\der_{a} tr\chi ) |  ds ||_{L^{2}_{\om}}^{2} \\
& \lesssim& | J  |^{2}  . || \sup_{\overline{s} \in [0, \hat{s} ] }   \overline{s}^{-1} |\overline{s} \der_{a} tr\chi | \int_{0}^{\overline{s}} s^{\frac{1}{2}} ds ||_{L^{2}_{\om}}^{2} \\
& \lesssim&| J  |^{2} .  | \sup_{\overline{s} \in [0, \hat{s} ] }   \overline{s}^{-1} \overline{s}^{\frac{3}{2}} |^{2} ||  \sup_{\overline{s} \in [0, \hat{s} ] }  \overline{s} \der_{a} tr\chi  ||_{L^{2}_{\om}}^{2} \\
& \lesssim& \hat{s}  || \sup_{\overline{s} \in [0, \hat{s} ] }  \overline{s} \der_{a} tr\chi  ||_{L^{2}_{\om}}^{2} 
\end{eqnarray*}

We have $$ ||  \sup_{\overline{s} \in [0, \hat{s} ] }  \overline{s} \der_{a} tr\chi  ||_{L^{2}_{\om}}^{2}  \lesssim 1 $$ (see proposition 3.2 in [Wang]). \

We get,
\bea
I_{2} \lesssim  \hat{s} \label{I2}
\eea

\subsubsection{Estimating $I_{3}$}

\begin{eqnarray*}
I_{3} &=& || \sup_{\overline{s} \in [0, \hat{s} ] }   \overline{s}^{-1} \int_{0}^{\overline{s}} s^{\frac{5}{2}} |  - \frac{1}{2}   \ze_{a}   tr \chi \la |  ds ||_{L^{2}_{\om}}^{2}  \\
&\lesssim& || \sup_{\overline{s} \in [0, \hat{s} ] }   \overline{s}  \la ||_{L^{\infty}_{\om}}^{2}   || \sup_{\overline{s} \in [0, \hat{s} ] }   \overline{s}^{-1} \int_{0}^{\overline{s}} s^{\frac{3}{2}} |  - \frac{1}{2}   \ze_{a}   tr\chi  |  ds ||_{L^{2}_{\om}}^{2}  \\
&\lesssim& |J|^{2}  . || \sup_{\overline{s} \in [0, \hat{s} ] }   \overline{s}  tr \chi ||_{L^{\infty}_{\om}}^{2}  || \sup_{\overline{s} \in [0, \hat{s} ] }   \overline{s}^{-1} \int_{0}^{\overline{s}} s^{\frac{1}{2}}   |\ze_{a}   |   ds ||_{L^{2}_{\om}}^{2}  \\
&\lesssim& || \sup_{\overline{s} \in [0, \hat{s} ] }   \overline{s}^{-1} ( \int_{0}^{\overline{s}} s  ds )^{\frac{1}{2}} (\int_{0}^{\overline{s}}   |\ze_{a}   |^{2} ds )^{\frac{1}{2}} ||_{L^{2}_{\om}}^{2}  \\
&& \text{(from \eqref{controlontracechi})} \\
&\lesssim& || \sup_{\overline{s} \in [0, \hat{s} ] }   \overline{s}^{-1} \overline{s} ||_{L^{\infty}_{\om}}^{2}   \int_{\SSS^{2}} \int_{0}^{\hat{s}}    |\ze_{a}   |^{2} ds  d\sigma^{2}  \\
\end{eqnarray*}
\bea
I_{3} \lesssim \int_{\SSS^{2}} \int_{0}^{\hat{s}}    |\ze_{a}   |^{2} ds  d\sigma^{2} \label{I3}
\eea

\subsubsection{Estimating $I_{4}$}

\begin{eqnarray*}
I_{4} &=&  || \sup_{\overline{s} \in [0, \hat{s} ] }   \overline{s}^{-1} \int_{0}^{\overline{s}} s^{\frac{5}{2}} | [F_{L a} , \la_{\a\b} ] |_{h}  ds ||_{L^{2}_{\om}}^{2}  \\
&\lesssim& || \sup_{\overline{s} \in [0, \hat{s} ] }   \overline{s}^{-1} \int_{0}^{\overline{s}} s^{\frac{5}{2}} s^{-1}   | [ F_{L a} , s \la_{\a\b} ] |_{h}  ds ||_{L^{2}_{\om}}^{2}   \\
&\lesssim& || \sup_{\overline{s} \in [0, \hat{s} ] } |  \overline{s} \la |_{h} ||_{L^{\infty}_{\om}}^{2}   || \sup_{\overline{s} \in [0, \hat{s} ] }   \overline{s}^{-1} \overline{s}^{\frac{1}{2}} \int_{0}^{\overline{s}} 1 . s |  F_{L a}|  ds ||_{L^{2}_{\om}}^{2}   \\
&\lesssim& |J|^{2}  || \sup_{\overline{s} \in [0, \hat{s} ] }   \overline{s}^{- \frac{1}{2}}  \overline{s}^{\frac{1}{2}} (\int_{0}^{\overline{s}}  s^{2} |  F_{L a}|^{2} ds )^{\frac{1}{2}}  ||_{L^{2}_{\om}}^{2}   \\
&\lesssim&  || \sup_{\overline{s} \in [0, \hat{s} ] }  (\int_{0}^{\overline{s}}  s^{2} |  F_{L a}|^{2} ds )^{\frac{1}{2}}  ||_{L^{2}_{\om}}^{2} 
\end{eqnarray*}
\bea
I_{4} \lesssim   \int_{\SSS^{2}} \int_{0}^{\hat{s}}  s^{2} |  F_{L a}|^{2} ds d\sigma^{2} \label{I4}
\eea

\subsubsection{Estimating $I_{5}$}

\begin{eqnarray*}
I_{5} &=& || \sup_{\overline{s} \in [0, \hat{s} ] }   \overline{s}^{-1} \int_{0}^{\overline{s}} s^{\frac{5}{2}} | {{R_{\a}}^{\ga}}_{L a} \la_{\ga \b} + {{R_{\b}}^{\ga}}_{L a} \la_{\a \ga}  |_{h}  ds ||_{L^{2}_{\om}}^{2}   \\
&\les& || \sup_{\overline{s} \in [0, \hat{s} ] }   \overline{s}^{-1} \int_{0}^{\overline{s}} s^{\frac{5}{2}} |h^{\si\ga}|.|  R_{\a\si L a} \la_{\ga \b} + R_{\b \si L a} \la_{\a \ga}  |_{h}  ds ||_{L^{2}_{\om}}^{2}   \\
&\lesssim& || \sup_{\overline{s} \in [0, \hat{s} ] }  | \overline{s} \la|_{h}  ||_{L^{\infty}_{\om}}^{2}   || \sup_{\overline{s} \in [0, \hat{s} ] }   \overline{s}^{-1} \int_{0}^{\overline{s}} s^{\frac{3}{2}}  | R_{\a\b L a} |_{h}    ds ||_{L^{2}_{\om}}^{2}   \\
&\lesssim& |J|^{2} . || \sup_{\overline{s} \in [0, \hat{s} ] }   \overline{s}^{-1} \overline{s}^{\frac{1}{2}}  \int_{0}^{\overline{s}}  s | R_{\a\b L a} |_{h} ds ||_{L^{2}_{\om}}^{2}   \\
&\lesssim&  || \sup_{\overline{s} \in [0, \hat{s} ] }    \overline{s}^{- \frac{1}{2}}  \int_{0}^{\overline{s}}  1.  s | R_{\a\b L a} |_{h}   ds ||_{L^{2}_{\om}}^{2}   \\
&\lesssim&    || \sup_{\overline{s} \in [0, \hat{s} ] }    \overline{s}^{- \frac{1}{2}}  \overline{s}^{\frac{1}{2}} [ \int_{0}^{\overline{s}}  s^{2} | R_{\a\b L a} |_{h}^{2} ds  ]^{\frac{1}{2}} ||_{L^{2}_{\om}}^{2}  \\
&\lesssim& || \sup_{\overline{s} \in [0, \hat{s} ] }    [ \int_{0}^{\overline{s}}  s^{2} | R_{\a\b L a} |_{h}  ds  ]^{\frac{1}{2}} ||_{L^{2}_{\om}}^{2}  \\
\end{eqnarray*}
\bea
I_{5} \lesssim \int_{\SSS^{2}} \int_{0}^{\hat{s}}  s^{2} | R_{\a\b L a} |_{h}^{2}  ds  d\sigma^{2}  \label{I5}
\eea

\subsubsection{Estimating $|| \cder_{a} \la ||_{L^{2}(N^{-}_{\tau}(p))}$ }\

Injecting \eqref{I1}, \eqref{I2}, \eqref{I3}, \eqref{I4}, \eqref{I5} in \eqref{I1I2I3I4I5}, and then in \eqref{boundonthesupremumoftheL2normonS2ofs3squareofcderlambda},
we obtain:

\begin{eqnarray*}
&& \sup_{\overline{s} \in [0, \hat{s} ] } \int_{\SSS^{2}} \overline{s}^{3} | \D^{(A)}_{a} \la |^{2} (0, \overline{s}, \om) d\sigma^{2}  \\
&\lesssim&  \hat{s} || \sup_{s \in [0, \hat{s} ] }  |s^{\frac{3}{2}}  \cder_{a} \la|_{h} ||_{L^{2}_{\om}}^{2} + \hat{s}  + \int_{\SSS^{2}} \int_{0}^{\hat{s}}    |\ze_{a}   |^{2} ds  d\sigma^{2}  \\
&&+  \int_{\SSS^{2}} \int_{0}^{\hat{s}}  s^{2} |  F_{L a}|^{2} ds d\sigma^{2} + \int_{\SSS^{2}} \int_{0}^{\hat{s}}  s^{2} | R_{\a\b L a} |_{h}^{2}  ds  d\sigma^{2}
\end{eqnarray*}

There exists $C(p)$ (constant depending on $p$) such that for $\hat{s} \les C(p) $, we have: \

\begin{eqnarray*}
&& \sup_{\overline{s} \in [0, \hat{s} ] } \int_{\SSS^{2}} \overline{s}^{3} | \D^{(A)}_{a} \la |^{2} (0, \overline{s}, \om) d\sigma^{2}  \\
&\lesssim&  \hat{s}  + \int_{\SSS^{2}} \int_{0}^{\hat{s}}  |\ze_{a}   |^{2} ds  d\sigma^{2} + \int_{\SSS^{2}} \int_{0}^{\hat{s}}  s^{2} |  F_{L a}|^{2} ds d\sigma^{2}\\
&& +  \int_{\SSS^{2}} \int_{0}^{\hat{s}}  s^{2} | R_{\a\b L a} |_{h}^{2}  ds  d\sigma^{2}
\end{eqnarray*}

For $\hat{s} \les C(p) $, we have,
\begin{eqnarray*}
&& \hat{s} \int_{\SSS^{2}} \hat{s}^{2} | \D^{(A)}_{a} \la |^{2} (0, \hat{s}, \om) d\sigma^{2} \\
& \lesssim&  \sup_{\overline{s} \in [0, \hat{s} ] } \int_{\SSS^{2}} \overline{s}^{3} | \D^{(A)}_{a} \la |^{2} (0, \overline{s}, \om) d\sigma^{2}  \\
&\lesssim&  \hat{s}  + \int_{\SSS^{2}} \int_{0}^{\hat{s}}   |\ze_{a}   |^{2} ds  d\sigma^{2} + \int_{\SSS^{2}} \int_{0}^{\hat{s}}  s^{2} |  F_{L a}|^{2} ds d\sigma^{2} +  \int_{\SSS^{2}} \int_{0}^{\hat{s}}  s^{2} | R_{\a\b L a} |_{h}^{2}  ds  d\sigma^{2}
\end{eqnarray*}

Hence,
\bea
\notag
&& \int_{\SSS^{2}} \hat{s}^{2} | \D^{(A)}_{a} \la |^{2} (0, \hat{s}, \om) d\sigma^{2} \\
\notag
& \lesssim&  1  + \frac{1}{\hat{s}} \int_{\SSS^{2}} \int_{0}^{\hat{s}}   |\ze_{a}   |^{2}  ds  d\sigma^{2} \\
&& +  \frac{1}{\hat{s}} \int_{\SSS^{2}} \int_{0}^{\hat{s}}  s^{2} |  F_{L a}|^{2} ds d\sigma^{2} + \frac{1}{\hat{s}}  \int_{\SSS^{2}} \int_{0}^{\hat{s}}  s^{2} | R_{\a\b L a} |_{h}^{2}  ds  d\sigma^{2} \label{estimateontheLonenormofstwolamdasquaredtouseLtwomaximumprinciple}
\eea

Integrating, we obtain
\bea
\notag
&& \int_{0}^{C(p)} \int_{\SSS^{2}} \hat{s}^{2} | \D^{(A)}_{a} \la |^{2} (0, \hat{s}, \om) d\sigma^{2} d\hat{s}  \\
\notag
&\lesssim&  \int_{0}^{C(p)} 1 d\hat{s}  + \int_{0}^{C(p)} [\frac{1}{\hat{s}} \int_{\SSS^{2}} \int_{0}^{\hat{s}}   |\ze_{a}   |^{2} ds  d\sigma^{2} ] d\hat{s}  \\
\notag
&& + \int_{0}^{C(p)} [ \frac{1}{\hat{s}} \int_{\SSS^{2}} \int_{0}^{\hat{s}}  s^{2} |  F_{L a}|^{2} ds d\sigma^{2} ] d\hat{s} + \int_{0}^{C(p)} [ \frac{1}{\hat{s}}  \int_{\SSS^{2}} \int_{0}^{\hat{s}}  s^{2} | R_{\a\b L a} |_{h}^{2}  ds  d\sigma^{2} ] d\hat{s}  \\
\notag
& \lesssim& 1  + ( \int_{0}^{C(p)} [\frac{1}{\hat{s}} \int_{\SSS^{2}} \int_{0}^{\hat{s}}    |\ze_{a}   |^{2}  ds  d\sigma^{2} ]^{2} d\hat{s} )^{\frac{1}{2}}+  ( \int_{0}^{C(p)} [ \frac{1}{\hat{s}} \int_{\SSS^{2}} \int_{0}^{\hat{s}}  s^{2} |  F_{L a}|^{2} ds d\sigma^{2} ]^{2} d\hat{s} )^{\frac{1}{2}} \\
&&+  ( \int_{0}^{C(p)} [ \frac{1}{\hat{s}}  \int_{\SSS^{2}} \int_{0}^{\hat{s}} ( s^{2} | R_{\a\b L a} |_{h}^{2}  )^{\frac{1}{2}}  ds  d\sigma^{2} ]^{2} d\hat{s}   \label{estimatethatallowstoapplytheLtwomaximumprincipletocontroltheLtwonormoflamdaontehnullcone}
\eea

Using the $L^{2}$ maximum principle, and letting $t(C(p))$ be the value of $t$ for which $s(t) = C(p)$, we get in view of \eqref{areaexpression}:

\bea
\notag
&& || \D^{(A)}_{a} \la ||_{L^{2}(N^{-}_{t(C(p))}(p))} \\
\notag
&\lesssim&  1  + ( \int_{0}^{C(p)} [ \int_{\SSS^{2}} \int_{0}^{\hat{s}}   |\ze_{a}   |^{2}  ds  d\sigma^{2} ]^{2} d\hat{s} )^{\frac{1}{2}}  +  ( \int_{0}^{C(p)} [  \int_{\SSS^{2}} \int_{0}^{\hat{s}}  s^{2} |  F_{L a}|^{2} ds d\sigma^{2} ]^{2} d\hat{s} )^{\frac{1}{2}} \\
\notag
&&+  ( \int_{0}^{C(p)}  [ \int_{\SSS^{2}} \int_{0}^{\hat{s}} ( s^{2} | R_{\a\b L a} |_{h}^{2}  )^{\frac{1}{2}}  ds  d\sigma^{2} ]^{2} d\hat{s}   \\
\notag
& \lesssim& 1 + ( \int_{0}^{C(p)} [ \int_{\SSS^{2}} \int_{0}^{s_{\tau} }   |\ze_{a}   |^{2}  ds  d\sigma^{2} ]^{2} d\hat{s} )^{\frac{1}{2}}+  ( \int_{0}^{C(p)} [  \int_{\SSS^{2}} \int_{0}^{s_{\tau}}  s^{2} |  F_{L a}|^{2} ds d\sigma^{2} ]^{2} d\hat{s} )^{\frac{1}{2}} \\
\notag
&& + ( \int_{0}^{C(p)}  [ \int_{\SSS^{2}} \int_{0}^{s_{\tau}} ( s^{2} | R_{\a\b L a} |_{h}^{2}  )^{\frac{1}{2}}  ds  d\sigma^{2} ]^{2} d\hat{s}   \\
&\lesssim& 1 + ( \int_{0}^{C(p)} [ \int_{\SSS^{2}} \int_{0}^{s_{\tau}}   |\ze_{a}   |^{2} ds  d\sigma^{2} ]^{2} d\hat{s} )^{\frac{1}{2}} +[ C(p) ( F^{\frac{\pr}{\pr t}} (N^{-}_{\tau} (p)) )^{2}]^{\frac{1}{2}}   \label{estimateonLtwonormonthenullconeofserivativelamda}
\eea
(since the metric is smooth).

We have,
\bea
 ( \int_{0}^{C(p)} [ \int_{\SSS^{2}} \int_{0}^{s_{\tau}}  ( |\ze_{a}   |^{2} )ds  d\sigma^{2} ]^{2} d\hat{s} )^{\frac{1}{2}} \lesssim 1
\eea
(see proposition 3.1 in the Appendix of [Wang]) .\

Thus,
$$ || \D^{(A)}_{a} \la ||_{L^{2}(N^{-}_{t(C(p))}(p))} \lesssim 1  $$\\

Since the metric is smooth, $ \la_{\a\b}$ is smooth away from $s = 0$; we finally obtain:

\bea
|| \D^{(A)}_{a} \la ||_{L^{2}(N^{-}_{\tau}(p))} \lesssim 1 
\eea

\subsection{Estimates for $ || \cder_{a} F ||_{L^{2}(N^{-}_{\tau}(p))} $} \label{controlofthegradientofFasinCS} \

We want to control $|| {\cder}_{a} F ||_{L^{2}(N^{-}_{\tau}(p))}$. For this, as in [CS], we take the energy momentum tensor for the wave equation, after considering a full contraction with respect to the Riemannian metric $h$, and define the 2-tensor:

\bea
T_{1}^{\a\b} = h^{\mu\nu}h^{\rho\sigma} [ <\D^{(A)\a} F_{\mu\rho}, \D^{(A)\b}F_{\nu\sigma} > - \frac{1}{2} \g^{\a\b} <\D^{(A)\la} F_{\mu\rho}, \D^{(A)}_{\la} F_{\nu\sigma} > ]   \label{energy-momuntumtensorwaveequationafterfullcontractionwithrespecttoh}
\eea

Let $\hat{t}_{\b} = ( \frac{\pr}{\pr \hat{t}} )_{\b}$

We have
\begin{eqnarray*}
 T_{1}^{\a\b} \hat{t}_{\b}  &=&  h^{\mu\nu}h^{\rho\sigma}  \hat{t}_{\b}    [ <\D^{(A)\a} F_{\mu\rho}, \D^{(A)\b}F_{\nu\sigma} > - \frac{1}{2} \g^{\a\b} <\D^{(A)\la} F_{\mu\rho}, \D^{(A)}_{\la} F_{\nu\sigma} > ]
\end{eqnarray*}

We would like to compute $\der_{\a} (T_{1}^{\a\b} \hat{t}_{\b} )$. Since it is a full contraction, we can compute it by choosing a normal frame , i.e. a frame where the Christoffel sympbols vanish at that point, and hence we can get the derivatives inside the scalar product as covariant derivatives and also as gauge covariant derivatives using the fact that the scalar product is Ad-invariant, instead of partial derivatives. We obatin

\begin{eqnarray*}
&& \der_{\a} (T_{1}^{\a\b} \hat{t}_{\b} ) \\
&=&  \der_{\a} ( h^{\mu\nu}h^{\rho\sigma} \hat{t}_{\b}  )   [ <\D^{(A)\a} F_{\mu\rho}, \D^{(A)\b}F_{\nu\sigma} > - \frac{1}{2} \g^{\a\b} <\D^{(A)\la} F_{\mu\rho}, \D^{(A)}_{\la} F_{\nu\sigma} > ]  \\
&& +  h^{\mu\nu}h^{\rho\sigma} \hat{t}_{\b} [ <\D^{(A)}_{\a} \D^{(A)\a} F_{\mu\rho}, \D^{(A)\b}F_{\nu\sigma} > + <\D^{(A)\a} F_{\mu\rho}, \D^{(A)}_{\a} \D^{(A)\b}F_{\nu\sigma} >  \\
&& - \frac{1}{2} \g^{\a\b} <\D^{(A)}_{\a} \D^{(A)\la} F_{\mu\rho}, \D^{(A)}_{\la} F_{\nu\sigma} > ]   - \frac{1}{2} \g^{\a\b} <\D^{(A)\la} F_{\mu\rho}, \D^{(A)}_{\a} \D^{(A)}_{\la} F_{\nu\sigma} > ]\\
\end{eqnarray*}

We have $$- \frac{1}{2} \g^{\a\b} <\D^{(A)}_{\a} \D^{(A)\la} F_{\mu\rho}, \D^{(A)}_{\la} F_{\nu\sigma} >  = - \frac{1}{2}  <\D^{(A)\b} \D^{(A)\la} F_{\mu\rho}, \D^{(A)}_{\la} F_{\nu\sigma} >  $$

Computing $$h^{\mu\nu}h^{\rho\sigma} <\D^{(A)\b} \D^{(A)\la} F_{\mu\rho}, \D^{(A)}_{\la} F_{\nu\sigma} >  = h^{\rho\sigma}h^{\mu\nu} <\D^{(A)\b} \D^{(A)\la} F_{\nu\sigma}, \D^{(A)}_{\la} F_{\mu\rho} >  $$

We get 
\bea
\notag
&& \der_{\a} (T_{1}^{\a\b} \hat{t}_{\b} ) \\
\notag
&=&  \der_{\a} ( h^{\mu\nu}h^{\rho\sigma} \hat{t}_{\b}  )   [ <\D^{(A)\a} F_{\mu\rho}, \D^{(A)\b}F_{\nu\sigma} > - \frac{1}{2} \g^{\a\b} <\D^{(A)\la} F_{\mu\rho}, \D^{(A)}_{\la} F_{\nu\sigma} > ]  \\
\notag
&& +  h^{\mu\nu}h^{\rho\sigma} \hat{t}_{\b} [ <\D^{(A)}_{\a} \D^{(A)\a} F_{\mu\rho}, \D^{(A)\b}F_{\nu\sigma} > + <\D^{(A)\a} F_{\mu\rho}, \D^{(A)}_{\a} \D^{(A)\b}F_{\nu\sigma} >\\
\notag
&&  - <\D^{(A)\b} \D^{(A)\la} F_{\nu\sigma}, \D^{(A)}_{\la} F_{\mu\rho} >  ] \\
\notag
&=&  \der_{\a} ( h^{\mu\nu}h^{\rho\sigma} \hat{t}_{\b}  )   [ <\D^{(A)\a} F_{\mu\rho}, \D^{(A)\b}F_{\nu\sigma} > - \frac{1}{2} \g^{\a\b} <\D^{(A)\la} F_{\mu\rho}, \D^{(A)}_{\la} F_{\nu\sigma} > ]  \\
\notag
&& +  h^{\mu\nu}h^{\rho\sigma} \hat{t}_{\b} [ <\Box^{(A)}_{\g} F_{\mu\rho}, \D^{(A)\b}F_{\nu\sigma} > \\
\notag
&& + <\D^{(A)\a} F_{\mu\rho}, \D^{(A)}_{\a} \D^{(A)\b}F_{\nu\sigma} - \D^{(A)\b} \D^{(A)}_{\a} F_{\nu\sigma}>    ] \\
\notag
&=&  \der_{\a} ( h^{\mu\nu}h^{\rho\sigma} \hat{t}_{\b}  )   [ <\D^{(A)\a} F_{\mu\rho}, \D^{(A)\b}F_{\nu\sigma} > - \frac{1}{2} \g^{\a\b} <\D^{(A)\la} F_{\mu\rho}, \D^{(A)}_{\la} F_{\nu\sigma} > ]  \\
\notag
&& +  h^{\mu\nu}h^{\rho\sigma} \hat{t}_{\b} [ -2<R_{\ga\mu\rho\a}F^{\a\ga}, \D^{(A)\b}F_{\nu\sigma} > \\
\notag
&& - < R_{\mu\ga}{F_{\rho}}^{\ga}, \D^{(A)\b}F_{\nu\sigma} > - < R_{\rho\ga}{F^{\ga}}_{\mu}, \D^{(A)\b}F_{\nu\sigma} > \\
\notag
&& - 2< [{F^{\a}}_{\mu}, F_{\rho\a}], \D^{(A)\b}F_{\nu\sigma} > \\
\notag
&& + <\D^{(A)\a} F_{\mu\rho}, [{F_{\a}}^{\b}, F_{\nu\sigma}] + {{{R_{\nu}}^{\ga}}_{\a}}^{\b} F_{\ga\sigma} + {{{R_{\sigma}}^{\ga}}_{\a}}^{\b } F_{\nu\ga} >    ]  \label{derivativeTalphabetac1contractedwiththat}
\eea
(where we used \eqref{hyperbolic}).

From \eqref{definitionofLbar}, we  have
\beaa
 \hat{t}   = -  \frac{1}{2} ( \g(L, \hat{t}) \Lb + \g(L, \hat{t})^{-1} L ) 
\eeaa

we get,
\bea
T_{1}^{\hat{t} L} =  -\frac{1}{2}  \g(L, \hat{t}) T^{\Lb L } - \frac{1}{2} \g(L, \hat{t})^{-1} T^{L L } \label{T1hattL}
\eea

Computing
\beaa
&& <\D^{(A)\la} F_{\mu\rho}, \D^{(A)}_{\la} F_{\nu\sigma} > \\
&=& <\D^{(A)L} F_{\mu\rho}, \D^{(A)}_{L} F_{\nu\sigma} >  + <\D^{(A)\Lb} F_{\mu\rho}, \D^{(A)}_{\Lb} F_{\nu\sigma} >  \\
\notag
&& + <\D^{(A)a} F_{\mu\rho}, \D^{(A)}_{a} F_{\nu\sigma} >  \\
&=& \g_{L \Lb} <\D^{(A) L} F_{\mu\rho}, \D^{(A) \Lb} F_{\nu\sigma} >  + \g_{L \Lb} <\D^{(A)\Lb } F_{\mu\rho}, \D^{(A) L} F_{\nu\sigma} >  \\
&& + <\D^{(A)a} F_{\mu\rho}, \D^{(A)}_{a} F_{\nu\sigma} > \\
&=& -4 <\D^{(A) L} F_{\mu\rho}, \D^{(A) \Lb} F_{\nu\sigma} > + <\D^{(A)a} F_{\mu\rho}, \D^{(A)}_{a} F_{\nu\sigma} > 
\eeaa
(using \eqref{scalarproductLLbar}).

We get,
\bea
\notag
&& T_{1}^{\Lb L} \\
\notag
&=& h^{\mu\nu}h^{\rho\sigma} [ <\D^{(A)\Lb} F_{\mu\rho}, \D^{(A) L}F_{\nu\sigma} > - \frac{1}{2} \g^{\Lb L} <\D^{(A)\la} F_{\mu\rho}, \D^{(A)}_{\la} F_{\nu\sigma} > ]  \\
\notag
&=&  h^{\mu\nu}h^{\rho\sigma} [ <\D^{(A)\Lb} F_{\mu\rho}, \D^{(A) L}F_{\nu\sigma} > + \frac{1}{4}  (-4 <\D^{(A) L} F_{\mu\rho}, \D^{(A) \Lb} F_{\nu\sigma} > \\
\notag
&& + <\D^{(A)a} F_{\mu\rho}, \D^{(A)}_{a} F_{\nu\sigma} >  ) ]  \\
&=&  \frac{1}{4}   h^{\mu\nu}h^{\rho\sigma}  <\D^{(A)a} F_{\mu\rho}, \D^{(A)}_{a} F_{\nu\sigma} >    \label{T1LbarL}
\eea

and we have,
\bea
\notag
T_{1}^{L L} &=& h^{\mu\nu}h^{\rho\sigma} [ <\D^{(A) L} F_{\mu\rho}, \D^{(A) L}F_{\nu\sigma} > - \frac{1}{2} \g^{L L} <\D^{(A)\la} F_{\mu\rho}, \D^{(A)}_{\la} F_{\nu\sigma} > ]  \\
&=& h^{\mu\nu}h^{\rho\sigma}  < \D^{(A) L} F_{\mu\rho}, \D^{(A) L}F_{\nu\sigma} > \label{T1LL}
\eea

Injecting \eqref{T1LbarL} and \eqref{T1LL} in \eqref{T1hattL}, we obtain
\bea
\notag
T_{1}^{\hat{t} L}  &=&  -\frac{1}{2}  \g(L, \hat{t}) T^{\Lb L } - \frac{1}{2} \g(L, \hat{t})^{-1} T^{L L }  \\
\notag
&=&  -\frac{1}{2} [ \frac{1}{4} \g(L, \hat{t})    h^{\mu\nu}h^{\rho\sigma}  <\D^{(A)a} F_{\mu\rho}, \D^{(A)}_{a} F_{\nu\sigma} > \\
\notag
&& + \g(L, \hat{t})^{-1}  h^{\mu\nu}h^{\rho\sigma}  < \D^{(A) L} F_{\mu\rho}, \D^{(A) L}F_{\nu\sigma} > ] \\
&=&  -\frac{1}{2} [ \frac{1}{4} \g(L, \hat{t})      |\D^{(A)a} F|^{2} + \g(L, \hat{t})^{-1}    | \D^{(A) L} F|^{2}  ]  \label{T1alphabetathatL}
\eea

On the other hand, we have

\bea
\notag
T_{1}^{\hat{t} \hat{t}} &=& h^{\mu\nu}h^{\rho\sigma} [ <\D^{(A)\hat{t}} F_{\mu\rho}, \D^{(A)\hat{t}}F_{\nu\sigma} > - \frac{1}{2} \g^{\hat{t}\hat{t}} <\D^{(A)\la} F_{\mu\rho}, \D^{(A)}_{\la} F_{\nu\sigma} > ]  \\
\notag
&=& h^{\mu\nu}h^{\rho\sigma} [ <\D^{(A)\hat{t}} F_{\mu\rho}, \D^{(A)\hat{t}} F_{\nu\sigma} > +\frac{1}{2}  ( <\D^{(A)\hat{t}} F_{\mu\rho}, \D^{(A)}_{\hat{t}} F_{\nu\sigma} >  \\
\notag
&&   +  <\D^{(A) \hat{n}} F_{\mu\rho}, \D^{(A)}_{ \hat{n}} F_{\nu\sigma} > + <\D^{(A)a} F_{\mu\rho}, \D^{(A)}_{a} F_{\nu\sigma} >  ) ] \\
\notag
&=& h^{\mu\nu}h^{\rho\sigma} [ <\D^{(A)\hat{t}} F_{\mu\rho}, \D^{(A)\hat{t}} F_{\nu\sigma} > - \frac{1}{2}   <\D^{(A)\hat{t}} F_{\mu\rho}, \D^{(A)\hat{t}} F_{\nu\sigma} >  \\
\notag
&&   + \frac{1}{2}  <\D^{(A) \hat{n}} F_{\mu\rho}, \D^{(A) \hat{n}} F_{\nu\sigma} > + \frac{1}{2} <\D^{(A)a} F_{\mu\rho}, \D^{(A)a} F_{\nu\sigma} >  ) ] \\
\notag
&=& \frac{1}{2}  h^{\mu\nu}h^{\rho\sigma} [ <\D^{(A)\hat{t}} F_{\mu\rho}, \D^{(A)\hat{t}} F_{\nu\sigma} >  +   <\D^{(A) \hat{n}} F_{\mu\rho}, \D^{(A) \hat{n}} F_{\nu\sigma} > \\
\notag
&& +  <\D^{(A)a} F_{\mu\rho}, \D^{(A)a} F_{\nu\sigma} >   ]  \\
&=& \frac{1}{2}   [ |\D^{(A)\hat{t}} F|^{2} +   |\D^{(A) \hat{n}} F|^{2} +  |\D^{(A)a} F|^{2}     ]  \label{T1hatthatt}
\eea

Denoting by $N^{-}_{t_{p} -\tau, t}(p)$ the portion of $N^{-}(p)$ that is to the future of $\Sigma_{t_{p} -\tau}$ and to the past of $\Sigma_{t}$.  Denoting the gradient of $F$ by $\cder F$, and defining,
\bea
|\cder F|^{2} &=&  h^{\a\b} h^{\mu\ga} h^{\nu\si} <\cder_{\a} F_{\ga\si}, \cder_{\b} F_{\mu\nu} > 
\eea

Applying the divergence theorem to $T_{1}^{\a\b} \hat{t}_{\b}$  in $J^{-}(p)\cap \Sigma_{t_{p} -\tau}^{+} \cap \Sigma^{-}_{t}$, using \eqref{derivativeTalphabetac1contractedwiththat} and the fact that the metric is sufficiently smooth so that $\der_{\a} ( h^{\mu\nu}h^{\rho\sigma} \hat{t}_{\b} ) $ is finite, using \eqref{T1alphabetathatL}, \eqref{T1hatthatt}, and applying Cauchy-Schwarz, we obtain:
\bea
\notag
&&  || \D^{(A)} F ||^{2}_{L^{2} (\Sigma_{t}\cap J^{-}(p))} + || \D^{(A)}_{a} F ||^{2}_{L^{2} (N^{-}_{t_{p}-\tau, t}(p))}  \\
\notag
&\lesssim& || \D^{(A)} F ||^{2}_{L^{2} (\Sigma_{t_{p} - \tau}\cap J^{-}(p))} \\
\notag
&& + \int_{t_{p} -\tau}^{t} \int_{\Sigma_{\overline{t}}\cap J^{-}(p)} |\cder F| ( |\cder F| + |F| + |F|^{2} ) dVol_{\Sigma_{\overline{t}}\cap J^{-}(p)} d\overline{t}  \\ \label{controlafterdivergencetheoremappliedwithT1}   
\eea

We get,
\begin{eqnarray*}
&&\int_{\Sigma_{t}\cap J^{-}(p)} |\cder F|^{2} dVol_{\Sigma_{t}\cap J^{-}(p)} + || \D^{(A)}_{a} F ||^{2}_{L^{2} (N^{-}_{t_{p} -\tau, t}(p))} \\
& \lesssim& || \D^{(A)} F ||^{2}_{L^{2} (\Sigma_{t_{p} - \tau}\cap J^{-}(p))} + \int_{t_{p}-\tau}^{t} \int_{\Sigma_{\overline{t}}\cap J^{-}(p)} |\cder F|^{2} dVol_{\Sigma_{\overline{t}}\cap J^{-}(p)} d\overline{t} \\
&&  + \int_{t_{p} -\tau}^{t} \int_{\Sigma_{\overline{t}}\cap J^{-}(p)} |\cder F| (|F|^{2} + |F|  ) dVol_{\Sigma_{\overline{t}}\cap J^{-}(p)} d\overline{t} \\
& \lesssim& || \D^{(A)} F ||^{2}_{L^{2} (\Sigma_{t_{p} - \tau}\cap J^{-}(p))} + \int_{t_{p} -\tau}^{t} \int_{\Sigma_{\overline{t}}\cap J^{-}(p)} |\cder F|^{2} dVol_{\Sigma_{\overline{t}}\cap J^{-}(p)} d\overline{t}  \\
&& + \int_{t_{p} -\tau}^{t} \int_{\Sigma_{\overline{t}}\cap J^{-}(p)} ( |\cder F|^{2} + |F|^{4} + |F|^{2}  ) dVol_{\Sigma_{\overline{t}}\cap J^{-}(p)} d\overline{t} \\
&\lesssim& || \D^{(A)} F ||^{2}_{L^{2} (\Sigma_{t_{p} - \tau}\cap J^{-}(p))} + \int_{t_{p} -\tau}^{t} \int_{\Sigma_{\overline{t}}\cap J^{-}(p)} ( |F|^{4} + |F|^{2}  ) dVol_{\Sigma_{\overline{t}}\cap J^{-}(p)} d\overline{t} \\
&& + \int_{t_{p} -\tau}^{t} \int_{\Sigma_{\overline{t}}\cap J^{-}(p)} |\cder F|^{2} dVol_{\Sigma_{\overline{t}}\cap J^{-}(p)} d\overline{t} \\
\eeaa
\bea
\notag
& \lesssim& 1 +  \int_{t_{p} -\tau}^{t} \int_{\Sigma_{\overline{t}}\cap J^{-}(p)}  |F|^{4}  dVol_{\Sigma_{\overline{t}}\cap J^{-}(p)} d\overline{t}\\
\notag
&& + \int_{t_{p} -\tau}^{t} \int_{\Sigma_{\overline{t}}\cap J^{-}(p)} |F|^{2}   dVol_{\Sigma_{\overline{t}}\cap J^{-}(p)} d\overline{t}  + \int_{t_{p} -\tau}^{t} \int_{\Sigma_{\overline{t}}\cap J^{-}(p)} |\cder F|^{2} dVol_{\Sigma_{\overline{t}}\cap J^{-}(p)} d\overline{t} \\
\notag
& \lesssim& 1 +     \int_{t_{p} -\tau}^{t} ( |F|^{2}_{L^{\infty}_{\Sigma_{\overline{t}}^{p}}}   \int_{\Sigma_{\overline{t}}}  |F|^{2}  dVol_{\Sigma_{\overline{t}}\cap J^{-}(p)} ) d\overline{t}\\
\notag
&& + \int_{t_{p} -\tau}^{t} \int_{\Sigma_{\overline{t}}} |F|^{2}   dVol_{\Sigma_{\overline{t}}\cap J^{-}(p)} d\overline{t}  + \int_{ t_{p} -\tau}^{t} \int_{\Sigma_{\overline{t}}\cap J^{-}(p)} |\cder F|^{2} dVol_{\Sigma_{\overline{t}}\cap J^{-}(p)} d\overline{t} \\
\notag
&  \lesssim& 1 +     \int_{t_{p} -\tau}^{t}  |F|^{2}_{L^{\infty}_{\Sigma_{\overline{t}}^{p}}}    E_{t=0}^{\frac{\pr}{\pr t}} (\Sigma\cap J^{-}(p)) d\overline{t}  + \int_{t_{p} -\tau}^{t}E_{t=0}^{\frac{\pr}{\pr t}} (\Sigma\cap J^{-}(p))  d\overline{t} \\
&& + \int_{t_{p} -\tau}^{t} \int_{\Sigma_{\overline{t}}\cap J^{-}(p)} |\cder F|^{2} dVol_{\Sigma_{\overline{t}}\cap J^{-}(p)} d\overline{t}  \label{inequalitytocontrolgradientofF}
\eea

From \eqref{inequalitytocontrolgradientofF}, we get

\begin{eqnarray*}
&&\int_{\Sigma_{t}\cap J^{-}(p)} |\cder F|^{2} dVol_{\Sigma_{t}\cap J^{-}(p)} \\
& \les& 1 +   \int_{t_{p} -\tau}^{t}  |F|^{2}_{L^{\infty}_{\Sigma_{\overline{t}}^{p}}}  d\overline{t} + \int_{t_{p} -\tau}^{t} ( \int_{\Sigma_{\overline{t}}\cap J^{-}(p)} |\cder F|^{2} dVol_{\Sigma_{\overline{t}}\cap J^{-}(p)} ) d\overline{t} 
\end{eqnarray*}

Using Gr\"onwall lemma, we obtain

\bea
\int_{\Sigma_{t}\cap J^{-}(p)} |\cder F|^{2} dVol_{\Sigma_{t}\cap J^{-}(p)}  \lesssim C(t) +   \int_{t_{p} -\tau}^{t}  |F|^{2}_{L^{\infty}_{\Sigma_{\overline{t}}^{p}}}  d\overline{t}   \label{inequalityfortheL2normofthegradientofF}
\eea

where $C(t)$ is a finite constant that depends on $t$.\

Injecting \eqref{inequalityfortheL2normofthegradientofF} in \eqref{inequalitytocontrolgradientofF}, we have
\begin{eqnarray*}
&&|| \D^{(A)}_{a} F ||^{2}_{L^{2} (N^{-}_{t_{p} -\tau, t}(p))}  \\
&\lesssim& 1  + \int_{t_{p} -\tau}^{t}  |F|^{2}_{L^{\infty}_{\Sigma_{\overline{t}}^{p}}}   d\overline{t}  + \int_{t_{p} -\tau}^{t} \int_{\Sigma_{\overline{t}}\cap J^{-}(p)} |\cder F|^{2} dVol_{\Sigma_{\overline{t}}\cap J^{-}(p)} d\overline{t} \\
& \lesssim& 1  + \int_{t_{p} -\tau}^{t}  |F|^{2}_{L^{\infty}_{\Sigma_{\overline{t}}^{p}}}   d\overline{t} +   \int_{t_{p}-\tau}^{t} ( C(t) +   \int_{t_{p}-\tau}^{t^{*} }  |F|^{2}_{L^{\infty}_{\Sigma_{\overline{t}}^{p}}}  d\overline{t} ) dt^{*}  \\
& \lesssim & c(t) +     \int_{t_{p}-\tau}^{t}  |F|^{2}_{L^{\infty}_{\Sigma_{\overline{t}}^{p}}}   d\overline{t} +   \int_{t_{p}-\tau}^{t}    \int_{t_{p}-\tau}^{t^{*}}  |F|^{2}_{L^{\infty}_{\Sigma_{\overline{t}}^{p}}}  d\overline{t}  dt^{*} 
\end{eqnarray*}
Finally,
\bea
|| \D^{(A)}_{a} F ||^{2}_{L^{2} (N^{-}_{t_{p} -\tau, t}(p))}   \lesssim c(t) +     \int_{t_{p} -\tau}^{t}  |F|^{2}_{L^{\infty}_{\Sigma_{\overline{t}}^{p}}}   d\overline{t} +   \int_{t_{p}-\tau}^{t}    \int_{t_{p} -\tau}^{t^{*}}  |F|^{2}_{L^{\infty}_{\Sigma_{\overline{t}}^{p}}}  d\overline{t}  dt^{*}  \label{cotrolontheL2normofcderaFonthenullcone}
\eea

\subsection{The proof}\

Let $p \in \Sigma_{t_{p}}$.

Let $q \in \Sigma_{t}$ where $t_{p} - \tau \le t \le t_{p}$.

Let $\Omega_{q} = J^{-}(q) \cap J^{+}(\Sigma_{t_{p} - \tau} )$.

Let $\Sigma_{t}^{p} = \Sigma_{t}\cap J^{-}(p) $.

Using an adaptation of the Klainerman-Rodnianski parametrix in [KR1] to the Yang-Mills setting, see Appendix \eqref{KSparametrixYMsetting}, we have
\begin{eqnarray*}
&& 4\pi <\J_{\a\b}, F^{\a\b}>(q) \\
&=& - \int_{\Omega_{q}}<\la_{\a\b} \delta(u),\Box_{\g}^{(A)}F^{\a\b}> + \int_{\Omega_{q}}\delta(u)< \hat{\lap}^{(A)}\la_{\a\b} + 2 \ze_{a}\cder_{a}\la_{\a\b} + \frac{1}{2}\hat{\mu}\la_{\a\b} \\
&& + [F_{L \Lb}, \la_{\a\b}]  - \frac{1}{2}{{R_{\a}}^{\ga}}_{\Lb L}\la_{\ga\b} - \frac{1}{2}{{R_{\b}}^{\ga}}_{\Lb L}\la_{\a\ga}, F^{\a\b}> + C_{t_{p}-\tau}
\end{eqnarray*}

where $\hat{\lap}^{(A)} \la_{\a\b}$ is the induced Laplacian on the span of  $\{e_{a}\}$, $a \in \{ 1, 2 \}$, of $\la_{\a\b}$, as in \eqref{laplacianonab}, and where $ C_{t_{p}-\tau}$ depends on the value of $F$ on $\Sigma_{t_{p}-\tau}$. Let $\hat{\mu}, \hat{\nu} \in \{ \hat{t}, n, e_{a}, e_{b} \}$. Hence,

\begin{eqnarray*}
&&  |F_{\hat{\mu}\hat{\nu}}(q)|  \\
&\lesssim&  \int_{\Omega_{q}}  | <\la_{\a\b} \delta(u),\Box_{\g}^{(A)}F^{\a\b}> | + \int_{\Omega_{q}} | \delta(u)< \hat{\lap}^{(A)}\la_{\a\b},  F^{\a\b}> | \\
&& +  \int_{\Omega_{q}} | \delta(u)<\ze_{a}\cder_{a}\la_{\a\b}, F^{\a\b}>|  +  \int_{\Omega_{q}} | \delta(u)< \frac{1}{2}\hat{\mu}\la_{\a\b}, F^{\a\b}> |  \\
&& +  \int_{\Omega_{q}} | \delta(u)< [F_{L \Lb}, \la_{\a\b}], F^{\a\b}> | +  \int_{\Omega_{q}} | \delta(u)< \frac{1}{2}{{R_{\a}}^{\ga}}_{\Lb L}\la_{\ga\b}, F^{\a\b}> |  \\
&& +  \int_{\Omega_{q}} | \delta(u)<\frac{1}{2}{{R_{\b}}^{\ga}}_{\Lb L}\la_{\a\ga}, F^{\a\b}> |+  C_{t_{p}-\tau}
\end{eqnarray*}

We have,
\begin{eqnarray*}
&&\int_{\Omega_{q}}  | <\la_{\a\b}\delta(u),\Box_{\g}^{(A)}F^{\a\b}> | = \int_{\Omega_{q}}  | <\la^{\mu\nu}\delta(u),\Box_{\g}^{(A)}F_{\mu\nu}> | \\
&=& \int_{\Omega_{q}}  | <\la^{\mu\nu} \delta(u) ,  -2R_{\ga\mu\nu\a}F^{\a\ga} - R_{\mu\ga}{F_{\nu}}^{\ga} - R_{\nu\ga}{F^{\ga}}_{\mu} - 2[{F^{\a}}_{\mu}, F_{\nu\a}] >\\
&& \text{(using \eqref{hyperbolic})} \\
& \lesssim& \int_{\Omega_{q}}  | <\la^{\mu\nu} \delta(u) ,  R_{\ga\mu\nu\a}F^{\a\ga}>|  +\int_{\Omega_{q}}  | <\la\delta(u) , R_{\mu\ga} {F_{\nu}}^{\ga} >|   \\
&& + \int_{\Omega_{q}}  | <\la\delta(u) ,  R_{\nu\ga}{F^{\ga}}_{\mu}> |  + \int_{\Omega_{q}}  | <\la^{\mu\nu} \delta(u) ,  [{F^{\a}}_{\mu}, F_{\nu\a}] >|
\end{eqnarray*}

Finally,

\bea
\notag
&& \sup_{q \in \Sigma_{t}\cap J^{-}(p)} |F_{\hat{\mu}\hat{\nu}}(q)| \\
\notag
&\lesssim&  \sup_{q \in \Sigma_{t}\cap J^{-}(p)} [ \int_{\Omega_{q}}   | <\la^{\mu\nu} \delta(u) ,  R_{\ga\mu\nu\a}F^{\a\ga}>|  +  \int_{\Omega_{q}}  | <\la^{\mu\nu} \delta(u) , R_{\mu\ga}{F_{\nu}}^{\ga} >| \\
\notag
&&   +  \int_{\Omega_{q}}   | <\la^{\mu\nu} \delta(u) ,  R_{\nu\ga}{F^{\ga}}_{\mu}> |+  \int_{\Omega_{q}}  | <\la^{\mu\nu} \delta(u) ,  [{F^{\a}}_{\mu}, F_{\nu\a}] >\\
\notag
&&   +   | \int_{\Omega_{q}}   \delta(u)< \hat{\lap}^{(A)}\la_{\a\b},  F^{\a\b} > |+  \int_{\Omega_{q}}  | \delta(u)<\ze_{a}\cder_{a}\la_{\a\b}, F^{\a\b} >| \\
\notag
&&  +   | \int_{\Omega_{q}} \delta(u)< \frac{1}{2}\hat{\mu}\la_{\a\b}, F^{\a\b} > | +   \int_{\Omega_{q}}  | \delta(u)< [F_{L \Lb}, \la_{\a\b}], F^{\a\b} > | \\
\notag
&&    +  | \int_{\Omega_{q}}  \delta(u)< \frac{1}{2}{{R_{\a}}^{\ga}}_{\Lb L}\la_{\ga\b}, F^{\a\b} > |   +  \int_{\Omega_{q}}  | \delta(u)<\frac{1}{2}{{R_{\b}}^{\ga}}_{\Lb L}\la_{\a\ga}, F^{\a\b} > |  ] \\
&&   + C_{t_{p}-\tau}    \label{termstoapplyGronwall}
\eea

In what follows, we note $\Sigma_{t}\cap J^{-}(p)$ as $\Sigma_{t}^{p}$.\

\begin{lemma} \label{controllingtermslamdaF}
We have,
\beaa
&& \sup_{q \in \Sigma_{t}^{p}}  [  \int_{\Omega_{q}}  | <\la^{\mu\nu} \delta(u) ,  R_{\ga\mu\nu\a}F^{\a\ga}>|  +
\int_{\Omega_{q}}  | <\la^{\mu\nu}\delta(u) , R_{\mu\ga}{F_{\nu}}^{\ga} >| \\
&& \quad \quad  + \int_{\Omega_{q}}  | <\la^{\mu\nu} \delta(u) ,  R_{\nu\ga}{F^{\ga}}_{\mu}> |  + \int_{\Omega_{q}} | \delta(u)< \frac{1}{2}{{R_{\a}}^{\ga}}_{\Lb L}\la_{\ga\b}, F^{\a\b}> |  \\
&&  \quad \quad  + \int_{\Omega_{q}} | \delta(u)<\frac{1}{2}{{R_{\b}}^{\ga}}_{\Lb L}\la_{\a\ga}, F^{\a\b}> | ]\\
&\lesssim & \tau^{\frac{3}{2}} +  [ \int_{t_{p} - \tau}^{t}  ||F||_{L^{\infty}(\Sigma_{t}^{p})}^{2}  dt ]  \tau^{\frac{3}{2}} 
\eeaa
\end{lemma}
\begin{proof}\

We have,
\begin{eqnarray*}
&&\sup_{q \in \Sigma_{t}^{p}} \int_{\Omega_{q}}  | <\la^{\mu\nu} \delta(u) ,  R_{\ga\mu\nu\a}F^{\a\ga}>| \\
&\lesssim& \sup_{q \in \Sigma_{t}^{p}}  \int_{t_{p} - \tau}^{t}   \int_{S_{t}}   | <s \la^{\mu\nu} ,  R_{\ga\mu\nu\a} s^{-1} F^{\a\ga}>| \phi da_{t} dt \\
& \lesssim& \sup_{q \in \Sigma_{t}^{p}} ||\phi ||_{L^{\infty}(N^{-}_{t_{p} - \tau, t}(q))}  \sup_{q \in \Sigma_{t}^{p}} ||s\la||_{L^{\infty}(N^{-}_{t_{p} - \tau, t}(q))}  \sup_{q \in \Sigma_{t}^{p}} ||  R ||_{L^{\infty}(N^{-}_{t_{p} - \tau, t}(q))}  \\
&&. \sup_{q \in \Sigma_{t}^{p}}   \int_{t_{p} - \tau}^{t}    ( ||F||_{L^{\infty}(\Sigma_{t})}   \int_{S_{t}}  s^{-1} da_{t} )  dt .\\
\end{eqnarray*}

We have,
\begin{eqnarray*}
\sup_{q \in \Sigma_{t}^{p}} ||  R ||_{L^{\infty}(N^{-}_{t_{p} - \tau, t}(q))}   \lesssim 1
\end{eqnarray*}
(since the metric is smooth) 
\begin{eqnarray*}
&& ||s\la||_{L^{\infty}(N^{-}_{t_{p} - \tau, t}(q))} \lesssim |J| \\
\end{eqnarray*}
thus $$ \sup_{q \in \Sigma_{t}}  ||s\la||_{L^{\infty}(N^{-}_{t_{p} - \tau, t}(q))} \lesssim 1 $$
\begin{eqnarray*}
&& \sup_{q \in \Sigma_{t}^{p}} ||\phi ||_{L^{\infty}(N^{-}_{t_{p} - \tau, t}(q))} \lesssim 1 
\end{eqnarray*}
(since $\phi$ is smooth and bounded)
\bea
\notag
&& \sup_{q \in \Sigma_{t}^{p}}  \int_{t_{p} - \tau}^{t} ( ||F||_{L^{\infty}(\Sigma_{t}^{p})}     \int_{S_{t}}  s^{-1} da_{t} ) dt \\
& \lesssim&  \sup_{q \in \Sigma_{t}^{p}} [ \int_{t_{p} - \tau}^{t} ( ||F||_{L^{\infty}(\Sigma_{t}^{p})}^{2} ) dt ]^{\frac{1}{2}}   \sup_{q \in \Sigma_{t}^{p}} [ \int_{t_{p} - \tau}^{t} (   \int_{S_{t}}  s^{-1} da_{t} )^{2} dt ]^{\frac{1}{2}}    \label{controllingsF} 
\eea

We get,\

\beaa
\int_{\Omega_{q}}  | <\la^{\mu\nu} \delta(u) ,  R_{\ga\mu\nu\a}F^{\a\ga}>|  &\lesssim &   [ \int_{t_{p} - \tau}^{t} ( ||F||_{L^{\infty}(\Sigma_{t})}^{2} ) dt ]^{\frac{1}{2}}  \sup_{q \in \Sigma_{t}^{p}} [ \int_{t_{p} - \tau}^{t} (   \int_{S_{t}}  s^{-1} da_{t} )^{2} dt ]^{\frac{1}{2}} \\
\eeaa
\beaa
[ \int_{t_{p} - \tau}^{t} ( ||F||_{L^{\infty}(\Sigma_{t})}^{2} ) dt ]^{\frac{1}{2}}  &\lesssim& 1 +   \int_{t_{p} - \tau}^{t}  ||F||_{L^{\infty}(\Sigma_{t})}^{2}  dt 
\eeaa

Recall that $$A_{t}(p) = O((t_{p} - t)^{2}) $$ (see \eqref{areaexpression}), and $$t_{p} - t = s + o(s)$$ (see \eqref{relationsandt}).   

Thus, $$A_{t}(p) = O(s^{2}) $$

We get,

$$   \sup_{q \in \Sigma_{t}^{p}} [ \int_{t_{p} - \tau}^{t}    (    \int_{S_{t}}  s^{-1}(t) da_{t}  )^{2} dt ]^{\frac{1}{2}} \lesssim  [ \int_{t_{p} - \tau}^{t_{p}}   s_{p}(t)^{2}  dt ]^{\frac{1}{2}}   \lesssim \tau^{\frac{3}{2}}$$

Thus,

\bea
\int_{\Omega_{q}}  | <\la^{\mu\nu}\delta(u) ,  R_{\ga\mu\nu\a}F^{\a\ga}>|  \lesssim  \tau^{\frac{3}{2}} +  [ \int_{t_{p} - \tau}^{t}  ||F||_{L^{\infty}(\Sigma_{t}^{p})}^{2}  dt ]  \tau^{\frac{3}{2}} 
\eea

In the same manner this controls the terms
\begin{eqnarray*}
&&\int_{\Omega_{q}}  | <\la^{\mu\nu}\delta(u) , R_{\mu\ga}{F_{\nu}}^{\ga} >| , \int_{\Omega_{q}}  | <\la^{\mu\nu} \delta(u) ,  R_{\nu\ga}{F^{\ga}}_{\mu}> | , \\
&& \int_{\Omega_{q}} | \delta(u)< \frac{1}{2}{{R_{\a}}^{\ga}}_{\Lb L}\la_{\ga\b}, F^{\a\b}> |, \mbox{and} \int_{\Omega_{q}} | \delta(u)<\frac{1}{2}{{R_{\b}}^{\ga}}_{\Lb L}\la_{\a\ga}, F^{\a\b} > | 
\end{eqnarray*}

\end{proof}

\begin{lemma}
We have,
\beaa
\sup_{q \in \Sigma_{t}^{p}}  \int_{\Omega_{q}} | \delta(u)< \frac{1}{2}\hat{\mu}\la_{\a\b}, F^{\a\b} > |  \lesssim  \tau^{\frac{1}{2}} +  [ \int_{t_{p} - \tau}^{t}  ||F||_{L^{\infty}(\Sigma_{t}^{p})}^{2}  dt ]     \tau^{\frac{1}{2}} 
\eeaa
\end{lemma}

\begin{proof}\

The term 
\begin{eqnarray*}
&& \sup_{q \in \Sigma_{t}^{p}} \int_{\Omega_{q}} | \delta(u)< \frac{1}{2}\hat{\mu}\la_{\a\b}, F^{\a\b}> | \lesssim \sup_{q \in \Sigma_{t}^{p}}  \int_{t_{p} - \tau}^{t}  \int_{S_{t}}  | < \frac{1}{2}\hat{\mu} s \la_{\a\b} ,   s^{-1} F^{\a\b} >| \phi da_{t} dt \\
&&\lesssim  \sup_{q \in \Sigma_{t}^{p}} ||s\la ||_{L^{\infty}(N^{-}_{t_{p} - \tau, t}(q))}   \sup_{q \in \Sigma_{t}^{p}}  \int_{t_{p} - \tau}^{t}  ( ||F||_{L^{\infty}(\Sigma_{t}^{p})}    \int_{S_{t}}  | \hat{\mu} | s^{-1} \phi  da_{t} )  dt 
\end{eqnarray*}

We get,
$$  \sup_{q \in \Sigma_{t}^{p}}  \int_{\Omega_{q}} | \delta(u)< \frac{1}{2}\hat{\mu}\la_{\a\b}, F^{\a\b} > |  \lesssim    [ \int_{t_{p} - \tau}^{t} ( ||F||_{L^{\infty}(\Sigma_{t}^{p})}^{2} ) dt ]^{\frac{1}{2}}    \sup_{q \in \Sigma_{t}^{p}}  [ \int_{t_{p} - \tau}^{t} (     \int_{S_{t}} | \hat{\mu} | s^{-1}\phi  da_{t} )^{2} dt  ]^{\frac{1}{2}} $$

And  we have
$$  \sup_{q \in \Sigma_{t}^{p}}   \int_{t_{p} - \tau}^{t}  ( \int_{S_{t}}   |\hat{\mu}| s^{-1} \phi  da_{t} )^{2} dt \lesssim    \int_{t_{p} - \tau}^{t}  \sup_{q \in \Sigma_{t}^{p}} (    \int_{S_{t}}   |\hat{\mu}|^{2} \phi  da_{t} )  (     \int_{S_{t}}    s^{-2} \phi  da_{t} ) dt$$

We have $\hat{\mu} = o(s^{-1})$ ( see proposition 3.1 in [Wang] ) \

Thus,
$$  (   \sup_{q \in \Sigma_{t}^{p}}  \int_{S_{t}}   |\hat{\mu}|^{2} \phi  da_{t} )  \lesssim \sup_{q \in \Sigma_{t}^{p}} ||\phi ||_{L^{\infty}(N^{-}_{t_{p} - \tau, t}(q))}  (   \sup_{q \in \Sigma_{t}^{p}}  \int_{S_{t}}   s^{-2}  da_{t} ) \lesssim 1$$

$$ (   \sup_{q \in \Sigma_{t}^{p}}  \int_{S_{t}}    s^{-2} \phi  da_{t} ) \lesssim 1 $$
Thus,
$$ \sup_{q \in \Sigma_{t}^{p}}  [ \int_{t_{p} - \tau}^{t} (     \int_{S_{t}} | \hat{\mu} | s^{-1}\phi  da_{t} )^{2} dt  ]^{\frac{1}{2}} \lesssim \tau^{\frac{1}{2}}$$

We obtain,
\bea
\sup_{q \in \Sigma_{t}^{p}}  \int_{\Omega_{q}} | \delta(u)< \frac{1}{2}\hat{\mu}\la_{\a\b}, F^{\a\b} > |  \lesssim  \tau^{\frac{1}{2}} +  [ \int_{t_{p} - \tau}^{t}  ||F||_{L^{\infty}(\Sigma_{t}^{p})}^{2}  dt ]     \tau^{\frac{1}{2}} 
\eea

\end{proof}

Next, we want to control the term $$\sup_{q \in \Sigma_{t}^{p}}  \int_{\Omega_{q}}  | <\la^{\mu\nu} \delta(u) ,  [{F^{\a}}_{\mu}, F_{\nu\a}] > =  \sup_{q \in \Sigma_{t}^{p}} \int_{t_{p} - \tau}^{t}  \int_{S_{t}}   | <\la^{\mu\nu} \delta(u) ,  [{F^{\a}}_{\mu}, F_{\nu\a}] >  \phi  da_{t} ) dt$$

\begin{lemma} \label{ControllinglamdabracketFF}
We have,
\beaa
\sup_{q \in \Sigma_{t}^{p}}  \int_{\Omega_{q}}  | <\la^{\mu\nu} \delta(u) ,  [{F^{\a}}_{\mu}, F_{\nu\a}] > \lesssim ( \tau )^{\frac{1}{2}} + [ \int_{t_{p} - \tau}^{t} ||F||_{L^{\infty}(\Sigma_{t}^{p})}^{2}  dt ]  ( \tau )^{\frac{1}{2}} 
\eeaa
\end{lemma}

\begin{proof}\

Following the remark of Eardley and Moncrief in [EM2], we have
$$| [{F^{\a}}_{\mu}, F_{\nu\a}] | \lesssim ||F||_{L^{\infty}(S_{t})} ( |F_{L\Lb}| + |F_{La}| + |F_{Lb}| + |F_{ab}| )$$
on $N^{-}_{t_{p} -\tau, t}(q)\cap \Sigma_{t}= S_{t}$, and therefore,

\begin{eqnarray*}
&& \sup_{q \in \Sigma_{t}^{p}}  \int_{\Omega_{q}}  | <\la^{\mu\nu} \delta(u) ,  [{F^{\a}}_{\mu}, F_{\nu\a}] > \\
&\lesssim& \sup_{q \in \Sigma_{t}^{p}} \int_{t_{p} - \tau}^{t}  \int_{S_{t}}   | < s \la^{\mu\nu} \delta(u) ,  [s^{-1} {F^{\a}}_{\mu}, F_{\nu\a}] >  \phi  da_{t} ) dt \\
&  \lesssim&  \sup_{q \in \Sigma_{t}^{p}} ||\phi||_{L^{\infty}(N^{-}_{t_{p} - \tau, t}(q))} \sup_{q \in \Sigma_{t}^{p}} ||s \la ||_{L^{\infty}(N^{-}_{t_{p} - \tau, t}(q))} \\
&& \sup_{q \in \Sigma_{t}^{p}} \int_{t_{p} - \tau}^{t} ||F||_{L^{\infty}(\Sigma_{t})}    \int_{S_{t}} ( s^{-1}|F_{L\Lb}| + s^{-1} |F_{La}| + s^{-1} |F_{Lb}| + s^{-1} |F_{ab}| )   dt da_{t} \\
& \lesssim&   [ \int_{t_{p} - \tau}^{t} ( ||F||_{L^{\infty}(\Sigma_{t}^{p})}^{2} ) dt ]^{\frac{1}{2}}  \\
&& . \sup_{q \in \Sigma_{t}^{p}}    [ \int_{t_{p} - \tau}^{t} ( \int_{S_{t}} ( s^{-1}|F_{L\Lb}| + s^{-1} |F_{La}| + s^{-1} |F_{Lb}| + s^{-1} |F_{ab}| )   da_{t} )^{2}dt  ]^{\frac{1}{2}} 
\end{eqnarray*}

(since $\phi $ is smooth and bounded near $p$).\

And we have,
\beaa
&&  [ \int_{S_{t}}  ( s^{-1}|F_{L\Lb}| + s^{-1} |F_{La}| + s^{-1} |F_{Lb}| + s^{-1} |F_{ab}| )  da_{t} ]^{2} \\
& \lesssim& (  \int_{S_{t}}  ( s^{-2})   da_{t} ) (    \int_{S_{t}}  ( |F_{L\Lb}|^{2} + |F_{La}|^{2} +  |F_{Lb}|^{2} +  |F_{ab}|^{2} )   da_{t})
\eeaa

(by Cauchy-Schwarz inequality)

and $$F^{\frac{\pr}{\pr t}} (N^{-}_{\tau} (q)) =  \int_{N^{-}_{\tau} (q)} \frac{1}{8} |F_{L \Lb}|^{2} + \frac{1}{2} |F_{L e_{a}}|^{2}  +  \frac{1}{2}  |F_{L e_{b}}|^{2}  +  \frac{1}{2}  |F_{ab}|^{2}$$

Thus,
\bea
\notag
&&  \int_{t_{p} - \tau}^{t} \int_{S_{t}} ( s^{-1}|F_{L\Lb}| + s^{-1} |F_{La}| + s^{-1} |F_{Lb}| + s^{-1} |F_{ab}| ) dt da_{t}   \\
\notag
&\lesssim&     ( \int_{t_{p} - \tau}^{t}  1   dt )^{\frac{1}{2}}       (F^{\frac{\pr}{\pr t}} (N^{-}_{t_{p} -\tau, t} (q)) )^{\frac{1}{2}} \\
& \lesssim& ( \tau )^{\frac{1}{2}}  ( E_{t=0}^{\frac{\pr}{\pr t}} )^{\frac{1}{2}}   \label{Controllingsminus1Fusingfinitnessofflux}
\eea

Thus,
\bea
\sup_{q \in \Sigma_{t}^{p}}  \int_{\Omega_{q}}  | <\la^{\mu\nu} \delta(u) ,  [{F^{\a}}_{\mu}, F_{\nu\a}] > \lesssim ( \tau )^{\frac{1}{2}} + [ \int_{t_{p} - \tau}^{t} ||F||_{L^{\infty}(\Sigma_{t}^{p})}^{2}  dt ]  ( \tau )^{\frac{1}{2}} 
\eea

\end{proof}

\begin{lemma}
We have,
\bea
\notag
\sup_{q \in \Sigma_{t}\cap J^{-}(p)} \int_{\Omega_{q}}  | \delta(u)< [F_{L \Lb}, \la_{\a\b}], F^{\a\b} > |   \lesssim  ( \tau )^{\frac{1}{2}} + [ \int_{t_{p} - \tau}^{t} ||F||_{L^{\infty}(\Sigma_{t}^{p})}^{2}  dt ]  ( \tau )^{\frac{1}{2}} \\
\eea
\end{lemma}

\begin{proof}\

By same as previously, the term
\begin{eqnarray*}
&&\sup_{q \in \Sigma_{t}\cap J^{-}(p)} \int_{\Omega_{q}}  | \delta(u)< [F_{L \Lb}, \la_{\a\b}], F^{\a\b} > | \\
&\lesssim&   [ \int_{t_{p} - \tau}^{t} ( ||F||_{L^{\infty}(\Sigma_{t}^{p})}^{2} ) dt ]^{\frac{1}{2}}  \sup_{q \in \Sigma_{t}^{p}}    [ \int_{t_{p} - \tau}^{t} ( \int_{S_{t}} ( s^{-1}|F_{L\Lb}|  )   da_{t} )^{2}dt  ]^{\frac{1}{2}} \\
& \lesssim&    ( \tau )^{\frac{1}{2}}     \sup_{q \in \Sigma_{t}^{p}}  (F^{\frac{\pr}{\pr t}}  (N^{-}_{t_{p} -\tau, t} (q)) )^{\frac{1}{2}} \\
& \lesssim & ( \tau )^{\frac{1}{2}}  ( E_{t=0}^{\frac{\pr}{\pr t}} )^{\frac{1}{2}} 
\end{eqnarray*}

\end{proof}

\begin{lemma}
We have,
\beaa
\sup_{q \in \Sigma_{t}\cap J^{-}(p)} \int_{\Omega_{q}} | \delta(u)<\ze_{a}\cder_{a}\la_{\a\b}, F^{\a\b} >| \lesssim ( \tau )^{4} + [ \int_{t_{p} - \tau}^{t} ||F||_{L^{\infty}(\Sigma_{t}^{p})}^{2}  dt ]  ( \tau )^{4}
\eeaa
\end{lemma}

\begin{proof}\

The term 
\begin{eqnarray*}
&& \sup_{q \in \Sigma_{t}\cap J^{-}(p)} \int_{\Omega_{q}} | \delta(u)<\ze_{a}\cder_{a}\la_{\a\b}, F^{\a\b} >|  \\
&=&  \sup_{q \in \Sigma_{t}\cap J^{-}(p)}  \int_{t_{p} - \tau}^{t}       ( ||F||_{L^{\infty}(S_{t})}      \int_{S_{t}}  | \ze_{a}\cder_{a}\la |\phi da_{t} ) dt \\
& \lesssim&   [ \int_{t_{p} - \tau}^{t} ( ||F||_{L^{\infty}(\Sigma_{t}^{p})}^{2} ) dt ]^{\frac{1}{2}}  \sup_{q \in \Sigma_{t}\cap J^{-}(p)}  [ \int_{t_{p} - \tau}^{t} ( [ \int_{S_{t}} | \ze_{a}|^{2} \phi da_{t}] [ \int_{S_{t}} |\cder_{a}\la |^{2} \phi  da_{t} ] ) dt ]^{\frac{1}{2}} 
\end{eqnarray*}

We have $$\ze_{a} = O(s)$$ (see proposition 3.1 in [Wang]).\

Thus, $$    \int_{S_{t}} | \ze_{a}|^{2} \phi da_{t} \lesssim  \int_{S_{t}} s^{2} da_{t} \lesssim s^{4} \lesssim \tau^{4} $$

Therefore,
\begin{eqnarray*}
&& \sup_{q \in \Sigma_{t}\cap J^{-}(p)} \int_{\Omega_{q}} | \delta(u)<\ze_{a}\cder_{a}\la_{\a\b}, F^{\a\b} >| \\
& \lesssim&  [ \int_{t_{p} - \tau}^{t} ( ||F||_{L^{\infty}(\Sigma_{t}^{p})}^{2} ) dt ]^{\frac{1}{2}} \tau^{4}  \sup_{q \in \Sigma_{t}^{p} \cap J^{-}(p)} [ \int_{t_{p} - \tau}^{t}  \int_{S_{t}} |\cder_{a}\la |^{2} \phi  da_{t}  dt ]^{\frac{1}{2}} 
\end{eqnarray*}

We showed previously that $ || \D^{(A)}_{a} \la ||_{L^{2}(N^{-}_{\tau}(q))} \lesssim 1 $, thus $$\sup_{q \in \Sigma_{t}^{p} \cap J^{-}(p)} || \D^{(A)}_{a} \la ||_{L^{2}(N^{-}_{\tau}(q))} \lesssim 1$$

Finally,
\bea
\notag
\sup_{q \in \Sigma_{t}\cap J^{-}(p)} \int_{\Omega_{q}} | \delta(u)<\ze_{a}\cder_{a}\la_{\a\b}, F^{\a\b} >| &\lesssim& ( \tau )^{4} + [ \int_{t_{p} - \tau}^{t} ||F||_{L^{\infty}(\Sigma_{t}^{p})}^{2}  dt ]  ( \tau )^{4} \\
\eea

\end{proof}

We are left with the term $ \sup_{q \in \Sigma_{t}^{p}} | \int_{\Omega_{q}}  \delta(u)< \hat{\lap}^{(A)}\la_{\a\b},  F^{\a\b} > | $. We recall that $\hat{\lap}^{(A)} \la_{\a\b}$ is the induced Laplacian on the span of $\{e_{a}\}$, $a \in \{ 1, 2 \}$,

\begin{lemma}
We have,
\beaa
&& \sup_{q \in \Sigma_{t}^{p}} | \int_{\Omega_{q}}  \delta(u)< \hat{\lap}^{(A)}\la_{\a\b},  F^{\a\b} >  |  \\
&\lesssim&  1 +     \int_{t_{p}-\tau}^{t}  ||F||^{2}_{L^{\infty}(\Sigma_{t}^{p})}   dt +   \int_{t_{p}-\tau}^{t}    \int_{t_{p}-\tau}^{t}  ||F||^{2}_{L^{\infty}(\Sigma_{t}^{p})}  d\overline{t}  dt 
\eeaa
\end{lemma}

\begin{proof}

\begin{eqnarray*}
 \sup_{q \in \Sigma_{t}^{p}} | \int_{\Omega_{q}}  \delta(u)< \hat{\lap}^{(A)}\la_{\a\b},  F^{\a\b} > | &=&   \sup_{q \in \Sigma_{t}^{p}}  \int_{t_{p} - \tau}^{t_{p}}  \int_{S_{t}} < \hat{\lap}^{(A)}\la_{\a\b},  F^{\a\b} >  \phi da_{t } dt | \\
&\lesssim&    \sup_{q \in \Sigma_{t}^{p}}  \int_{t_{p} - \tau}^{t_{p}} | \int_{S_{t}}  < \hat{\lap}^{(A)}\la_{\a\b},  F^{\a\b} >    \phi da_{t} | dt  
\end{eqnarray*}

\begin{definition}

We define a restriction of the covariant derivative of $\der_{b} e_{a}$ to the span of $\{ e_{a} \}$, $a \in \{1, 2 \}$ at $q \in N^{-}(p) \backslash \{p\}$ as being $\rder_{b} e_{a}$.

\end{definition}

\begin{definition}
We define,
\bea
\rcder_{b} \rcder_{a} \la_{\a\b} =  \D^{(A)}_{b} (\D^{(A)}_{a} \la)_{\a\b}  - \D^{(A)}_{\rder_{b} e_{a}} \la_{\a\b}
\eea
whereas,  $$\D^{(A)}_{b} \D^{(A)}_{a} \la_{\a\b} =  \D^{(A)}_{b} (\D^{(A)}_{a} \la_{\a\b} ) - \D^{(A)}_{\der_{b} e_{a}} \la_{\a\b} $$
\end{definition}

We have
\bea
\notag
&& \hat{\lap}^{(A)} \la_{\a\b} \\
\notag
&=& ({\rcder}^{a} \rcder_{a}\la)(e_{\a}, e_{\b})  \\
\notag
&=&\pa^{a} [(\cder_{a}\Psi)(e_{\a}, e_{\b})  ]  + [A^{a}, (\cder_{a}\Psi)(e_{\a}, e_{\b})  ]  \\
\notag
&&- (\cder_{a} \Psi )({\der}^{a}e_{\a}, e_{\b})  - (\cder_{a} \Psi)  (e_{\a}, {\der}^{a}e_{\b})   - (\cder_{{\rder}^{a} e_{a}} \Psi)(e_{\a}, e_{\b}) 
\eea
Hence,
\beaa
&& < \hat{\lap}^{(A)}\la_{\a\b},  F^{\a\b} >  \\
\notag
&=& < \pa^{a} [(\cder_{a}\Psi)(e_{\a}, e_{\b})  ]  + [A^{a}, (\cder_{a}\Psi)(e_{\a}, e_{\b})  ],  F^{\a\b} >  \\
\notag
&& -< (\cder_{a} \Psi )({\der}^{a}e_{\a}, e_{\b}),  F^{\a\b} >  -  <(\cder_{a} \Psi)  (e_{\a}, {\der}^{a}e_{\b}),  F^{\a\b} > \\
\notag
&&  - <(\cder_{{\rder}^{a} e_{a}} \Psi)(e_{\a}, e_{\b}),  F^{\a\b} > 
\eeaa

To compute $ <(\cder_{{\rder}^{a} e_{a}} \Psi)(e_{\a}, e_{\b}),  F^{\a\b} >$, since it is a full contraction on the 2-spheres $S_{t}$, we can choose a normal frame with respect to the induced metric on $S_{t}$, i.e. a frame where the restricted covariant derivative of elements of the frame $\cder_{{\rder}^{a} e_{a}}$ vanish at that point. Hence, this term vanishes.

Whereas to the terms $$<- (\cder_{a} \Psi )({\der}^{a}e_{\a}, e_{\b}),  F^{\a\b} >$$ and $$<(\cder_{a} \Psi)  (e_{\a}, {\der}^{a}e_{\b}),  F^{\a\b} >$$ since they are full contractions with respect to the space-time metric $\g$, we can compute those with respect to a normal frame where $\der_{\a}e_{\b} = 0$ at that point. We can then express $\der^{a}$ as a combination of covariant derivatives at that frame $\der^{\a}$, and hence we get that ${\der}^{a}e_{\a}$ vanish.

Consequently,
\beaa
 < \hat{\lap}^{(A)}\la_{\a\b},  F^{\a\b} >  &=&  < \pa^{a} [(\cder_{a}\Psi)(e_{\a}, e_{\b})  ]  + [A^{a}, (\cder_{a}\Psi)(e_{\a}, e_{\b})  ],  F^{\a\b} >  
\eeaa

Similarly,
\beaa
< \cder_{a} \la_{\a\b},   {\cder}^{a}  F^{\a\b} > = < \cder_{a} \la_{\a\b},   {\pa}^{a}  F^{\a\b} +  [A^{a}, F^{\a\b}] >
\eeaa

Using the fact that the scalar product $<\;, \; >$ is Ad-invariant, we get
\bea
\notag
   < \hat{\lap}^{(A)}\la_{\a\b},  F^{\a\b} >     &=& {\rder}^{a} < \rcder_{a}  \la_{\a\b},  F^{\a\b}  > -   < \cder_{a} \la_{\a\b},   {\cder}^{a}  F^{\a\b} >
\eea

Integrating on $S_{t}$, then applying the divergence theorem, and using the fact that we have no boundary terms since it is an integral on $S_{t}$, we get \
\bea
\notag
&& | \int_{S_{t}}  \delta(u)< \hat{\lap}^{(A)}\la_{\a\b},  F^{\a\b} > \phi  da_{t} |  \\
&=& | \int_{S_{t}}  -  < \cder_{a} \la_{\a\b},   {\cder}^{a}  F^{\a\b} > \phi  da_{t} |     \label{integrationbypartsons2}  \\
\notag
&\lesssim&  ( \int_{S_{t}} | \cder_{a} \la |^{2}  \phi    da_{t})^{\frac{1}{2}}  ( \int_{S_{t}} |\cder F |^{2}  \phi    da_{t} )^{\frac{1}{2}}
\eea

Thus,
\begin{eqnarray*}
&& |  \int_{t_{p} - \tau}^{t_{p}} \int_{S_{t}}  < \hat{\lap}^{(A)}\la_{\a\b},  F_{\hat{\mu}\hat{\nu}}> |   \phi  da_{t} dt \\
&\lesssim& || \D^{(A)}_{a} \la ||_{L^{2}(N^{-}_{\tau}(q))}   ( \int_{t_{p} - \tau}^{t_{p}}    \int_{S_{t}} |\cder F|^{2}  \phi    da_{t}  dt )^{\frac{1}{2}}
\end{eqnarray*}

We proved that $ || \cder_{a} \la ||_{L^{2}(N^{-}_{\tau}(p))} \lesssim 1 $. We also have

\begin{eqnarray*}
( \int_{t_{p} - \tau}^{t_{p}}    \int_{S_{t}} |\cder F|^{2}  \phi    da_{t}  dt )^{\frac{1}{2}}  &\lesssim& 1 +   \int_{t_{p} - \tau}^{t_{p}}    \int_{S_{t}} |\cder F|^{2}  \phi    da_{t}  dt  \\
&\lesssim& 1 + || \D^{(A)}_{a} F ||^{2}_{L^{2} (N^{-}_{t_{p}-\tau, t}(q))} 
\end{eqnarray*}

We get,
\bea
\sup_{q \in \Sigma_{t}^{p}}   | \int_{\Omega_{q}}  \delta(u)< \hat{\lap}^{(A)}\la_{\a\b},  F^{\a\b} >  |  \lesssim  1 + \sup_{q \in \Sigma_{t}^{p}}   || \D^{(A)}_{a} F ||^{2}_{L^{2} (N^{-}_{t_{p}-\tau, t}(q))} 
\eea

We proved that,
$$|| \D^{(A)}_{a} F ||^{2}_{L^{2} (N^{-}_{t_{p}-\tau, t}(q))}   \lesssim c(t) +     \int_{t_{p}-\tau}^{t}  |F|^{2}_{L^{\infty}_{\Sigma_{\overline{t}}^{q}}}   d\overline{t} +   \int_{t_{p}-\tau}^{t}    \int_{t_{p}-\tau}^{t^{*}}  |F|^{2}_{L^{\infty}_{\Sigma_{\overline{t}}^{q}}}  d\overline{t}  dt^{*} $$
where $c(t)$ is a finite constant for all $t$ .
\beaa
&& \sup_{q \in \Sigma_{t}^{p}}   || \D^{(A)}_{a} F ||^{2}_{L^{2} (N^{-}_{t_{p}-\tau, t}(q))}    \\
& \lesssim & 1 +     \int_{t_{p}-\tau}^{t}  ||F||^{2}_{L^{\infty}(\Sigma_{\overline{t}}^{p})}   d\overline{t} +   \int_{t_{p}-\tau}^{t}    \int_{t_{p}-\tau}^{t^{*}}  ||F||^{2}_{L^{\infty}(\Sigma_{\overline{t}}^{p})}  d\overline{t}  dt^{*} 
\eeaa

Thus,
\bea
\notag
&& \sup_{q \in \Sigma_{t}^{p}} | \int_{\Omega_{q}}  \delta(u)< \hat{\lap}^{(A)}\la_{\a\b},  F^{\a\b} >  | \\
 &\lesssim&  1 +     \int_{t_{p}-\tau}^{t}  ||F||^{2}_{L^{\infty}(\Sigma_{\overline{t}}^{p})}   d\overline{t} +   \int_{t_{p}-\tau}^{t}    \int_{t_{p}-\tau}^{t^{*}}  ||F||^{2}_{L^{\infty}(\Sigma_{\overline{t}}^{p})}  d\overline{t}  dt^{*} 
\eea
\end{proof}
Finally, summing over all the indices we obtain,
\bea
||F||_{L^{\infty}(\Sigma_{t}^{p})} \lesssim  1  +    \int_{t_{p}-\tau}^{t}  ||F||^{2}_{L^{\infty}(\Sigma_{t}^{p})}   dt +   \int_{t_{p}-\tau}^{t}    \int_{t_{p}-\tau}^{t^{*}}  ||F||^{2}_{L^{\infty}(\Sigma_{\overline{t}}^{p})}  d\overline{t}  dt^{*} \label{Pachpatte}
\eea
Using the result of Pachpatte in [Pach], we get,
\beaa
||F||_{L^{\infty}(\Sigma_{t}^{p})} \lesssim 1 \mbox{,  for all   }  t \in [t_{p} - \tau, t_{p}] 
\eeaa
From a local existence result, at each point $p$ of the space-time, we have either the Yang-Mills fields blow up or they can be extended as solutions to the Yang-Mills equations up to that point. This combined with pointwise estimate above prove that the solutions can be extended up to the point $p$, and this can be done to any point $p$ in the space-time, under the assumption of global hyperbolicity.  \\

\section{Appendix: Kirchoff-Sobolev Parametrix for $\Box^{(A)}_{\g}F_{\mu\nu}$ }

We assume $(M, \g)$ to be globally hyperbolic, i.e. it admits a Cauchy surface $\Si$, which means a space-like hypersurface $\Si \subset M$, that is intersected precisely once by every inextendible causal curve. We also assume that the null cones are regular past the space-like hypersurface $\Si$.\

We sketch an adaptation to the Yang-Mills setting of the original construction by S. Klainerman and I. Rodnianski in [KR1] of the Kirchoff-Soboloev parametrix. Recall that the original construction presented in [KR1] was done in the context of a one tensor with values in the tangent bundle verifying the tensorial wave equation, and cannot be applied directly as it is to the the Yang-Mills equations. However, as noted in [KR1] this construction could be systematically generalized to ${\cal G}$-valued tensors of arbitrary order verifying the gauge covariant tensorial wave equations with a compatible Ad-invariant scalar product $<$ , $>$ on ${\cal G}$, leading to a representation formula suitable to present a gauge invariant proof of the global existence of Yang-Mills fields on the 4-dimensional Minkowski background. We are sketching the adaptation in this appendix so as to use it to give a proof of the global existence of Yang-Mills fields on curved backgrounds.\

\begin{definition} \label{definitionoftheparameters}
Let $p$ be a point to the future of $\Si$. The affine parameter $s$ on $N^{-}(p)$ is defined by fixing a future unit time-like vector $\T_{p}$ at $p$ and considering for every $\omega$ in $\SSS^{2}$, the null vector $l_{\om}$ in $T_{p}(M)$, such that
\bea
\g(l_{\om},\T_{p}) = 1   \label{normalisationcondition}
\eea
and associate to it the null geodesic $\ga_{\om}(s)$ such that $\ga_{\om}(0) = p$, $\frac{\partial}{\partial s}\ga_{\om}(0) = l_{\om} $, and $L = \frac{\partial}{\partial s}\ga_{\om}(s)$ where $s$ is chosen so that $\der_{L}L = 0.$ Thus, $L(s) = 1, s(p) = 0$.
\end{definition}

\begin{definition}
Let $N^{-}(p)$ be the boundary of the causal past of $p$. Let $\chi$ denote the null second fundamental form of $N^{-}(p)$, that is, for all $q \in N^{-}(p) \backslash \{p\}$,
\bea
\chi(X,Y)(q)=\g( \der_X L, Y)(q) \label{definitionofchi}
\eea
for all $X$, $Y$ in $T_{q} N^{-}(p)$.
\end{definition}

\begin{lemma} \label{definitionoftraceofchi}
$\chi$ is symmetric and thus $\chi$ is diagonalisable, moreover $\chi(L, X) = \chi (L, L) = 0$, for all $X \in T_{q} N^{-}(p)$, consequently, we can define $ tr \chi = \chi (L,L) + \chi (e_{1}, e_{1}) +  \chi (e_{2}, e_{2}) = \chi_{11} + \chi_{22} $.
\end{lemma}

\begin{proof}\

Given a point $q \in N^{-}(p) \backslash \{p\}$, we can define a null frame $\{L, \underline{L}, e_{1}, e_{2} \}$ - where $e_{1}$ and $e_{2}$ are tangent to $N^{-}(p) \cap \{s=constant\}$ 2-surfaces - that forms a basis of $T_{q}M$, such that at $q \in N^{-}(p) \backslash \{p\}$,
\bea
\g(L, L) &=& \g(\underline{L}, \underline{L}) = 0 \label{defnullframe1} \\
\g(L, \underline{L})&=& -2 \label{defnullframe2} \\
\g(e_{a}, e_{b}) &=& \delta_{ab},  \quad   a, b \in\{ 1,2\} \label{defnullframe3} \\
\g(L, e_{a}) &=& \g(\underline{L}, e_{a}) = 0,    \quad a, b \in\{ 1,2\}  \label{defnullframe4}
\eea
This null frame can be extended locally in a neighbourhood of $q \in N^{-}(p) \backslash \{p\}$ such that {$L, \underline{L}, e_{1}, e_{2}$} are vector fields in the neighbourhood and $\g(L, L)=\g(\underline{L}, \underline{L})= 0$ in the neighbourhood.\

Let, $X, Y \in T_{q}N^{-}(p)$. Since the metric is Killing, we have,
\bea
\notag
0 &=& \der_{X} \g(L, Y) = X \g(L, Y) - \g( \der_{X} L, Y) -  \g(L, \der_{X} Y) \\
&=& - \g( \der_{X} L, Y) -  \g(L, \der_{X} Y) \\
\notag
&& \text{(since $\g(L, Y) = 0$ for $Y \in T_{q}N^{-}(p)$)}
\eea
Thus, $\g( \der_{X} L, Y) = -  \g(L, \der_{X} Y)$, and since we have $[X, Y] \in T_{q} N^{-}(p)$, we get,
\beaa
0 = \g(L, [X, Y]) =  \g(L, \der_X Y - \der_Y X) 
\eeaa
which gives, $\g(L, \der_X Y) = \g(L, - \der_Y X) $, we finally get,
\beaa
\g( \der_{X} L, Y) = \g(L,  \der_Y X) 
\eeaa
Again,
\beaa
 \g( \der_{Y} L, X) = -  \g(L, \der_{Y} X)
\eeaa
(by inverting the roles of $X$ and $Y$ in before)

This gives,
\beaa
\g( \der_{X} L, Y) = \g (\der_{Y} L,  X) 
\eeaa
Consequently $\chi$ is symmetric, and hence $\chi(L, X) = \chi(X, L) = \g (\der_{L} L,  X) = 0$ because $\der_{L} L = 0$ by construction.
\end{proof}

\begin{definition}
Let $\J_{p}$ be a fixed ${\cal G}$-valued anti-symmetric 2-tensor at $p$, and let $\la_{\a\b}$ be the unique 2-tensor field along $N^{-}(p)$, that verifies the linear transport equation:
\bea\label{eq: transport}
\textbf{D}^{(A)}_{L}\la_{\a\b} + \frac{1}{2}tr \chi\la_{\a\b} = 0 \label{eq:transport} \\
(s\la_{\a\b})(p) = \J_{\a\b}(p) \label{eq:initial condition}
\eea
$\la_{\a\b}$ can be extended smoothly to be defined in a similar way in a neighborhood away from $N^{-}(p) \backslash \{p\}.$
\end{definition}

\begin{definition}
For small $\eps > 0 $, let $T_{\eps} : (1-\eps, 1+ \eps) \longmapsto M$ be the timelike geodesic from $p$ such that $T_{\eps}(1)=p$ and $T^{'}_{\eps}(1)= \T_{p}$. We define $u$, optical function, as $u_{|N^{-}(q)}= t -1$ for each $q = T_{\eps}(t)$, where $N^{-}(q)$ is the boundary of the past set of $q$, assumed to be regular.\
\end{definition}

\begin{definition} \label{definitionintegralonthenullcone}
The following integral in $\Sigma^{+}$, future of $\Sigma$, for any ${\cal G}$-valued 2-tensors $\la_{\a\b}$ and  $\La_{\a\b}$ supported in  $\Sigma^{+}$, is defined as,
\beaa
\int_{\Sigma^{+} }<\la_{\a\b}\delta(u), \La^{\a\b}>   = <\delta(u), <\la_{\a\b}, \La^{\a\b}>> 
\eeaa

 in the sense of the distribution, where $u$ is defined in a neighborhood $D_{\eps}$ of $N^{-}(p)\cap\Sigma^{+}$ as in above, and 
\bea
<\delta(u), <\la_{\a\b}, \La^{\a\b} >> = \int^{\infty}_{0}\int_{\SSS^{2}}<\la_{\a\b}, \La^{\a\b}>(t=1, s, \omega)dsd{A}_{{S}^{2}}  \label{definitionoftheintegralonthenullcone}
\eea
where $d{A}_{{S}^{2}}$ is the induced volume form on the $2$-surfaces defined by $s= constant$, and $t= 1 $. This integral depends only on $<\la_{\a\b}, \La_{\mu\nu}>$ on $N^{-}(p)$ and the normalisation condition 
\beaa
\g(l_{\omega}, \T_{p})= 1 \quad  \big( \g(l_{\omega},  l_{\omega})=0 \big).
\eeaa

Therefore, for any continuous function $f$ supported in $\Sigma^{+}$, we can define $\int_{N^{-}(p)}f$ as $<\delta(u), f>$ in the sense of the distribution.\
\end{definition}

\subsection{Computing $\int_{J^{-}(p)\cap\Sigma^{+}} <\la_{\a\b}\delta(u), \Box_{\g}^{(A)}F^{\a\b}>$ }\

Now, our goal is to compute $\int_{J^{-}(p)\cap\Sigma^{+}} <\la_{\a\b}\delta(u), \Box_{\g}^{(A)}F^{\a\b}>$ for $F$ supported in $\Sigma^{+}$.\

\begin{definition}
We define a timelike foliation near $p$, by extending locally the parameter $t$ near $p$ by starting with a fixed spacelike hypersurface $\Sigma_{1}$ passing through $p$ and orthogonal to the future unit timelike vectorfield $\T_{p}$ and considering the timelike geodesics orthogonal to $\Sigma_{1}$.
\end{definition}

\begin{definition}
We define $\Omega_{\eps} = (J^{-}(p)\cap\Sigma^{+})\backslash \cup_{t  \in [1-\eps, 1]} \Sigma_{t}$.
\end{definition}

We have 
\bea
\int_{J^{-}(p)\cap\Sigma^{+}} <\la_{\a\b}\delta(u), \Box_{\g}^{(A)}F^{\a\b}> = \lim_{\eps \to 0} \int_{\Omega_{\eps}}<\la_{\a\b}\delta(u), \Box_{\g}^{(A)}F^{\a\b}> \label{limitonomegatocoverp}
\eea

On the other hand, 
$$\int_{\Omega_{\eps}}<\la_{\a\b}\delta(u), \Box_{\g}^{(A)}F^{\a\b}> = \int_{\Omega_{\eps}}<\la_{\a\b}\delta(u), \D^{(A)\ga}\D^{(A)}_{\ga}F^{\a\b}>.$$

\begin{lemma}
 Given any two ${\cal G}$-valued tensors $K$ and $G$, since $< \; , \; >$ is Ad-invariant, we have, 
\bea
\der_{\ga}<K_{\a\b}, G^{\a\b}> = <\D_{\ga}^{(A)}K_{\a\b}, G^{\a\b}> + <K_{\a\b},\D^{(A)}_{\ga}G^{\a\b}> 
\eea
\end{lemma}

\begin{proof}\

Now, given any two ${\cal G}$-valued tensors $K_{\a\b}$ and $G_{\a\b}$, we have
\bea
\pa_{\ga}<K_{\a\b}, G^{\a\b}> - <\pa_{\ga}K_{\a\b}, G^{\a\b}> - <K_{\a\b}, \pa_{\ga}G^{\a\b}> = 0 \label{assumptionscalrproductliealgebra}
\eea

Since $<K_{\a\b}, G^{\a\b}>$ does not depend on the choice of the basis, one can choose a normal frame as in \eqref{cartanformalism} to compute \eqref{assumptionscalrproductliealgebra}. In such a frame $\pa_{\ga} K_{\a\b} =\der_{\ga}K_{\a\b}$ and $\pa_{\ga}G_{\a\b} =\der_{\ga}K_{\a\b}$,
and by abuse of notation, we will wright $\pa_{\ga}<K_{\a\b}, G^{\a\b}> $ as $\der_{\ga}<K_{\a\b}, G^{\a\b}>$. Hence, we have
$$\der_{\ga}<K_{\a\b}, G^{\a\b}> - <\der_{\ga}K_{\a\b}, G^{\a\b}> - <K_{\a\b}, \der_{\ga}G^{\a\b}> = 0$$
So we have
\begin{eqnarray*}
\der_{\ga}<K_{\a\b}, G^{\a\b}> &=& <\der_{\ga} K_{\a\b}, G^{\a\b}> + <K_{\a\b}, \der_{\ga}G^{\a\b}> \\
&=& <\der_{\ga}K_{\a\b}, G^{\a\b}> - <K_{\a\b}, [A_{\ga},G^{\a\b}]> \\
&& + <K_{\a\b},[A_{\ga},G^{\a\b}]>  + <K_{\a\b}, \der_{\ga}G^{\a\b}> \\
&=& <\der_{\ga}K_{\a\b}, G^{\a\b} > - <[K_{\a\b}, A_{\ga}],G^{\a\b}> \\
&& + <K_{\a\b},[A_{\ga},G^{\a\b}]>  + <K_{\a\b}, \der_{\ga}G^{\a\b}>
\end{eqnarray*}
 (since $<$ $,$ $>$ is Ad-invariant)
\begin{eqnarray*}
&&= <\der_{\ga}K_{\a\b}, G^{\a\b}> + <[A_{\ga},K_{\a\b}],G^{\a\b}> + <K_{\a\b},[A_{\ga},G^{\a\b}] + \der_{\ga}G^{\a\b}> \\
&& = <\D_{\ga}^{(A)}K_{\a\b}, G^{\a\b}> + <K_{\a\b},\D^{(A)}_{\ga}G^{\a\b}> \\
\end{eqnarray*}
\end{proof}

Let $D_{\eps} = ( \{ -\eps^{'} \leq u(t) \leq \eps^{'} \} \backslash \Sigma_{1-\eps}^{+} ) \cap \Sigma_{0}^{+} $, for $\eps^{'}$ chosen small enough so that $u(t)$ would be defined on $[-\eps^{'}, \eps^{'}]$. Also, recall that $\la$ is smooth in a neighborhood away from $p$, and thus, by choosing $\eps^{'}$ small enough, $\la$ is smooth in $D_{\eps}$.

Given this, we have, 
\begin{eqnarray*}
&&\int_{D_{\eps}}<\la_{\a\b}\delta(u), \D^{\ga (A)}\D_{\ga}^{(A)} F^{\a\b}> \\
&=& \int_{D_{\eps}}\der^{\ga}<\la_{\a\b}\delta(u), \cder_{\ga}F^{\a\b}> - \int_{D_{\eps}}\der^{\ga}<\cder_{\ga}(\la_{\a\b}\delta(u)), F^{\a\b}> \\
&& + \int_{D_{\eps}}<\Box_{\g}^{(A)} (\la_{\a\b} \delta(u)), F^{\a\b}>
\end{eqnarray*}

so,
\bea
\notag
&& \int_{D_{\eps}}<\la_{\a\b}\delta(u), \Box_{\g}^{(A)}F^{\a\b}> \\
\notag
&=& \int_{D_{\eps}}<\Box_{\g}^{(A)} (\la_{\a\b} \delta(u)), F^{\a\b}> \\
&&+ \int_{D_{\eps}}\der^{\ga}[<\la_{\a\b}\delta(u), \cder_{\ga}F^{\a\b}> - <\cder_{\ga}(\la_{\a\b}\delta(u)), F^{\a\b}>]
\eea
By divergence theorem,
\beaa
\notag
&&\int_{D_{\eps}}\der^{\ga}[<\la_{\a\b}\delta(u), \cder_{\ga}F^{\a\b}> - <\cder_{\ga}(\la_{\a\b}\delta(u)), F^{\a\b}>] \\
\notag
&=& - \int_{\Sigma_{t}}T^{\ga}[<\la_{\a\b}\delta(u), \cder_{\ga}F^{\a\b}> - <\cder_{\ga}(\la_{\a\b}\delta(u)), F^{\a\b}>]^{t=1-\eps}_{t=0} \\
&&+ \int_{(N^{-}(T(-\eps^{'}) )\backslash \Sigma_{1-\eps}^{+} ) \cap \Sigma_{0}^{+} }L^{\ga}[<\la_{\a\b}\delta(u), \cder_{\ga}F^{\a\b}> - <\cder_{\ga}(\la_{\a\b}\delta(u)), F^{\a\b}>] \\
&&+ \int_{(N^{-}(T(\eps^{'})) \backslash \Sigma_{1-\eps}^{+} ) \cap \Sigma_{0}^{+} }L^{\ga}[<\la_{\a\b}\delta(u), \cder_{\ga}F^{\a\b}> - <\cder_{\ga}(\la_{\a\b}\delta(u)), F^{\a\b}>]
\eeaa

Since the distributions $\delta$ and $\delta^{'}$ are supported on $N^{-}(T(0))= N^{-}(p)$, we get,
\beaa
 \int_{(N^{-}(T(-\eps^{'}) )\backslash \Sigma_{1-\eps}^{+} ) \cap \Sigma_{0}^{+} }L^{\ga}[<\la_{\a\b}\delta(u), \cder_{\ga}F^{\a\b}> - <\cder_{\ga}(\la_{\a\b}\delta(u)), F^{\a\b}>] &=& 0\\
 \int_{(N^{-}(T(\eps^{'})) \backslash \Sigma_{1-\eps}^{+} ) \cap \Sigma_{0}^{+} }L^{\ga}[<\la_{\a\b}\delta(u), \cder_{\ga}F^{\a\b}> - <\cder_{\ga}(\la_{\a\b}\delta(u)), F^{\a\b}>]&=& 0 \\
\int_{D_{\eps}}\der^{\ga}[<\la_{\a\b}\delta(u), \cder_{\ga}F^{\a\b}> - <\cder_{\ga}(\la_{\a\b}\delta(u)), F^{\a\b}>] && \\
= \int_{\Omega_{\eps}}\der^{\ga}[<\la_{\a\b}\delta(u), \cder_{\ga}F^{\a\b}> - <\cder_{\ga}(\la_{\a\b}\delta(u)), F^{\a\b}>] && \\
\eeaa

This yields to,
\beaa
\notag
&&\int_{\Omega_{\eps}}\der^{\ga}[<\la_{\a\b}\delta(u), \cder_{\ga}F^{\a\b}> - <\cder_{\ga}(\la_{\a\b}\delta(u)), F^{\a\b}>] \\
\notag
&=& - \int_{J^{-}(p)\cap \Sigma_{t}}T^{\ga}[<\la_{\a\b}\delta(u), \cder_{\ga}F^{\a\b}> - <\cder_{\ga}(\la_{\a\b}\delta(u)), F^{\a\b}>]^{t=1-\eps}_{t=0} 
\eeaa

where $\Sigma_{0}=\Sigma$, and where $T$ is defined on $\Sigma$ as being the unit normal timelike vectorfield on $\Sigma$.

We get,
\bea
\notag
 \int_{\Omega_{\eps}}<\la_{\a\b}\delta(u), \Box_{\g}^{(A)}F^{\a\b}> &=& \int_{\Omega_{\eps}}<\Box_{\g}^{(A)}(\la_{\a\b} \delta(u)), F^{\a\b}> \\
 \notag
&& - [\int_{J^{-}(p)\cap\Sigma_{t}}<\la_{\a\b}\delta(u), \cder_{T}F^{\a\b}>]^{t=1-\eps}_{t=0} \\
\notag
&& + [ \int_{J^{-}(p)\cap\Sigma_{t}}<\cder_{T}(\la_{\a\b}\delta(u)), F^{\a\b}>]^{t=1-\eps}_{t=0}  \label{afterdiv}  \\
\eea

\subsection{Computing $\int_{\Omega_{\eps}}<\Box_{\g}^{(A)}(\la_{\a\b}\delta(u)), F^{\a\b}>$ }\

Now, we would like to compute $\int_{\Omega_{\eps}}<\Box_{\g}^{(A)}(\la_{\a\b} \delta(u)), F^{\a\b}> = \int_{\Omega_{\eps}}<\Box_{\g}^{(A)}(\la\delta(u)), F>$ in \eqref{afterdiv}.\

We start by computing $\Box_{\g}^{(A)}(\la\delta(u))$. We have:
\begin{eqnarray*}
\Box_{\g}^{(A)}(\la\delta(u)) &=& \g^{\mu\nu}\D_{\mu}^{(A)}\D_{\nu}^{(A)}(\la\delta(u)) \\
&=& \g^{\mu\nu}\D_{\mu}^{(A)}(\der_{\nu}(\la\delta(u)) + [A_{\nu},(\la\delta(u))] \\
& =& \g^{\mu\nu}\D_{\mu}^{(A)}(\der_{\nu}(\la)\delta(u) + [A_{\nu},(\la\delta(u))] + \la\bigtriangledown_{\nu}(\delta(u)) \\
& =& \g^{\mu\nu}\D_{\mu}^{(A)}[(\der_{\nu}(\la) + [A_{\nu},\la])\delta(u) + \la\bigtriangledown_{\nu}(\delta(u))] \\
& =& \g^{\mu\nu}\D_{\mu}^{(A)}[\D_{\nu}^{(A)}(\la)\delta(u) + \la\delta^{'}(u)\der_{\nu}(u)] \\
&=&  \g^{\mu\nu}\cder_{\mu}\cder_{\nu}(\la)\delta(u) + \delta^{'}(u)\der_{\mu}(u)\g^{\mu\nu}\cder_{\nu}(\la) \\
&& + \delta^{'}(u)\g^{\mu\nu}\der_{\mu}(\la\der_{\nu}(u)) + \g^{\mu\nu}\delta^{''}(u)\der_{\mu}(u)\der_{\nu}(u)\la 
\end{eqnarray*}
(by the symmetry of the metric tensor). Thus,
\bea
\notag
\Box_{\g}^{(A)}(\la\delta(u)) &=&  \Box_{\g}^{(A)}(\la)\delta(u) + \delta^{'}(u)(\Box_{\g}(u)\la + 2g^{\mu\nu}\der_{\nu}u\cder_{\mu}\la) \\
&& + \delta^{''}(u)(\g^{\mu\nu}\der_{\mu}u\der_{\nu}u)\la \label{1}
\eea
Now, we want to compute $\Box_{\g}(u) = \der^{\a}\der_{\a}u$, at $q \in N^{-}(p) \backslash \{p\}.$

\begin{lemma}
We have, $\Box_{\g}u = tr \chi$, at $q \in N^{-}(p) \backslash \{p\}.$
\end{lemma}
\begin{proof} \

Now, let $\der u= \der^{\nu}u\partial_{\nu}$, defined in a neighbourhood $D_{\eps}$ of $N^{-} \cap \Sigma^{+}$. Since $$L=\frac{d}{ds}\ga_{\omega}(s)$$ where $\ga_{\omega}(s)$ is the null geodesic initiating at $p$, we have $L \in T_{q}(N^{-}(p))$ for $q \in N^{-}(p) \backslash \{p\}$ and since $u$ is constant on $N^{-}(m)$, for $m \in T_{\eps}(t)$, $t \in [1-\eps, 1+\eps]$, we have
\bea
L(u) = 0 = du(L) = \g(\der u, L) \label{Lu=0}
\eea
And since $e_{a} \in T_{q}(N^{-}(p))$, $a \in \{ 1, 2 \}$. We also have,
\bea
e_{a}(u) = du(e_{a}) = \g(\der u, e_{a}) = 0 \label{eau=0}
\eea

\eqref{Lu=0} and \eqref{eau=0} give that,
\beaa
\der u (p) = f(p)  L  
\eeaa

Hence,
\bea
\der^{\nu}u\der_{\nu} u = (\der u ) u =  f L(u) = 0 \label{eikonalequation}
\eea

At a point $p$ of the space-time, one can choose a normal frame, which means a frame such that $
\g(e_\a, e_\b) (p) =\mbox{diag}(-1,1,\ldots,1) $, and $\frac{\pa}{\pa \si} \g(e_{\a}, e_{\b})(p) = 0$. Hence, in such a frame $\Ga_{k l}^{i} = \frac{1}{2} \g^{im} (\frac{\pa \g_{mk}}{\pa x^{l}} + \frac{\pa \g_{ml}}{\pa x^{k}} - \frac{\pa \g_{kl}}{\pa x^{m}} )= 0$.
Computing in such a frame,

\beaa
( \der_{\der u} \der u )^{\ga} (p)&=& (\der_{\der^{\nu}u\partial_{\nu} } ( \der^{\mu}u\partial_{\mu}  ) )^{\ga}= ( \der^{\nu}u \der_{\nu } ( \der^{\mu}u\partial_{\mu} ) )^{\ga} \\
&=&  (\der^{\nu}u ( \der_{\nu } \der^{\mu}u ) \partial_{\mu})^{\ga} + (\der^{\nu}u  \der^{\mu}u \der_{\nu } \partial_{\mu} )^{\ga} \\
&=& \der^{\nu}u ( \der_{\nu } \der^{\mu}u ) {\delta_{\mu}}^{\ga} + \der^{\nu}u  \der^{\mu}u  \Ga^{\ga}_{\nu\mu} \\
&=& \der^{\nu}u ( \der^{\mu } \der_{\nu}u ) {\delta_{\mu}}^{\ga} + \der^{\nu}u  \der^{\mu}u  \Ga^{\ga}_{\nu\mu} \\
&& \text{(using that the metric is compatible. i.e. $\der \g = 0$)} \\
&=& \frac{1}{2} \der^{\mu } (\der_{\nu}u  \der^{\nu}u ) {\delta_{\mu}}^{\ga} \\
&=& 0 
\eeaa
(in view of \eqref{eikonalequation}).

Therefore, $\der u$ is parallel, this gives,
\beaa
\der u = c  L
\eeaa
where $c$ is a constant. We have,
\beaa
T(u) (p) = 1 = \g(\der u, \T_{p})  = \g(c L, \T_{p}) 
\eeaa
In view of \eqref{normalisationcondition}, we have,
\beaa
\g(c L, \T_{p})  = c
\eeaa
Thus, $$c = 1$$
which gives,
\bea
\der u =  L = \der^{\nu}u\partial_{\nu} \label{expressionL}
\eea
Computing,
$$\Box_{\g}u = \der^{\a}\der_{\a}u = \der^{L}\der_{L}u +  \der^{\underline{L}}\der_{\underline{L}}u +  \der^{a}\der_{a}u$$
$a =\{1, 2\}$, at $q \in N^{-}(p) \backslash \{p\}$. Therefore, at $q \in N^{-}(p) \backslash \{p\}$,\
\begin{eqnarray*}
\Box_{\g}u &=& \g^{L\a}\der_{\a}\der_{L}u +  \g^{\underline{L}\a}\der_{\a}\der_{\underline{L}}u + \g^{a\a}\der_{\a}\der_{a}u \\
 &=& -\frac{1}{2}\der_{\Lb}\der_{L}u - \frac{1}{2}\der_{L}\der_{\underline{L}}u + \der_{a}\der_{a}u \\
 &=&  -\frac{1}{2}\der_{\Lb}(\der_{L}u) + \frac{1}{2}\der_{\der_{\Lb}L} u - \frac{1}{2}\der_{L}(\der_{\underline{L}}u) + \frac{1}{2}\der_{\der_{L}\Lb} u + \der_{a}(\der_{a}u) - \der_{\der_{a}e_{a}} u 
\end{eqnarray*}
and at $q \in N^{-}(p) \backslash \{p\}$,
\begin{eqnarray*}
&&\der_{L}u = L(u) = du(L) = \g(L,L), \\
&& \der_{\underline{L}}u = \underline{L}(u) = du(\underline{L}) = \g(L,\underline{L}), \\
&& \der_{a}u = e_{a}(u) = du(e_{a}) = \g(L,e_{a}).
\end{eqnarray*}

We get 
$$\der_{\underline{L}}(\der_{L}u) = \underline{L}\g(L,L) = \g(\der_{\underline{L}}L, L) + \g(L, \der_{\underline{L}}L)$$
thus
\begin{eqnarray*}  
\der_{\underline{L}}(\der_{L}u) &=& 2g(\der_{\underline{L}}L,L) \\
\der_{L}(\der_{\underline{L}}u) &=& Lg(L, \underline{L}) = 0 \\
\der_{a}(\der_{a}u) &=& e_{a}\g(L,e_{a}) = 0
\end{eqnarray*}
Therefore,
$$ \Box_{\g}u = - \g(\der_{\underline{L}}L,L) + \frac{1}{2}\der_{\der_{\Lb}L} u + \frac{1}{2}\der_{\der_{L}\Lb} u  - \der_{\der_{a}e_{a}} u
$$
We recall that in the frame $\{ L, \Lb, e_{a}, e_{b} \}$ a vector field $X$ can be written as:
\bea
X= - \frac{1}{2} \g(X, \Lb) L -  \frac{1}{2} \g(X, L) \Lb  + \g(X, e_{a}) e_{a} \label{representationinnullframe}
\eea
Thus, taking $X = \der_{L}\underline{L}$, we get $$\der_{L}\underline{L}= - \frac{1}{2} \g(\der_{L}\underline{L}, \Lb) L -  \frac{1}{2} \g(\der_{L}\underline{L}, L) \Lb  + \g(\der_{L}\underline{L}, e_{a}) e_{a}$$

Therefore, $$\frac{1}{2}\der_{\der_{L}\Lb} u = \frac{1}{2}\g(L, \der_{L}\Lb) = -  \frac{1}{4} \g(\g(\der_{L}\underline{L}, L) \Lb, L) = \frac{1}{2} \g(\der_{L}\underline{L}, L) = 0$$

$$\frac{1}{2}\der_{\der_{\Lb}L} u = \frac{1}{2}\g(L, \der_{\Lb}L) = -  \frac{1}{4} \g(\g(\der_{\Lb}L, L) \Lb, L) = \frac{1}{2} \g(\der_{\Lb}L, L) = 0$$

We are left with $- \der_{\der_{a}e_{a}} u$\

Taking $X = \der_{a}e_{a}$, we obtain,
\bea
\notag
\der_{a}e_{a}&=& - \frac{1}{2} \g(\der_{a}e_{a}, \Lb) L -  \frac{1}{2} \g(\der_{a}e_{a}, L) \Lb  + \g(\der_{a}e_{a}, e_{b}) e_{b} \\
\notag
&=& \frac{1}{2} \g(e_{a}, \der_{a}\Lb) L +  \frac{1}{2} \g(e_{a}, \der_{a} L) \Lb  + \g(\der_{a}e_{a}, e_{b}) e_{b} \\
&=&  \frac{1}{2} tr{\underline{\chi}} L +  \frac{1}{2} tr \chi \Lb  + \g(\der_{a}e_{a}, e_{b}) e_{b} \label{derivativeaa}
\eea

Finally, $$- \der_{\der_{a}e_{a}} u = -  \g(L , \der_{a}e_{a}) = -  \g(L , \frac{1}{2} tr \chi \Lb) = tr \chi $$

This yields to
\bea
\Box_{\g}u = tr \chi
\eea
\end{proof}
Going back to \eqref{1}, we have now shown that 
\bea
&& \Box_{\g}^{(A)}(\la\de(u)) \label{2} \\
\notag
&=& \Box_{\g}^{(A)}(\la)\delta(u) + 2\delta^{'}(u)(\frac{Tr X}{2}\la + \g^{\mu\nu}\der_{\nu}u\cder_{\mu}\la) + \delta^{''}(u)(\g^{\mu\nu}\der_{\mu}u\der_{\nu}u)\la
\eea
 at $q \in N^{-}(p) \backslash \{p\}$.\

We recall that $u$ is constant on $N^{-}(p)$ and $L \in T_{q}N^{-}(p)$ for $q \in N^{-}(p) \backslash \{p\}$.\

\eqref{Lu=0} and \eqref{expressionL} yield to 
$$L(u) = 0 = \g^{\mu\nu}\der_{\mu}u\der_{\nu}u$$
at $q$. Also, 
$$ \g^{\mu\nu}\der_{\nu}u\cder_{\mu}\la = L^{\mu}\cder_{\nu}\la$$ 
(since $L = \g^{\mu\nu}\der_{\mu}u\partial_{\nu}$ gives $L^{\nu} =  \g^{\mu\nu}\der_{\mu}u$).

Thus \eqref{2} becomes,
\bea
\Box_{\g}^{(A)}(\la\de(u)) = \Box_{\g}^{(A)}(\la)\delta(u) + 2\delta^{'}(u)(\cder_{L}\la + \frac{tr \chi}{2}\la)
\eea
Hence,
\bea
\int_{\Omega_{\eps}}<\Box_{\g}^{(A)}(\la\delta(u)), F> = \int_{\Omega_{\eps}}<\Box_{\g}^{(A)}(\la)\delta(u) + 2\de^{'}(u)(\cder_{L}\la + \frac{tr \chi}{2}\la), F> 
\eea

\subsection{Evaluating $\int_{\Om_{\eps}} \de^{'}(u)<\cder_{L}\la_{\a\b} + \frac{tr \chi}{2}\la_{\a\b}, F^{\a\b}>$ }\

We have $\der_{\Lb}\de(u) = \de^{'}(u)\der_{\Lb}(u)$.
$$\der_{\Lb}(u) = du(\Lb) = \g(\der u, \Lb) = \g(L, \Lb) = -2$$

Thus, 
$$\der_{\Lb}\de(u) = -2\de^{'}(u)$$ or $$\de^{'}(u) = - \frac{1}{2}\der_{\Lb}\de(u)$$

This yields to,
\bea
\notag
&& 2\int_{\Om_{\eps}}\de^{'}(u)<\cder_{L}\la + \frac{tr \chi}{2}\la, F> \\
\notag
&=& 2(\frac{-1}{2})\int_{\Om_{\eps}}\der_{\Lb}\de(u)<\cder_{L}\la + \frac{tr \chi}{2}\la, F> \\
\notag
&=& - \int_{J^{-}(p)\cap\Sigma_{t}}\g(\Lb, T) \de(u)<\cder_{L}\la + \frac{tr \chi}{2}\la, F>  \\
\notag
&& + \int_{\Om_{\eps}}\de(u)<\cder_{\Lb}(\cder_{L}\la + \frac{tr \chi}{2}\la), F> \\
\notag
&& + \int_{\Om_{\eps}}\de(u)<\cder_{L}\la + \frac{tr \chi}{2}\la, \cder_{\Lb}F>  \\
&& + \int_{\Om_{\eps}}\de(u) \der_{\a} \Lb^{\a} <\cder_{L}\la + \frac{tr \chi}{2}\la, F> \label{derivativetransport}
\eea

 (by integration by parts)\

The integrals $$- \int_{J^{-}(p)\cap\Sigma_{t}}\g(\Lb, T)\de(u)<\cder_{L}\la + \frac{tr \chi}{2}\la, F>$$ $$\int_{\Om_{\eps}}\de(u)<\cder_{L}\la + \frac{tr \chi}{2}\la, \cder_{\Lb}F>$$ and $$\int_{\Om_{\eps}}\de(u) \der_{\a} \Lb^{\a} <\cder_{L}\la + \frac{tr \chi}{2}\la, F>$$ depend only on the values of the integrated function on $N^{-}(p)$ and on the normalisation condition $\g(L, T) (p) = 1$. As $\cder_{L}\la + \frac{tr \chi}{2}\la = 0$ on $N^{-}(p)$, and due to the presence of $\de(u)$, these three terms vanish. Thus, \eqref{derivativetransport} can be written as,
\bea
\notag
2\int_{\Om_{\eps}}\de^{'}(u)<\cder_{L}\la + \frac{tr \chi}{2}\la, F> = \int_{\Om_{\eps}}\de(u)< \cder_{\Lb} ({\cder}_{L} \la )+ \cder_{\Lb} ( \frac{tr \chi}{2}\la ), F> \\
\eea
and thus,
\bea
\notag
&& \int_{\Omega_{\eps}}<\Box_{\g}^{(A)}(\la\delta(u)), F> \\
&=& \int_{\Omega_{\eps}}<\Box_{\g}^{(A)}(\la)\delta(u), F>  \\
\notag
&& + \int_{\Omega_{\eps}}\de(u)<\cder_{\Lb} ({\cder}_{L} \la ) + \cder_{\Lb} (\frac{tr \chi}{2}\la), F> \label{Lbarderivativeoftransport}
\eea

\subsection{Evaluating $\int_{\Omega_{\eps}}<\Box_{\g}^{(A)}(\la_{\a\b})\delta(u), F^{\a\b}>$}\

Now, in its turn, we would like to compute the tensorial $$\Box_{\g}^{(A)}(\la) = \g^{\a\b}{\cder}^{2}_{\a\b}\la$$ where $$ \g^{\a\b}{\cder}^{2}_{\a\b}\la =  \g^{\a\b}\cder_{\a}\cder_{\b}\la$$ which we will distinguish from $\g^{\a\b}\cder_{\a}(\cder_{\b}\la)$ as ${\cder}^{2}_{\a\b}$ is the tensorial second order derivative defined by
\bea
{\cder}^{2}_{\a\b}\la_{\mu\nu} = (\cder_{\a}(\cder_{\b}\la))_{\mu\nu} - (\cder_{\der_{\a}e_{\b}}\la)_{\mu\nu}
\eea
where the tensorial derivative $\cder_{\a}\la$ is defined by,
\bea
(\cder_{\a}\la)(X, Y) = \cder_{\a}(\la(X, Y)) - \la(\der_{\a}X, Y) - \la (X, \der_{\a}Y)
\eea
for any $X$, $Y \in TM$.
We have,
\bea
\Box_{\g}^{(A)}\la_{\a\b} = - \frac{1}{2}{\cder}_{L\Lb}^{2}\la_{\a\b} - \frac{1}{2}{\cder}_{\Lb L}^{2}\la_{\a\b} + \de^{ab}{\cder}_{ab}^{2}\la_{\a\b} \label{boxlamda}
\eea
\begin{eqnarray*}
{\cder}_{L\Lb}^{2}\la_{\a\b} - {\cder}_{\Lb L}^{2}\la_{\a\b} &=& \der_{L}\der_{\Lb}\la_{\a\b} - \der_{\Lb}\der_{L}\la_{\a\b} + [F_{L\Lb}, \la_{\a\b}] \\
&=& {{R_{\a}}^{\ga}}_{\Lb L}\la_{\ga\b} +  {{R_{\b}}^{\ga}}_{\Lb L}\la_{\a\ga} + [F_{L\Lb}, \la_{\a\b}] 
\end{eqnarray*}
\bea
{\cder}^{2}_{\Lb L}\la_{\mu\nu}=  \cder_{\Lb}(\cder_{L}\la_{\a\b}) - \cder_{\der_{\Lb}L}\la_{\a\b} \label{LbarLlamda}
\eea

\begin{lemma}
We have,
\bea
\der_{\Lb}L = 2\ze_{a}e_{a} - 2\om L \label{LbarL}
\eea
where,
\bea
\ze_{a} &=& \frac{1}{2}\g(\der_{a}L, \Lb)  \label{defzea}\\
\om &=& - \frac{1}{4}\g(\der_{\Lb}\Lb, L)  \label{defom}
\eea

\end{lemma}

\begin{proof}
Using \eqref{representationinnullframe},
\begin{eqnarray*}
\der_{\Lb}L &=& - \frac{1}{2} \g(\der_{\Lb}L, \Lb) L -  \frac{1}{2} \g(\der_{\Lb}L, L) \Lb  + \g(\der_{\Lb}L, e_{a}) e_{a} \mbox{,         }  a \in {1, 2}\\
& =& \frac{1}{2} \g(\der_{\Lb}\Lb, L) L + 0  -  \g(L,\der_{\Lb} e_{a}) e_{a}
\end{eqnarray*}

Let $q_{m}$ be the 1-parameter group generated by $\Lb$, $\th_{r}$ the 1-parameter group generated by $e_{a}$. Let $\Om (r, m) = \th_{-r}\circ  q_{-m}\circ\th_{r}\circ  q_{m}$. We have,

\beaa
[\Lb, e_{a} ] (p) = \frac{\pa^{2}}{\pa r \pa m }\Om (0, 0 ) (p) 
\eeaa

$\th_{r}$ maps $J^{-}(p)\cap \Sigma_{t}$ into itself for all $t$, and $q_{m}$ maps $J^{-}(p)\cap \Sigma_{t}$ into say $Q_{m}(t)$, where the vector field $e_{a}$ can still be constructed to be tangent to $Q_{m}(t)$ for all $t$ in a neighborhood of $p$, and therefore $\th_{r}$ maps $Q_{m}(t)$ into itself. We get that $\th_{-r}\circ  q_{-m}\circ\th_{r}\circ  q_{m}$ maps $J^{-}(p)\cap \Sigma_{t}$ into itself for each $t$ and consequently, $\Om(r, m)$ maps $J^{-}(p)\cap \Sigma_{t}$ into itself. Therefore $[\Lb, e_{a}]$ is tangential to $J^{-}(p)\cap \Sigma_{t}$. Hence, $$\g(L, [\Lb,  e_{a}]) = 0$$

Thus,
\begin{eqnarray*}
\der_{\Lb}L &=&  \frac{1}{2} \g(\der_{\Lb}\Lb, L) L -  \g(L,\der_{e_{a}} \Lb) e_{a} \\
&=& \frac{1}{2} \g(\der_{\Lb}\Lb, L) L +  \g(\der_{e_{a}}L, \Lb) e_{a} 
\end{eqnarray*}

With the notation \eqref{defzea} and \eqref{defom}, we get \eqref{LbarL}.\\

\end{proof}

Therefore \eqref{LbarLlamda} can be written as, 
\bea
\label{LLbarlambda}
&& {\cder}_{L\Lb}^{2}\la_{\a\b} \\
\notag
&=& \cder_{\Lb}(\cder_{L}\la_{\a\b}) - 2\ze_{a}\cder_{a}\la_{\a\b} + 2\om\cder_{L}\la_{\a\b}\\
\notag
&&  + {{R_{\a}}^{\ga}}_{\Lb L}\la_{\ga\b} +  {{R_{\b}}^{\ga}}_{\Lb L}\la_{\a\ga}  + [F_{L\Lb}, \la_{\a\b}]
\eea

We define,

\bea
\notag
\lap\la_{\a\b} = \de^{ab}{\cder}^{2}_{ab}\la_{\a\b}   \label{laplaciannullcone}
\eea

Injecting \eqref{LLbarlambda} in \eqref{boxlamda}, we obtain,\

\bea
\notag
\Box_{\g}^{(A)}\la_{\a\b} &=& - \frac{1}{2}\cder_{\Lb}(\cder_{L}\la_{\a\b}) - \frac{1}{2} {\cder}^{2}_{\Lb L}\la_{\a\b} + \ze_{a}\cder_{a}\la_{\a\b} - \om\cder_{L}\la_{\a\b} \\
&& - \frac{1}{2}{{R_{\a}}^{\ga}}_{\Lb L}\la_{\ga\b}  - \frac{1}{2}{{R_{\b}}^{\ga}}_{\Lb L}\la_{\a\ga} + [F_{L \Lb}, \la_{\a\b}] + \lap\la_{\a\b}  \label{beforelastexpressionboxlamda}
\eea

Recall \eqref{LbarL}, then \eqref{LbarLlamda} can be written as,
\bea
{\cder}^{2}_{\Lb L}\la_{\a\b} = \cder_{\Lb}(\cder_{L}\la_{\a\b}) - 2\ze_{a} \cder_{e_{a}} \la_{\a\b}  + 2\om  \cder_{L }\la_{\a\b} \label{secondexpressionLbarLlamda}
\eea

Injecting  \eqref{secondexpressionLbarLlamda} in \eqref{beforelastexpressionboxlamda}, we get
\bea
&& \Box_{\g}^{(A)}\la_{\a\b}  \label{lastexpressionboxlambda}  \\
\notag
&=& - \cder_{\Lb}(\cder_{L}\la_{\a\b}) + \ze_{a}\cder_{a}\la_{\a\b} - \om\cder_{L}\la_{\a\b} - \frac{1}{2}{{R_{\a}}^{\ga}}_{\Lb L}\la_{\ga\b} \label{lastexpressionboxlambda} \\
\notag
&&- \frac{1}{2}{{R_{\b}}^{\ga}}_{\Lb L}\la_{\a\ga} + [F_{L \Lb}, \la_{\a\b}] + \lap\la_{\a\b}  + \frac{1}{2}\cder_{\der_{\Lb}L}\la_{\a\b} 
\eea

\subsection{Revisiting $\int_{\Omega_{\eps}}<\Box_{\g}^{(A)}(\la_{\a\b}\delta(u)), F^{\a\b}>$}\

We showed \eqref{Lbarderivativeoftransport} that is,
\begin{eqnarray*}
\int_{\Omega_{\eps}}<\Box_{\g}^{(A)}(\la\delta(u)), F> &=& \int_{\Omega_{\eps}}<\Box_{\g}^{(A)}(\la)\delta(u), F> \\
&& + \int_{\Omega_{\eps}}\de(u)<\cder_{\Lb} (\cder_{L}\la )+ \cder_{\Lb}(\frac{tr X}{2}\la), F> \\
\cder_{\Lb}(\frac{tr\chi}{2}\la_{\a\b}) &=& \der_{\Lb}(\frac{tr \chi}{2})\la_{\a\b} + \frac{tr\chi}{2}\cder_{\Lb}\la_{\a\b} 
\end{eqnarray*}

Recall \eqref{LbarLlamda}, we also have $\cder_{L}\la_{\a\b} + \frac{tr\chi}{2}\la_{\a\b} = 0$ at $q \in N^{-}(p) \backslash \{p\}$, using  \eqref{lastexpressionboxlambda} we obtain,
\begin{eqnarray*}
&& \int_{\Omega_{\eps}}<\Box_{\g}^{(A)}(\la_{\a\b} \delta(u)), F^{\a\b} > \\
&=& \int_{\Omega_{\eps}} \delta(u) <\ze_{a}\cder_{a}\la_{\a\b} - \om\cder_{L} \la_{\a\b} - \frac{1}{2}{{R_{\a}}^{\ga}}_{\Lb L}\la_{\ga\b} - \frac{1}{2}{{R_{\b}}^{\ga}}_{\Lb L}\la_{\a\ga} \\
&& \quad+ [F_{L \Lb}, \la_{\a\b}] + \lap\la_{\a\b} + \frac{1}{2}tr \chi \cder_{\Lb}\la_{\a\b}    + \frac{1}{2}\cder_{\der_{\Lb}L}\la_{\a\b} + \der_{\Lb}(\frac{tr\chi}{2})\la_{\a\b} , F^{\a\b}> \\
&=&  \int_{\Omega_{\eps}} \delta(u)<\ze_{a}\cder_{a}\la_{\a\b} - 2 \om\cder_{L} \la_{\a\b} + \om\cder_{L} \la_{\a\b} -\frac{1}{2} tr\underline{\chi} \cder_{L}\la_{\a\b}  \\
&& \quad + \frac{1}{2} tr \underline{\chi} \cder_{L}\la_{\a\b} - \frac{1}{2}{{R_{\a}}^{\ga}}_{\Lb L}\la_{\ga\b}  - \frac{1}{2}{{R_{\b}}^{\ga}}_{\Lb L}\la_{\a\ga} + [F_{L \Lb}, \la_{\a\b}] + \lap\la_{\a\b}   \\
&& \quad +  \frac{1}{2}tr \chi \cder_{\Lb}\la_{\a\b} + \frac{1}{2}\cder_{\der_{\Lb}L}\la_{\a\b} + \der_{\Lb}(\frac{tr\chi}{2})\la_{\a\b} , F^{\a\b}>
\end{eqnarray*}

Let $\mu$ be the mass aspect function defined by
\bea
\mu = \der_{\Lb}tr \chi + \frac{1}{2}tr \chi tr \underline{\chi} + 2\om tr \chi
\eea
We have, 
\bea
\notag
&& \int_{\Omega_{\eps}}<\Box_{\g}^{(A)}(\la\delta(u)), F>\\
\notag
&=& \int_{\Omega_{\eps}} \delta(u)<\ze_{a}\cder_{a}\la_{\a\b}  +  \om\cder_{L} \la_{\a\b} + \frac{1}{2} tr\underline{\chi} \cder_{L}\la_{\a\b} - \frac{1}{2}{{R_{\a}}^{\ga}}_{\Lb L}\la_{\ga\b}  - \frac{1}{2}{{R_{\b}}^{\ga}}_{\Lb L}\la_{\a\ga} \\
\notag
&&\quad + [F_{L \Lb}, \la_{\a\b}] + \lap\la_{\a\b} +  \frac{1}{2}tr \chi \cder_{\Lb}\la_{\a\b} + \frac{1}{2}\cder_{\der_{\Lb}L}\la_{\a\b} + \frac{1}{2}\mu\la_{\a\b}, F^{\a\b}> \\
\notag
&=& \int_{\Omega_{\eps}} \delta(u) <2 \ze_{a} \cder_{a} \la_{\a\b}  + \frac{1}{2} tr \underline{\chi} \cder_{L}\la_{\a\b} - \frac{1}{2}{{R_{\a}}^{\ga}}_{\Lb L}\la_{\ga\b} - \frac{1}{2}{{R_{\b}}^{\ga}}_{\Lb L}\la_{\a\ga}  \\
&&\quad + [F_{L \Lb}, \la_{\a\b}]  + \lap\la_{\a\b} +  \frac{1}{2}tr \chi \cder_{\Lb}\la_{\a\b} + \frac{1}{2}\mu\la_{\a\b}, F^{\a\b}> \label{boxlambdadeltaF} 
\eea
(where we used \eqref{LbarL})

\subsection{Estimating $\lim_{\eps \to 0} |- \int_{J^{-}(p)\cap\Sigma_{1 - \eps}} <\la_{\a\b}\delta(u), \cder_{T}F^{\a\b}>| $ }\

Computing $\de^{ab}{\cder}^{2}_{ab}\la_{\a\b}$ we have,
\bea
{\cder}^{2}_{ab}\la_{\a\b} = \cder_{a}(\cder_{b}\la_{\a\b}) - \cder_{\der_{a}e_{b}}\la_{\a\b}  \label{der_ab^2}
\eea

\begin{definition}

We define a restriction of the covariant derivative to the span of $\{ e_{a} \}$, $a \in \{1, 2 \}$ at $q \in N^{-}(p) \backslash \{p\}$ as being,
\bea
\rder_{a}e_{b} = \der_{a}e_{b} - \frac{1}{2}\chi_{ab}\Lb  - \frac{1}{2}\underline{\chi}_{ab} L
\eea

where
\bea
\underline{\chi}_{ab} = \g(\der_{a} \Lb, e_{b})
\eea
We have $$\g(\rder_{a}e_{b}, L) = \g(\der_{a}e_{b}, L) - \frac{1}{2}\chi_{ab}\g(\Lb, L)  - \frac{1}{2}\underline{\chi}_{ab}\g(L, L) = \g(\der_{a}e_{b}, L) + \chi_{ab}$$

We have $\g(e_{b}, L) = 0$ along $N^{-}(p)$ and since $e_{a}$ is tangent to $N^{-}(p)$ at $q$, we get $$e_{a}\g(e_{b}, L) = 0 = \g(\der_{a}e_{b}, L) + \g(e_{b}, \der_{a}L)$$

so $$\g(\der_{a}e_{b}, L) = - \g(e_{b}, \der_{a}L) = - \chi_{ab} $$ and hence $$\g(\rder_{a}e_{b}, L) = 0$$

Therefore, $\rder_{a}e_{b}$ is tangent to $N^{-}(p)$.\

Computing, $$\g(\rder_{a}e_{b}, \Lb) = \g(\der_{a}e_{b}, \Lb) - \frac{1}{2}\chi_{ab}\g(\Lb, \Lb)  - \frac{1}{2}\underline{\chi}_{ab}\g(L, \Lb) = \g(\der_{a}e_{b}, \Lb) + \underline{\chi}_{ab}$$

Similarly, we get $$\g(\der_{a}e_{b}, \Lb)  = - \underline{\chi}_{ab} $$ and therefore $$\g(\rder_{a}e_{b}, \Lb) = 0$$

Finally, we get that $\rder_{a}e_{b}$ is in the span of $\{ e_{a} \}$, $a \in \{1, 2 \}$. \ 

\end{definition}

Going back to ${\cder}^{2}_{ab}\la_{\a\b}$, we have $$\cder_{\der_{a}e_{b}}\la_{\a\b} = \cder_{\rder_{a}e_{b}+\frac{1}{2}\chi_{ab}\Lb +\frac{1}{2}\underline{\chi}_{ab}L }\la_{\a\b} =  \cder_{\rder_{a}e_{b}}\la_{\a\b} +  \frac{1}{2}\chi_{ab}\cder_{\Lb}\la_{\a\b} +  \frac{1}{2}\underline{\chi}_{ab}\cder_{L}\la_{\a\b} $$

Hence, we have at $q \in N^{-}(p) \backslash \{p\}$, $$\de^{\a\b}{\cder}^{2}_{ab}\la_{\a\b} = \de^{\a\b}\cder_{a}(\cder_{b}\la_{\a\b}) - \de^{ab}\cder_{\rder_{a}e_{b}}\la_{\a\b} - \frac{1}{2}tr \chi\cder_{\Lb}\la_{\a\b} - \frac{1}{2}tr \underline{\chi}\cder_{L}\la_{\a\b}$$

$\rder_{a}e_{b}$ is in the span of $e_{a} \in T_{q}N^{-}(p)$ for $a \in \{1, 2 \}$, we therefore define
\bea
\hat{\D}^{2}_{ab}\la_{\a\b} = \cder_{a}(\cder_{b}\la_{\a\b}) - \cder_{\rder_{a}e_{b}}\la_{\a\b}
\eea
where $\hat{\D}$ is the restriction of the gauge covariant derivative $\cder$ along to the span of $e_{a} \in T_{q}N^{-}(p)$, $a \in \{1, 2 \}$. We get
\bea
\notag
\de^{ab}{\cder}^{2}_{ab}\la_{\a\b} &=& \de^{ab}\hat{\D}^{2}_{ab}\la_{\a\b} - \frac{1}{2}\de^{ab}\chi_{ab}\cder_{\Lb}\la_{\a\b}   - \frac{1}{2}\de^{ab}\underline{\chi}_{ab}\cder_{L}\la_{\a\b} \\
&=& \de^{ab}\hat{\D}^{2}_{ab}\la_{\a\b} - \frac{1}{2}tr \chi \cder_{\Lb}\la_{\a\b} - \frac{1}{2}tr \underline{\chi} \cder_{L}\la_{\a\b} \label{laplaciannullcone}
\eea

\begin{definition}

Let $\hat{\lap}^{(A)} \la_{\a\b}$ be the induced Laplacian on the span of  $\{e_{a}\}$, $a \in \{ 1, 2 \}$ defined by,
\bea
\notag
\hat{\lap}^{(A)} \la_{\a\b} = \de^{ab}\hat{\D}^{2}_{ab}\la_{\a\b} =  \lap^{(A)} \la_{\a\b} + \frac{1}{2}tr\underline{\chi} \cder_{L}\la_{\a\b} +  \frac{1}{2}tr\chi \cder_{\Lb} \la_{\a\b}  \label{laplacianonab} \\
\eea

\end{definition}

We obtain after injecting \eqref{laplacianonab} and \eqref{boxlambdadeltaF} in \eqref{afterdiv},
\begin{eqnarray*}
&&\int_{\Omega_{\eps}}<\la\delta(u),\Box_{\g}^{(A)}F>_{\g} \\
&=& \int_{\Omega_{\eps}}< \hat{\lap}^{(A)}\la_{\a\b}  + 2 \ze_{a}  \cder_{a}  \la_{\a\b} + \frac{1}{2}\mu\la_{\a\b} + [F_{L \Lb}, \la_{\a\b}]  \\
&& - \frac{1}{2}{{R_{\a}}^{\ga}}_{\Lb L}\la_{\ga\b} - \frac{1}{2}{{R_{\b}}^{\ga}}_{\Lb L}\la_{\a\ga}, F^{\a\b}> \\
&& - [\int_{J^{-}(p)\cap\Sigma_{t}}<\la_{\a\b}\delta(u), \cder_{T}F^{\a\b}>]^{t=1-\eps}_{t=0} \\
&&  + [\int_{J^{-}(p)\cap\Sigma_{t}}<\cder_{T}(\la_{\a\b}\delta(u)), F^{\a\b}>]^{t=1-\eps}_{t=0} 
\end{eqnarray*}

We have by definition,
$$  \int_{N^{-}(p)}<\la_{\a\b}, \cder_{T}F^{\a\b}>= \int_{\SSS^{2}} \int_{0}^{\infty} <\la_{\a\b}, \cder_{T}F^{\a\b}> (u=0, s, \om) ds dA_{s}$$

We have $$\frac{dt}{ds} = L(t) = dt(L) = \g(-T, L)$$

We define,
\bea
\phi =  \g(T,L)^{-1}
\eea
 the null lapse function. We have,
\bea
ds = - \phi dt
\eea
We denote by $dA_{S_{t}(p)}$ the area element of the 2-surface $S_{t}(p) = N^{-}(p)\cap\Sigma_{t}$\

Thus,
\beaa
&&  \int_{0}^{\infty} \int_{\SSS^{2}} <\la_{\a\b}, \cder_{T}F^{\a\b}> (u=0, s, \om)  dA_{s} ds \\
&=&   \int_{t = 1}^{ t=-\infty} \int_{S_{t}} <\la_{\a\b}, \cder_{T}F^{\a\b}>   (-\phi) dA_{t} dt\\
&=& \int_{t = -\infty}^{ t=1}  \int_{S_{t}} <\la_{\a\b}, \cder_{T}F^{\a\b}>   \phi dA_{t} dt
\eeaa

We get, $$- \int_{\J^{-}(p)\cap\Sigma_{1-\eps}}<\la_{\a\b}\delta(u), \cder_{T}F^{\a\b}> = - \int_{N^{-}(p)\cap\Sigma_{1 - \eps}}<\la_{\a\b}, \cder_{T}F^{\a\b}>\phi dA_{S_{t = 1 - \eps}(p)}$$

\begin{definition}
We define positive definite Riemannian metric in the following manner:

\bea
h(e_{\a}, e_{\b}) = \g(e_{\a}, e_{\b}) + 2 \g(e_{\a},  \frac{\pr}{\pr \hat{t}} ) . \g(e_{\b}, \frac{\pr}{\pr \hat{t}} ) \label{positiveriemannianmetrich}
\eea

where
\bea
\frac{\pr}{\pr \hat{t}} = (- \g( \frac{\pr}{\pr t} , \frac{\pr}{\pr t}  ) )^{-\frac{1}{2}} \frac{\pr}{\pr t} 
\eea
\end{definition}
\begin{definition}
For any ${\cal G}$-valued 2-tensor $K$, we let
\bea
|K|^{2} =  h_{\a\mu} h_{\b\nu} |K^{\mu\nu}|. |K^{\a\b}|
\eea
\end{definition}

\begin{lemma}
For any two  ${\cal G}$-valued tensors $K$ and $G$, we have
\bea
| <K_{\a\b}, G^{\a\b}>| \les ( |K|^{2} )^{\frac{1}{2}}. ( |G|^{2} )^{\frac{1}{2}} \label{Cauchy-Schwarzinequalitywithmetrich}
\eea
\end{lemma}

\begin{proof}
\beaa
| <K_{\a\b}, G^{\a\b}>| &=& | \g_{\a\mu} \g_{\b\nu}  <K^{\mu\nu}, G^{\a\b}> | \\
&\le&  | \g_{\a\mu}|.| \g_{\b\nu} |. |<K^{\mu\nu}, G^{\a\b}> | \\
&\le&   h_{\a\mu} h_{\b\nu}  | <K^{\mu\nu}, G^{\a\b}> |  \\
&\le&   h_{\a\mu} h_{\b\nu}  | K^{\mu\nu}|. |G^{\a\b} |  \\
&\le&  ( h_{\a\mu} h_{\b\nu}   | K^{\mu\nu}|. |K^{\a\b} |  )^{\frac{1}{2} } . (  h_{\a\mu} h_{\b\nu}   | G^{\mu\nu}|. |G^{\a\b} |   )^{\frac{1}{2}} 
\eeaa
(by applying Cauchy-Schwarz)\\
\end{proof}

Hence,

\begin{eqnarray*}
&&|\int_{N^{-}(p)\cap\Sigma_{1 - \eps}}<\la_{\a\b}, \cder_{T}F^{\a\b}>\phi dA_{S_{1 - \eps}(p)}| \\
&& \lesssim ||\phi||_{L^{\infty}} (\int_{N^{-}(p)\cap\Sigma_{1 - \eps}}|\la|^{2}dA_{S_{1 - \eps}(p)})^{\frac{1}{2}}(\int_{N^{-}(p)\cap\Sigma_{1 - \eps}}|\cder_{T}F|^{2}dA_{S_{1 - \eps}(p)})^{\frac{1}{2}}
\end{eqnarray*}

(by Cauchy-Schwarz)
\bea
\lesssim ||\phi||_{L^{\infty}} ||\cder_{T}F||_{L^{\infty}} ||\la||_{L^{2}(N^{-}(p)\cap\Sigma_{1 - \eps})}|A_{S_{1 - \eps}(p)}|^{\frac{1}{2}} \label{philambdaarea}
\eea
where $A_{S_{1 - \eps}(p)}$ denote the area of $S_{1 - \eps}(p) = N^{-}(p)\cap\Sigma_{t = 1 - \eps}$ and where,
\bea
||\la||_{L^{2}(\Sigma_{1 - \eps}(p))} = ( \int_{S_{1 - \eps} =  N^{-}(p)\cap\Sigma_{1 - \eps}}   h_{\a\mu} h_{\b\nu} |\la^{\mu\nu}|. |\la^{\a\b}| dA_{S_{1 - \eps}(p)}  )^{\frac{1}{2}}
\eea
and
\bea
||\cder_{T}F ||_{L^{\infty}} = ||(  h_{\a\mu} h_{\b\nu} |\cder_{T}F^{\mu\nu}|. |\cder_{T}F^{\a\b}| )^{\frac{1}{2}}||_{L^{\infty}}
\eea

We want to study now the behavior of $|A_{S_{1 - \eps} (p)}|^{\frac{1}{2}}$ near $\eps = 0$.\

\begin{lemma}
We have,
\bea 
\notag
|A_{S_{  t } }| &=& 4\pi s^{2} +  o(s^{2}) \\
&=& 4\pi (1-t)^{2} +  o( (1-t)^{2} ) \label{areaexpression}
\eea
\end{lemma}

\begin{proof}\

We have,
\bea
\frac{d}{dt}|A_{S_{t}(p)}| = \int_{S_{t}(p)} \phi tr\chi dA \label{derivativearea}
\eea

We also have,
\bea
\frac{d}{ds}(tr\chi) + \frac{1}{2}(tr\chi)^{2} = - |\hat{\chi}|^{2} - Ric(L, L)
\eea
where $\hat{\chi}$ is the traceless part of $\chi$ and $(str\chi)(p) = 2, \hat{\chi}(p) = 0$. This yields to
\bea
\lim_{q \longmapsto p} |tr\chi(q) - \frac{2}{s}| = 0 \label{behaviourtracechi}
\eea
(see [Wang]).

$$\phi =  \g(T,L)^{-1} = ( \der_{T} u )^{-1} .$$ Since $u$ is smooth and $$\phi (p) = 1 = ( \der_{T} u )^{-1} (p)$$ we also have that $\phi $ smooth and bounded near $p$. Thus,
\bea
\lim_{q \longmapsto p}  |\phi (q) - 1| = 0 \label{behaviourphinearp}
\eea
Since, $ds = - \phi dt$, we get,
\beaa
\int_{0}^{s} 1. ds =\int_{1}^{t} - \phi  dt 
\eeaa
Hence, using \eqref{behaviourphinearp}, we get
\bea
s = 1 - t + o(1-t) \label{relationsandt}
\eea

 \eqref{derivativearea}, \eqref{behaviourtracechi}, \eqref{behaviourphinearp} and \eqref{relationsandt} yield to \eqref{areaexpression}.

\end{proof}

Injecting \eqref{areaexpression} in \eqref{philambdaarea}, we get,
\bea
\notag
| - \int_{N^{-}(p)\cap\Sigma_{1 - \eps}} <\la_{\a\b}, \cder_{T}F^{\a\b}> \phi dA_{S_{1 - \eps} (p)}| &\lesssim& \eps ||\phi||_{L^{\infty}} ||\cder_{T} F ||_{L^{\infty}} ||\la ||_{L^2(S_{1 - \eps})}  \\  \label{lambdatf}
\eea
and
\beaa
||\la||_{L^2(S_{1 - \eps})} = ( \int_{S_{1 - \eps}} |\la|^{2} dA )^{\frac{1}{2}} &\lesssim& ||\la||_{L^{\infty}(S_{1 - \eps})} |S_{1-\eps}|^{\frac{1}{2}} \lesssim \eps ||\la ||_{L^{\infty}(S_{1-\eps})}
\eeaa

Now, we would want to study the behavior of $||\la ||_{L^{\infty}(S_{1-\eps})}$ when $\eps \longmapsto 0$.

\begin{lemma}
Let $\Psi_{\a\b} $ be a ${\cal G}$-valued tensor, and $|\Psi_{\a\b}| = <\Psi_{\a\b}, \Psi_{\a\b}> ^{\frac{1}{2}} $. Then,
\bea
| \der_{\si}  |\Psi_{\a\b}|  | &\leq& |\cder  \Psi_{\a\b}  | +  | \Psi (\der_{\si} e_{\a}, e_{\b})  |  + | \Psi (e_{\a}, \der_{\si}  e_{\b})  | \label{derivativeestimateonthenormofacomponent}
\eea
\end{lemma}

\begin{proof}

We can compute
\beaa
\der_{\si}  |\Psi_{\a\b}|  &=& \frac{2 <\der_{\si} ( \Psi_{\a\b}), \Psi_{\a\b}>  }{ 2<\Psi_{\a\b}, \Psi_{\a\b}> ^{\frac{1}{2}}}   \\
&=& \frac{2 <\der_{\si} ( \Psi_{\a\b}), \Psi_{\a\b}> - 2  <\Psi_{\a\b}, [A_{\si}, \Psi_{\a\b}]> + 2 <\Psi_{\a\b},[A_{\si}, \Psi_{\a\b}]> }{ 2<\Psi_{\a\b}, \Psi_{\a\b}> ^{\frac{1}{2}}}  \\
&=& \frac{2 <\der_{\si} ( \Psi_{\a\b}), \Psi_{\a\b}> - 2  <[\Psi_{\a\b}, A_{\si}], \Psi_{\a\b}> + 2 <\Psi_{\a\b}, [A_{\si}, \Psi_{\a\b}]> }{ 2<\Psi_{\a\b}, \Psi_{\a\b}> ^{\frac{1}{2}}}  
\eeaa
(since $< \;,\; >$ is Ad-invariant)
\beaa
\notag
&=& \frac{2 <\der_{\si}  \Psi_{\a\b}, \Psi_{\a\b}>  + 2  <[A_{\si}, \Psi_{\a\b}] , \Psi_{\a\b}> + 2 <\Psi_{\a\b}, [A_{\si}, \Psi_{\a\b}]> }{ 2<\Psi_{\a\b}, \Psi_{\a\b}> ^{\frac{1}{2}}}  \\
\notag
&& +  \frac{  2 < \Psi (\der_{\si} e_{\a}, e_{\b}), \Psi_{\a\b}> +  2< \Psi (e_{\a}, \der_{\si}  e_{\b}) , \Psi_{\a\b}>}{ 2<\Psi_{\a\b}, \Psi_{\a\b}> ^{\frac{1}{2}}} 
\eeaa
Hence,
\beaa
\notag
| \der_{\si}  |\Psi_{\a\b}|  | &\leq& \frac{|\cder_{\si}  \Psi_{\a\b}  | |\Psi_{\a\b}|}{| \Psi_{\a\b}|} + \frac{| \Psi (\der_{\si} e_{\a}, e_{\b})  | |\Psi_{\a\b}|}{| \Psi_{\a\b}|} +  \frac{|\Psi (e_{\a}, \der_{\si}  e_{\b})  | |\Psi_{\a\b}|}{| \Psi_{\a\b}|} \\
&\leq& |\cder  \Psi_{\a\b}  | +  | \Psi (\der_{\si} e_{\a}, e_{\b})  |  + | \Psi (e_{\a}, \der_{\si}  e_{\b})  | 
\eeaa

\end{proof}

\begin{lemma}
We have,
\bea
\der_{\si} h(e_{\a}, e_{\b}) &=& 2     \g(e_{\a}, \der_{\si} \frac{\pr}{\pr \hat{t}} ) . \g(e_{\b}, \frac{\pr}{\pr \hat{t}} )  + 2  \g( e_{\a},  \frac{\pr}{\pr \hat{t}} ) . \g(e_{\b}, \der_{\si} \frac{\pr}{\pr \hat{t}} )   \label{derivativeofthemetrich}
\eea
\end{lemma}

\begin{proof}

\beaa
\der_{\si} h(e_{\a}, e_{\b}) &=& \pa_{\si} h(e_{\a}, e_{\b})  -   h(\der_{\si}e_{\a}, e_{\b}) -  h(e_{\a}, \der_{\si}e_{\b})  \\
&=& \der_{\si} \g(e_{\a}, e_{\b}) + 2 \der_{\si} [\g(e_{\a},  \frac{\pr}{\pr \hat{t}} ) . \g(e_{\b}, \frac{\pr}{\pr \hat{t}} ) ] \\
&=& 2 \pa_{\si} [\g(e_{\a},  \frac{\pr}{\pr \hat{t}} ) . \g(e_{\b}, \frac{\pr}{\pr \hat{t}} ) ] -  2  \g(\der_{\si}e_{\a},  \frac{\pr}{\pr \hat{t}} ) . \g(e_{\b}, \frac{\pr}{\pr \hat{t}} )  \\
&& - 2 \g( e_{\a},  \frac{\pr}{\pr \hat{t}} ) . \g( \der_{\si} e_{\b}, \frac{\pr}{\pr \hat{t}} ) 
\eeaa
(since the metric $\g$ is Killing)
\beaa
&=& 2  \pa_{\si} \g(e_{\a},  \frac{\pr}{\pr \hat{t}} ) . \g(e_{\b}, \frac{\pr}{\pr \hat{t}} )  +  2 \g(e_{\a},  \frac{\pr}{\pr \hat{t}} ) . \pa_{\si}  \g(e_{\b}, \frac{\pr}{\pr \hat{t}} )  \\
&& -  2  \g(\der_{\si}e_{\a},  \frac{\pr}{\pr \hat{t}} ) . \g(e_{\b}, \frac{\pr}{\pr \hat{t}} )   - 2  \g(e_{\a},  \frac{\pr}{\pr \hat{t}} ) . \g( \der_{\si}  e_{\b}, \frac{\pr}{\pr \hat{t}} )  \\
&=& 2  [ \pa_{\si} \g(e_{\a},  \frac{\pr}{\pr \hat{t}} )   -    \g(\der_{\si}e_{\a},  \frac{\pr}{\pr \hat{t}} )] . \g(e_{\b}, \frac{\pr}{\pr \hat{t}} ) \\
&& + 2 [  \pa_{\si} \g(e_{\b},  \frac{\pr}{\pr \hat{t}} )    -  \g(\der_{\si} e_{\b},  \frac{\pr}{\pr \hat{t}} ) ]. \g(e_{\a}, \frac{\pr}{\pr \hat{t}} )  
\eeaa

Using the fact that $\der g = 0$, we get,

\beaa
\der_{\si} h(e_{\a}, e_{\b}) &=& 2     \g(e_{\a}, \der_{\si} \frac{\pr}{\pr \hat{t}} ) . \g(e_{\b}, \frac{\pr}{\pr \hat{t}} )  + 2  \g( e_{\a},  \frac{\pr}{\pr \hat{t}} ) . \g(e_{\b}, \der_{\si} \frac{\pr}{\pr \hat{t}} )  
\eeaa

\end{proof}

Let,
\beaa
\hat{t}_{\a} = ( \frac{\pr}{\pr \hat{t}} )_{\a} = \g_{\mu\a}  ( \frac{\pr}{\pr \hat{t}} )^{\mu}
\eeaa
Hence, we can write \eqref{positiveriemannianmetrich} as,
\bea
h_{\a\b} = \g_{\a\b} + 2 (\frac{\pr}{\pr \hat{t}})_{\a} (\frac{\pr}{\pr \hat{t}})_{\b}
\eea
and \eqref{derivativeofthemetrich} as,
\bea
\der_{\si} h_{\a\b} &=& 2  [ \der_{\si} \hat{t}_{\a} . \hat{t}_{\b} +    \hat{t}_{\a}  .  \der_{\si} \hat{t}_{\b} ]
\eea

\begin{lemma}
Let $\Psi$ be a ${\cal G}$-valued tensor, we have,
\bea
| \der_{\si} |\Psi|^{2} |(p) &\le& C(p) [ |\cder_{\si} \Psi|^{2} +   |\Psi|^{2} ] \label{estimateonderivativeofthenormwithrespecttoh}
\eea
where $C(p)$ depends on the space-time geometry on the point $p$.
\end{lemma}

\begin{proof}

\beaa
 \der_{\si} |\Psi|^{2} &=&  \der_{\si} ( h_{\a\mu} h_{\b\nu}  |\Psi^{\mu\nu}|. |\Psi^{\a\b} | ) =  \der_{\si} ( h_{\a\mu} h_{\b\nu} ) .  |\Psi^{\mu\nu}|. |\Psi^{\a\b} | \\
&& +  h_{\a\mu} h_{\b\nu}  . \der_{\si} ( |\Psi^{\mu\nu}|. |\Psi^{\a\b} | ) \\
\eeaa
Therefore,
\beaa
| \der_{\si} |\Psi|^{2} | &\le& | (\der_{\si} h_{\a\mu})  h_{\b\nu}  |.  |\Psi^{\mu\nu}|. |\Psi^{\a\b} | + | h_{\a\mu} (\der_{\si} h_{\b\nu} )| .  |\Psi^{\mu\nu}|. |\Psi^{\a\b} | \\
&& + | h_{\a\mu} h_{\b\nu} | .  ( | \cder_{\si} \Psi^{\mu\nu}| + |\Psi(\der_{\si} e^{\mu}, e^{\nu})| + | \Psi(e^{\mu}, \der_{\si} e^{\nu}|   ) . |\Psi^{\a\b} |) \\
&& + | h_{\a\mu} h_{\b\nu}|. |  \Psi^{\mu\nu}|. ( |\cder_{\si} \Psi^{\a\b} |  + | \Psi(\der_{\si} e^{\a}, e^{\b}) |+ | \Psi (e^{\a}, \der_{\si} e^{\b}) | ) 
\eeaa
(due to \eqref{derivativeestimateonthenormofacomponent}).\\

Using \eqref{derivativeofthemetrich}, applying Cauchy-Schwarz, using the fact that the metric is smooth, and the inequality $a.b \les a^{2} + b^{2}$, we get,
\beaa
\notag
| \der_{\si} |\Psi|^{2} |(p) &\le& C(p) [ h_{\a\mu} h_{\b\nu}  |\cder_{\si} \Psi^{\mu\nu}|. |\cder_{\si} \Psi^{\a\b} | +  h_{\a\mu} h_{\b\nu}  |\Psi^{\mu\nu}|. |\Psi^{\a\b} | ] \\ 
&\les&  |\cder_{\si} \Psi|^{2} +   |\Psi|^{2} 
\eeaa

\end{proof}

Finally, we get

\begin{lemma} \label{boundingB}
We have,
\bea
\sup _{0 \le \overline{s} \le s}  |\overline{s} \la|^{2} \leq C(p, s) |J|^{2}
\eea
\end{lemma}

\begin{proof}\

We also have at $q \in N^{-}(p) \backslash \{p\}$, $$\cder_{L}(s\la_{\a\b}) = L(s) \la_{\a\b} + s \cder_{L}\la_{\a\b} = \la_{\a\b} + s(- \frac{1}{2}tr\chi\la_{\a\b}) = - \frac{s}{2}\la_{\a\b}(tr\chi - \frac{2}{s})$$

As $|tr\chi - \frac{2}{s}| \longmapsto 0$, we get
\bea
\cder_{L}(s\la_{\a\b}) = O(1) s\la_{\a\b} \label{cderLslambda}
\eea

Hence,
\beaa
| \der_{L} |s\la|^{2} | &\les&   |\cder_{L} ( s\la ) |^{2} +   |s\la|^{2} \\
&\les&    |s\la|^{2}
\eeaa
(due to \eqref{cderLslambda}).\\

For all $(u=0, \overline{s}, \om) \in N^{-}_{\tau}(p)$,  $$\int_{0}^{s} \der_{L} |s\la|^{2} d\overline{s} = |s\la|^{2}(s) - |s\la|^{2}(p) \leq O(s) C(p) \sup _{0 \le \overline{s} \le s} |s\la|^{2}$$

As $$|s\la|^{2}(p) = |J|^{2}$$ choosing $s$ small depending on $p$ we have \eqref{boundingB}.

\end{proof}

Therefore $s\la_{\a\b}$ remains bounded near $p$, and it is also smooth away from $p$,  so $\eps ||\la ||_{L^{\infty}(S_{1 - \eps})}$ remains bounded and therefore $||\la ||_{L^{2}(S_{1 - \eps})}$ remains bounded. Therefore \eqref{lambdatf} gives,
\bea
\lim_{\eps \to 0} |- \int_{J^{-}(p)\cap\Sigma_{1 - \eps}} <\la_{\a\b}\delta(u), \cder_{T}F^{\a\b}>| = 0 \label{firstinitialdataterm}
\eea

\subsection{Estimating $ \lim_{\eps \to 0} \int_{ J^{-}(p)\cap\Sigma_{1 - \eps}}<\cder_{T}(\la_{\a\b}\delta(u)), F^{\a\b}> $}\

Examining now, 
\bea
\notag
&& \int_{J^{-}(p)\cap\Sigma_{1 - \eps}}<\cder_{T}(\la_{\a\b}\delta(u)), F^{\a\b}> \\
\notag
&&= \int_{J^{-}(p)\cap\Sigma_{1 - \eps}} <\cder_{T}(\la_{\a\b})\delta(u) + \der_{T}u\delta^{'}(u)\la_{\a\b}, F^{\a\b}> \\
\notag
&&=  \int_{J^{-}(p)\cap\Sigma_{1 - \eps}} \delta(u)<\cder_{T}(\la_{\a\b}), F^{\a\b}> +  \int_{J^{-}(p)\cap\Sigma_{1 - \eps}} \delta^{'}(u)\phi^{-1}<\la_{\a\b}, F^{\a\b}> \\
&& = I_{\eps}^{1} + I_{\eps}^{2} \label{i1andi2}
\eea

At $q \in N^{-}(p) \backslash \{p\}$ define,
\bea
N = \phi L + T
\eea
 We have,
\begin{eqnarray*}
\g(N, N) &=& \g(\phi L + T, \phi L + T) = \phi \g(T, L) + \phi \g(L,T) + \g(T,T) \\
&=& 2\phi \phi^{-1}  -1 = 2 -1 = 1 
\end{eqnarray*}
 means that $N$ is unit. For all $X \in T_{q}S_{1 - \eps}(p)$ tangent to $S_{1 - \eps}(p)$, i.e. $X \in T_{q}N^{-}(p)\cap T_{q}\Sigma_{1 - \eps}$, we have
\bea
\g(N, X) = \g(\phi L + T, X) = \phi \g(L, X) + \g(T, X) 
\eea
\bea
\g(L, X) = 0 \label{glx}
\eea
(since $X \in T_{q}N^{-}(p)$), and
\bea
\g(T, X) = 0 \label{gtx}
\eea
(since $X \in T_{q}\Sigma_{1 - \eps}$).\

\eqref{glx} and \eqref{gtx} show that $N$ is the unit normal to $S_{1 - \eps} = N^{-}(p) \cap \Sigma_{1 - \eps}$, it can be extended locally to define a vectorfield.\

Thus, we have
\begin{eqnarray*}
\der_{N}\delta(u) &=& \der_{N}(u).\delta^{'}(u) = \der_{\phi L + T}(u).\delta^{'}(u) \\
&=& (\phi \der_{L}(u) + \der_{T}(u) ).\delta^{'}(u) = \der_{T}(u).\delta^{'}(u) \\
&=& \phi^{-1}\delta^{'}(u)
\end{eqnarray*}

Thus
\bea
\notag
I_{\eps}^{2} &=&  \int_{J^{-}(p)\cap\Sigma_{1 - \eps}} \der_{N}\delta(u) <\la_{\a\b}, F^{\a\b}>    \\
\notag
&=& - \int_{J^{-}(p)\cap\Sigma_{1 - \eps}} \delta(u) \der_{N}<\la_{\a\b}, F^{\a\b}> \\
\notag
&& - \int_{J^{-}(p)\cap\Sigma_{1 - \eps}} \text{div}(N).\delta(u) <\la_{\a\b}, F^{\a\b}>  \label{ibpforn}
\eea
(by integration by parts)\

We choose $e_{1}, e_{2}$ tangent to $S_{1 - \eps}$. Since $N, e_{1}, e_{2}$ are unit we have
\beaa
\text{div} N = \g(\der_{N}N, N) + \g(\der_{a}N, e_{a}), \quad a \in \{1, 2\} 
\eeaa

$\g(N, N) = 1$ gives $\g(\der_{N}N, N) = 0$. We get div$N = \g(\der_{a}N, e_{a})$ and $N$ is unit normal to $S_{1 - \eps}$, and $N \in T_{q}\Sigma_{1 - \eps}$ since,
\beaa
\g(N, T) = \g(\phi L + T, T) = \phi \g(L, T) + \g(T, T) = \phi \phi^{-1} -1 = 1-1 = 0 
\eeaa
so we get div$N = tr\th$, where $\th$ is the second fundamental form of the surface $S_{1-\eps}$ embedded in $\Sigma_{1-\eps}$, defined as,
\bea
\th(X, Y) = \g(\der_{X}N, Y)
\eea
for all $X, Y \in T_{q}S_{1 - \eps}$.\
Thus   \eqref{ibpforn} becomes,
\bea
\notag
I_{\eps}^{2} &=& - \int_{J^{-}(p)\cap\Sigma_{1 - \eps}} \delta(u) (\der_{N}<\la_{\a\b}, F^{\a\b}> + tr\th<\la_{\a\b}, F^{\a\b}> ) \\
\notag
&=& - \int_{J^{-}(p)\cap\Sigma_{1 - \eps}} \delta(u) (<\cder_{N}\la_{\a\b}, F^{\a\b}> + tr\th<\la_{\a\b}, F^{\a\b}> ) \\
&& - \int_{J^{-}(p)\cap\Sigma_{1 - \eps}} \delta(u) <\la_{\a\b}, \cder_{N}F^{\a\b}> \label{traceth}
\eea

We showed \eqref{firstinitialdataterm}, in the same manner, we have,
\bea
\lim_{\eps \to 0}  \int_{J^{-}(p)\cap\Sigma_{1 - \eps}} \delta(u) <\la_{\a\b}, \cder_{N}F^{\a\b}> = 0 \label{inthesamemanner}
\eea
Thus, injecting \eqref{inthesamemanner} and \eqref{traceth} in \eqref{i1andi2}, we get,
\bea
\notag
 \lim_{\eps \to 0} I_{\eps}^{1} + I_{\eps}^{2} &=& \lim_{ \eps \to 0}  (\int_{J^{-}(p)\cap\Sigma_{1 - \eps}} \delta(u) [<\cder_{T}\la_{\a\b}, F^{\a\b}> \\
\notag
&&  - <\cder_{N}\la_{\a\b}, F^{\a\b}> + tr\th<\la_{\a\b}, F^{\a\b}> ])  \\
\notag 
&=& \lim_{\eps \to 0} - \int_{J^{-}(p)\cap\Sigma_{1 - \eps}} \delta(u) (<\cder_{N - T}\la_{\a\b}, F^{\a\b}> + tr\th<\la_{\a\b}, F^{\a\b}> ) \\   \label{i1andi2secondform}
\eea

We recall that $N = \phi L + T$, thus $\phi L = N - T$, therefore $$\cder_{N - T}\la_{\a\b} = \phi \cder_{L}\la_{\a\b}$$ and we recall that by construction of $\la$, we have $\cder_{L}\la_{\a\b} + \frac{1}{2}tr \chi \la_{\a\b} = 0$ at $q \in N^{-}(p) \backslash \{p\}$.\

We obtain,
\bea
\cder_{N - T}\la_{\a\b} = - \frac{1}{2} tr\chi \phi \la_{\a\b} \label{derivnminustlamda}
\eea
at $q \in N^{-}(p) \backslash \{p\}. $\

Hence, from \eqref{derivnminustlamda} we can write \eqref{i1andi2secondform} as,
\bea
\lim_{\eps \to 0}  I_{\eps}^{1} + I_{\eps}^{2} = - \lim_{\eps \to 0}  \int_{J^{-}(p)\cap\Sigma_{1 - \eps}} \delta(u) ( - \frac{1}{2} \phi tr\chi + tr\th) <\la_{\a\b}, F^{\a\b}> \label{i1andi2thirdform}
\eea

\begin{eqnarray*}
\th_{aa} &=& \g(\der_{a}N, e_{a}) = \g(\der_{a}(\phi L + T), e_{a}) \\
&=& \g(e_{a}(\phi)L + \phi\der_{a}L + \der_{a}T, e_{a}) \\
&=&  e_{a}(\phi) \g(L, e_{a}) + \phi \g(\der_{a}L, e_{a}) + \g(\der_{a}T, e_{a}).
\end{eqnarray*}

We have $\g(L, e_{a}) = 0$, therefore,
\bea
\th_{aa} = \phi \chi_{aa} + k_{aa} \label{thaa}
\eea
where,
\bea
k_{aa} = \g(\der_{a}T, e_{a}) \label{kaa}
\eea

Injecting \eqref{thaa} and \eqref{kaa} in \eqref{i1andi2thirdform} we get,
\bea
\lim_{\eps \to 0}  I_{\eps}^{1} + I_{\eps}^{2} = - \lim_{\eps \to 0} \int_{\Sigma_{1 - \eps}} \delta(u) (\frac{1}{2} \phi tr\chi + k_{aa}).<\la_{\a\b}, F^{\a\b}>
\eea

where the repeated index $k_{aa}$ means summation $\sum_{a=1,2} k_{aa} = \delta^{ab}k_{ab}$.\

We get,
\bea
&& \lim_{\eps \to 0}   I_{\eps}^{1} + I_{\eps}^{2} \label{laststepi1andi2} \\
\notag
&=& - \frac{1}{2} \lim_{\eps \to 0} \int_{\Sigma_{1 - \eps}} \delta(u)\phi tr\chi <\la_{\a\b}, F^{\a\b}> - \lim_{\eps \to 0} \int_{\Sigma_{1 - \eps}} \delta(u) \delta^{ab}k_{ab} <\la_{\a\b}, F^{\a\b}>
\eea

\begin{lemma} \label{deltaabkab}
We have,
\beaa
\lim_{\eps \to 0} \int_{\Sigma_{1 - \eps}} \delta(u) \delta^{ab}k_{ab} <\la_{\a\b}, F^{\a\b}> = 0
\eeaa
\end{lemma}
\begin{proof}:\\
\begin{eqnarray*}
&&|\int_{\Sigma_{1 - \eps}} \delta^{ab}k_{ab} \delta(u) <\la_{\a\b}, F^{\a\b}>| = \int_{N^{-}(p)\cap\Sigma_{1 - \eps}} \delta^{ab}k_{ab} <\la_{\a\b}, F^{\a\b}> \phi dA_{S_{t = 1 - \eps}}| \\
&&\lesssim ||k||_{L^{\infty}}  ||F||_{L^{\infty}} ||\phi||_{L^{\infty}} ||\la||_{L^{2}(S_{1-\eps}(p))} |A_{S_{1-\eps}(p)}|^{\frac{1}{2}} \\
&&\lesssim \eps ||k||_{L^{\infty}}  ||F||_{L^{\infty}} ||\phi||_{L^{\infty}} ||\la||_{L^{2}(S_{1-\eps}(p))}  
\end{eqnarray*}

And as we showed previously $ ||\la||_{L^{2}(S_{1-\eps}(p))}$ remains bounded as $\eps \longmapsto 0$. Thus,
\bea
\lim_{\eps \to 0} \int_{\Sigma_{1 - \eps}} \delta(u) \delta^{ab}k_{ab} <\la_{\a\b}, F^{\a\b}> = 0
\eea
\end{proof}
We are left to estimate $- \frac{1}{2} \lim_{\eps \to 0}  \int_{\Sigma_{1 - \eps}} \delta(u) \phi tr\chi <\la_{\a\b}, F^{\a\b}>$ in  \eqref{laststepi1andi2}.
\begin{lemma} \label{FtoFpinthelimitintegral}
We have, $$- \frac{1}{2} \lim_{\eps \to 0}  \int_{\Sigma_{1 - \eps}} \delta(u) \phi tr\chi <\la_{\a\b}, F^{\a\b}> = - \frac{1}{2}\lim_{\eps \to 0} \int_{N^{-}(p)\cap\Sigma_{1 - \eps}} \phi^{2} tr\chi <\la_{\a\b}, F^{\a\b}(p)> dA$$
\end{lemma}
\begin{proof}:\

$$- \frac{1}{2} \lim_{\eps \to 0}  \int_{\Sigma_{1 - \eps}} \delta(u) \phi tr\chi <\la_{\a\b}, F^{\a\b}> = - \frac{1}{2} \lim_{\eps \to 0}  \int_{N^{-}(p)\cap\Sigma_{t = 1 - \eps}} \phi^{2} tr\chi <\la_{\a\b}, F^{\a\b}>_{\g}dA$$

We have
\begin{eqnarray*}
&& | \int_{N^{-}(p)\cap\Sigma_{1 - \eps}} \phi^{2} tr\chi <\la_{\a\b}, F^{\a\b} - F^{\a\b}(p)> dA| \\
&=& | \int_{S_{1 - \eps}(p)} \phi^{2} tr\chi <\la_{\a\b}, F^{\a\b} - F^{\a\b}(p)> dA |\\
&\lesssim& ||\phi||_{L^{\infty}}^{2} ||A||_{L^2 (S_{1-\eps}(p))} ||F - F(p)||_{L^{\infty} (S_{1-\eps}(p))} |S_{1-\eps}(p)|^{\frac{1}{2}} ||tr\chi||_{L^{\infty} (S_{1-\eps}(p))} 
\end{eqnarray*}
As,
\beaa
|S_{1-\eps}(p)| \sim 4\pi\eps^{2} \quad \text{as} \quad \eps \longmapsto 0
\eeaa
 and
\beaa
|tr\chi| \sim \frac{2}{ \eps } \quad \text{as} \quad \eps \longmapsto 0
\eeaa
we get,
\beaa
|S_{1-\eps}(p)|^{\frac{1}{2}} ||tr\chi||_{L^{\infty} (S_{1-\eps}(p))} \sim 8\pi \quad \text{as} \quad \eps \longmapsto 0
\eeaa
This yields to
\begin{eqnarray*}
&&| \int_{N^{-}(p)\cap\Sigma_{1 - \eps}} \phi^{2} tr\chi <\la_{\a\b}, F^{\a\b} - F^{\a\b}(p)> dA|\\
&&\lesssim ||\phi||_{L^{\infty}}^{2} ||A||_{L^2 (S_{1-\eps}(p))} ||F - F(p)||_{L^{\infty} (S_{1-\eps}(p))}
\end{eqnarray*}

Since, $$\lim_{\eps \to 0} ||F - F(p)||_{L^{\infty} (S_{1-\eps}(p))} = 0$$ and as we showed $ ||\la_{\a\b}||_{L^{2}(S_{1-\eps}(p))}$ remains bounded as $\eps \longmapsto 0$, we get $$\lim_{\eps \to 0} |\int_{N^{-}(p)\cap\Sigma_{1 - \eps}} \phi^{2} tr\chi <\la_{\a\b}, F^{\a\b} - F^{\a\b}(p)>_{\g}dA| = 0$$

Therefore,
\bea
\notag
&& \lim_{\eps \to 0} \int_{N^{-}(p)\cap\Sigma_{1 - \eps}} \phi^{2} trX <\la_{\a\b}, F^{\a\b}> dA \\
&=& \lim_{\eps \to 0} \int_{N^{-}(p)\cap\Sigma_{1 - \eps}} \phi^{2} trX <\la_{\a\b}, F^{\a\b}(p)> dA
\eea
\end{proof}
\begin{lemma} \label{finalleammatogetFp}
We have, $$\lim_{\eps \to 0} \int_{N^{-}(p)\cap\Sigma_{1 - \eps}} \phi^{2} trX <\la_{\a\b}, F^{\a\b}(p)> dA = 8\pi <(\J_{p})_{\a\b}, F^{\a\b} (p)>$$
\end{lemma}
\begin{proof}\

We have,
\bea
\notag
&& \lim_{\eps \to 0} \int_{N^{-}(p)\cap\Sigma_{1 - \eps}} \phi^{2} tr\chi <\la_{\a\b}, F^{\a\b}(p)> \\
&=& \lim_{\eps \to 0} \int_{N^{-}(p)\cap\Sigma_{1 - \eps}} \phi^{2} (s^{-1}tr\chi) <(s\la_{\a\b}), F^{\a\b}(p)> \label{ssminus1}
\eea

As,
\bea
|tr\chi - \frac{2}{s}| = O(s^{2})
\eea
where $O$ depends on the geometry of the space-time (see for example proposition 3.2 in the thesis of Q. Wang [Wang]), we get,
\bea
\lim_{s \to 0} \sup _{S_{1-\eps}(p)} |s^{-1}tr\chi - \frac{2}{s^{2}}|  = 0
\eea
and we know that,
\bea
\lim_{s \to 0} \sup _{S_{1-\eps}(p)} |\phi - 1| = 0
\eea
and,
\bea
\lim_{s \to 0} (s\la_{\a\b})  = \J_{p} 
\eea
This yields to
\bea
|s^{-1}tr\chi| \sim_{\eps \to 0} \frac{2}{\eps^{2}}  \\
|\phi| \sim_{\eps \to 0} 1  \\
(s\la_{\a\b}) \sim_{\eps \to 0} \J_{p} 
\eea
and therefore,
\bea
\notag
&& \int_{N^{-}(p)\cap\Sigma_{1 - \eps}} \phi^{2} (s^{-1}tr\chi) <(s\la_{\a\b}), F^{\a\b}(p)> dA \\
&& \sim_{\eps \to 0} \int_{N^{-}(p)\cap\Sigma_{1 - \eps}} \frac{2}{\eps^{2}} <(\J_{p})_{\a\b}, F^{\a\b}(p)> dA
\eea

and since,
\bea
|N^{-}(p)\cap\Sigma_{1 - \eps}| = |S_{1-\eps}(p)| \sim_{\eps \to 0} 4\pi\eps^{2}
\eea
we get,
\bea
\int_{N^{-}(p)\cap\Sigma_{1 - \eps}} \frac{2}{\eps^{2}} <(\J_{p})_{\a\b}, F^{\a\b}(p)> dA =  \frac{2}{\eps^{2}}(4\pi) \eps^{2} <(\J_{p})_{\a\b}, F^{\a\b}(p)> + O(\eps)
\eea
where $O(\eps) \longmapsto 0$ as $\eps \longmapsto 0$\

Given \eqref{ssminus1}, this yields to 

\bea
\lim_{\eps \to 0} \int_{N^{-}(p)\cap\Sigma_{1 - \eps}} \phi^{2} tr\chi <\la_{\a\b}, F^{\a\b}(p)> = 8\pi <(\J_{p})_{\a\b}, F^{\a\b}(p)>
\eea
\end{proof}
From \eqref{FtoFpinthelimitintegral} we get,
$$- \frac{1}{2} \lim_{\eps \to 0}  \int_{\Sigma_{1 - \eps}} \delta(u) \phi tr\chi <\la_{\a\b}, F^{\a\b}> = - 4\pi <(\J_{p})_{\a\b}, F(p)^{\a\b}>$$

\subsection{The parametrix}\

Finally, combining \eqref{limitonomegatocoverp}, \eqref{afterdiv}, \eqref{boxlambdadeltaF}, \eqref{firstinitialdataterm}, \eqref{i1andi2}, \eqref{laststepi1andi2}, \eqref{deltaabkab}, \eqref{FtoFpinthelimitintegral} and \eqref{finalleammatogetFp}, we get,

\begin{eqnarray*}
&&\int_{\Omega}<\la_{\a\b}\delta(u),\Box_{\g}^{(A)}F^{\a\b}> \\
&=& \lim_{\eps \to 0} [ \int_{\Omega_{\eps}}\delta(u)< \hat{\lap}^{(A)}\la_{\a\b}  +2 \ze_{a} \cder_{a} \la_{\a\b} + \frac{1}{2}\mu\la_{\a\b} \\
&&+ [F_{L \Lb}, \la_{\a\b}] - \frac{1}{2}{{R_{\a}}^{\ga}}_{\Lb L}\la_{\ga\b} - \frac{1}{2}{{R_{\b}}^{\ga}}_{\Lb L}\la_{\a\ga}, F^{\a\b}> ] + 0 - 4\pi <(\J_{p})_{\a\b}, F^{\a\b}(p)>\\
&& + \int_{J^{-}(p)\cap\Sigma_{t}}<\la_{\a\b}\delta(u), \cder_{T}F^{\a\b}>|_{t=0}  - \int_{J^{-}(p)\cap\Sigma_{t}}<\cder_{T}(\la_{\a\b}\delta(u)), F^{\a\b}>|_{t=0} 
\end{eqnarray*}

Therefore,
\bea
\notag
 4\pi <(\J_{p})_{\a\b}, F^{\a\b}(p)> &=& - \int_{\Omega}<\la\delta(u),\Box_{\g}^{(A)}F^{\a\b}> \\
\notag
&&  + \int_{\Omega}\delta(u)< \hat{\lap}^{(A)}\la_{\a\b}  + 2 \ze_{a} \cder_{a} \la_{\a\b} + \frac{1}{2}\mu\la_{\a\b} \\
\notag
&& + [F_{L \Lb}, \la_{\a\b}] - \frac{1}{2}{{R_{\a}}^{\ga}}_{\Lb L}\la_{\ga\b} - \frac{1}{2}{{R_{\b}}^{\ga}}_{\Lb L}\la_{\a\ga}, F^{\a\b}> \\
\notag
&& + \int_{J^{-}(p)\cap\Sigma}<\la_{\a\b}\delta(u), \cder_{T}F^{\a\b}> \\
&& - \int_{J^{-}(p)\cap\Sigma}<\cder_{T}(\la_{\a\b}\delta(u)), F^{\a\b}> \label{KSparametrixYMsetting}
\eea

where $\hat{\lap}^{(A)} \la_{\a\b}$ is the induced Laplacian on the span of $\{e_{a}\}$, $a \in \{ 1, 2 \}$, of $\la_{\a\b}$, defined by \eqref{laplacianonab}, and where the last two terms are the contribution of the initial data, the first term is the contribution of the nonlinear term in the tensorial wave equation, and the middle term is related to the geometry of the problem.\\



\begin{thebibliography}{Mo-Se-So}

\expandafter\ifx\csname url\endcsname\relax
\def\url#1{{\tt #1}}\fi
\expandafter\ifx\csname urlprefix\endcsname\relax\def\urlprefix{URL }\fi

\bibitem[Br]{Br} Y. Choqu\'et-Bruhat, \textit{Th\'eor\`eme d'existence pour certains
syst\`emes d'\'equations aux d\'eriv\'ees partielles nonlin\'eaires.}, Acta Math.
\textbf{88} (1952), 141-225.

\bibitem[C]{C} D. Christodoulou, \textit{The formation of black holes in general relativity}, Monographs in Mathematics, European
Mathematical Soc. 2009.

\bibitem[C-K]{C-K} D. Christodoulou, S. Klainerman, \textit{The global nonlinear stability of
the Minkowski space}, Princeton Math. Series {\bf 41}, 1993.

\bibitem[CS]{CS} P. Chru\'sciel, J. Shatah, \textit{Global existence of solutions of the Yang-Mills equations on globally hyperbolic four-dimensional Lorentzian manifolds},
Asian J. Math. {\bf 1} (1997), no. 3, 530--548.

\bibitem[EM1]{EM1} D. Eardley, V. Moncrief, \textit{The global
existence of Yang-Mills-Higgs fields in $4$-dimensional Minkowski space.
I. Local existence and smoothness properties.}  Comm. Math. Phys. {\bf
83} (1982), no. 2, 171--191.

\bibitem[EM2]{EM2} D.  Eardley, V. Moncrief, \textit{The global existence of
Yang-Mills-Higgs fields in $4$-dimensional Minkowski space. II.
Completion of proof.}  Comm. Math. Phys. {\bf  83} (1982), no. 2,
193--212.

\bibitem[Fried]{Fried} H. G. Friedlander, {\it The Wave Equation on a Curved Space-time},
Cambridge University Press, 1976.

\bibitem[HE]{HE}
S. W. Hawking \& G. F. R. Ellis,
{\it The Large Scale Structure of Space-time}, Cambridge: Cambridge University Press,
1973

\bibitem[KR1]{KR1} S. Klainerman, I. Rodnianski,
\textit{A Kirchoff-Sobolev parametrix for  the
 wave equation and applications}, Journ. Hyperb Diff. Eqts. 4, Nr 3 (2007), 401-433

 \bibitem[KR2]{KR2} S. Klainerman, I. Rodnianski,
 \textit{Causal Geometry of Einstein-Vacuum Spacetimes with Finite Curvature Flux},
Invent. math. 159 (2005), 437-529.

 \bibitem[KR3]{KR3} S. Klainerman, I. Rodnianski,
 \textit{On the breakdown criterion in General Relativity},
 J. Amer. Math. Soc. 23 (2010), 345-382

\bibitem[M]{M} V. Moncrief, \textit{An integral equation for space-time curvature in General Relativity}, Surveys in Differential Geometry, Vol X, International Press (2006), pages 109-146.

\bibitem[Pach]{Pach} B. G. Pachpatte, \textit{A note on integral inequalities of the Bellman-Bihari type}, J. Math. Anal. Appl., {\bf 49} (1975), 295-301.

\bibitem[Wang]{Wang} Q. Wang, \textit{Causal geometry of Einstein vacuum space-times}. PhD thesis, Princeton
University, 2006.

\bibitem[W]{W} R. M. Wald, \textit{General Relativity}, University of Chicago Press (Chicago,
1984). 

\end{thebibliography}
\end{document}